\UseRawInputEncoding

\documentclass[a4, 12pt]{amsart}
\usepackage{amssymb,amsthm,amstext,amsmath,amscd,latexsym,amsfonts,amscd}
\usepackage{color, mathrsfs, ulem}
\usepackage[all]{xy}

\def\m{\mathfrak m}
\def\p{\mathfrak p}

\def\Z{\mathbb Z}

\def\C{{\mathscr C}(R)}
\def\K{{\mathscr K}(R)}
\def\Ko{{\mathscr K}}
\def\uK{{\underline{{\mathscr K}(R)}}}
\def\F{{\mathrm{Add}(R)}}
\def\U{{\mathcal U }}

\def\Hom{\mathrm{Hom}}
\def\RHom{\mathrm{RHom}}

\def\Ker{\mathrm{Ker}}
\def\ker{\mathrm{ker}}
\def\Cok{\mathrm{Coker}}

\def\Ext{\mathrm{Ext}}

\def\Mod{\mathrm{Mod}}
\def\mod{\mathrm{mod}}
\def\Ass{\mathrm{Ass}}
\def\depth{\mathrm{depth}}

\def\transpose{\mathrm{Tr}}

\def\proj{\mathrm{proj}}
\def\Supp{\mathrm{Supp}}

\theoremstyle{plain} 
\newtheorem{theorem}{\bfseries Theorem}[section]
\newtheorem{theorem/definition}[theorem]{\textbf Theorem and Definition}
\newtheorem{lemma}[theorem]{\textbf Lemma}
\newtheorem{corollary}[theorem]{\textbf Corollary}
\newtheorem{proposition}[theorem]{\textbf Proposition}

\theoremstyle{definition}
\newtheorem{definition}[theorem]{\textbf Definition}
\newtheorem{definition/corollary}[theorem]{\textbf Definition and Corollary}
\newtheorem{remark}[theorem]{\textbf Remark}
\newtheorem{example}[theorem]{\textbf Example}
\newtheorem{notation}[theorem]{\textbf Notation}

\numberwithin{equation}{section}
\theoremstyle{plain}

\begin{document}
\title[unbounded complexes of projective modules]{Homotopy categories of unbounded complexes of projective modules}
\author{Yuji Yoshino}
\address{Department of Mathematics, Okayama University, Okayama, 700-8530, Japan}
\email{yoshino@math.okayama-u.ac.jp}

\pagestyle{headings}

\maketitle
\renewcommand{\thefootnote}{\fnsymbol{footnote}}
\renewcommand{\thefootnote}{\arabic{footnote}}

\thispagestyle{empty} 
\begin{abstract}
We develop in this paper the stable theory for projective complexes, 
by which we mean to consider a chain complex of finitely generated projective modules as an object of the factor category of the homotopy category modulo split complexes. 
As a result of the stable theory we are able to prove that any complex of finitely generated projective modules over a generically Gorenstein ring is acyclic if and only if its dual complex is acyclic. 
This shows the dependence of total reflexivity conditions for modules over a generically Gorenstein ring. 
\end{abstract}
\vspace{12pt}
{\bf Contents}
\begin{enumerate}
\item[$1$] Introduction  
\item[$2$] Preliminary observation for complexes 
\item[$3$] *Torsion-free and *reflexive complexes
\item[$4$] Complexes over a generically Gorenstein ring
\item[$5$] Split complexes and $\F$
\item[$6$] The stable category of $\K$
\item[$7$] $\F$-approximations
\item[$8$] Contractions
\item[$9$] Remarks on partial $\F$-resolutions
\item[$10$] Counit morphism for the adjoint pair $(\Sigma^{n}, \Omega^n)$
\item[$11$] The main theorem and the proof 
\item[$12$] Applications
\end{enumerate}
\vspace{12pt}

\section{Introduction}

In this paper we are mainly interested in unbounded cochain complexes consisting of finitely generated projective modules over a commutative Noetherian ring. 
Of most interest to us in the present paper are the properties of complexes that are independent of any additional split summands. 
For this purpose we develop the stable theory for those complexes.  
For the module category such an idea was first proposed and established by Auslander-Bridger \cite{AB} under the name of \lq stable module theory\rq. 
We apply their idea to the homotopy category of complexes of finitely generated projective modules. 

The whole of our stable theory for complexes is devoted to prove the following single theorem.

\begin{theorem} [See Theorem \ref{last main theorem}]\label{1}
Let  $R$  be a commutative Noetherian ring that is a generically Gorenstein ring, and 
$X$ a complex of finitely generated projective  $R$-modules. 
Then, $X$  is acyclic  (i.e.  $H(X)=0$)  if and only if  the $R$-dual $X^*$  is acyclic (i.e. $H(X^*)=0$). 
\end{theorem}

Recall that a commutative Noetherian ring is called a generically Gorenstein ring if the total ring of quotients is a Gorenstein ring, or equivalently $R_{\p}$  is a Gorenstein ring for every associated prime $\p \in \Ass (R)$. 
As a matter of fact,  every  Noetherian integral domain, a little more generally, every reduced Noetherian ring, is a generically Gorenstein ring. 
A similar theorem to Theorem \ref{1}, but in a more special setting, was considered in \cite[Corollary 1.4]{Y2}.  

Analogously to the stable module theory of Auslander and Bridger, we will introduce the parallel notion of torsion-freeness and reflexivity for complexes,   which we call *torsion-free complexes and *reflexive complexes in this paper (Definition \ref{deftorfree}). 
We observe in Theorem \ref{exact} that there is an exact sequence similar to the Auslander-Bridger sequence. 
If the base ring $R$ is a generically Gorenstein ring, then as we shall show in Theorem \ref{theorem torsion-free}, a complex $X$ is *torsion-free if and only if  the cohomology modules  $H^i(X^*) \ (i \in \Z)$ are torsion-free as $R$-modules.    

A crucial point for the proof of  Theorem \ref{1} is how one can relate a generic condition of the ring  such as the generic Gorenstein condition with the *torsion-free or the *reflexive property for complexes. 
This will be accomplished  by considering a factor category of the homotopy category. 
To be more precise, let  $\K$  be the homotopy category of all complexes of finitely generated projective modules over a commutative Noetherian ring  $R$, and let  $\F$  be its additive full subcategory consisting of all split complexes. 
See Definition \ref{smallest additive sub} and Theorem \ref{characterization of F} for further details. 
We show in Lemmas \ref{right approximation} and \ref{left approximation} that $\F$  is functorially finite in $\K$ and hence every complex in $\K$  can be resolved by complexes in $\F$.

We define  $\uK$  to be the factor category $\K / \F$ and call it the stable category. 
Then we are able to define the syzygy functor $\Omega$ and the cosyzygy functor $\Sigma$  on $\uK$, and as a result we have an adjoint pair $(\Sigma, \Omega)$ of functors (Theorem \ref{adjoint}). 
Thus there is a natural counit morphism  
$\pi _X ^n : \Sigma ^{n}\Omega ^n X \to X$ for any positive integer  $n$  and for any complex  $X$. 
In terms of syzygy functors,  we can characterize the *torsion-free property for X  as the counit morphism $\pi_X^1$  is an isomorphism in  $\uK$ (Theorem \ref{torsion-free2}).  
We develop in \S 8 some new idea to construct complexes by successive use of mapping cone constructions, which we shall call the contraction. 
See  Theorem and Definition \ref{theorem/definition}. 

Now taking the mapping cones of the counit morphisms in $\K$, we have triangles of the form 
$$
\xymatrix{\Delta ^{(n, 0)} (X) \ar[r] & \Sigma ^{n}\Omega ^n (X) \ar[r]^(0.65){\pi_X^n}  & X \ar[r] & \Delta ^{(n, 0)}  (X)[1], } 
$$
by which we define the complexes  $\Delta^{(n, 0)} (X)$  for  $X$. 
Then we shall show that  all such complexes  $\Delta^{(n, 0)} (X)$  have a finite $\F$-resolution of length at most  $n-1$. 
See Theorem \ref{Delta has finite res} for more precise statement, which is one of the key results in order to prove Theorem \ref{1}. 
After observing these facts, we prove in {Lemma} \ref{main lemma to prove theorem} that any syzygy complex  $\Omega ^r X$  is *torsion-free if  $H(X^*)=0$. 
This is the second key result to prove Theorem \ref{1}. 
As a consequence of this theorem, we are eventually able to prove the main theorem \ref{1} in \S 11. 
 
 \vspace{12pt}

The following are the corollaries that are proved straightforwardly from Theorem \ref{1}, and each one is proved in the last section of this paper.

\begin{corollary}\label{2}
Assume that the ring  $R$ is a generically Gorenstein ring. 
Let  $f : X \to Y$  be a chain homomorphism between complexes of finitely generated projective modules over $R$. 
Then, $f$  is a quasi-isomorphism if and only if  the $R$-dual  $f^* : Y^* \to X^*$  is  a quasi-isomorphism. 
\end{corollary}

\begin{corollary}\label{3}
Assume that the ring  $R$ is a generically Gorenstein ring. 
Let  $M$ be a finitely generated $R$-module. 
Then the following conditions are equivalent: 
\begin{enumerate}
\item
$M$ is a totally reflexive $R$-module. \medskip
\item
$\Ext_R^i (M, R) =0$  for all $i >0$. \medskip
\item
$M$ is an infinite syzygy, i.e. there is an exact sequence of infinite length of the form 
$
\xymatrix@C=24pt{0 \ar[r] & M \ar[r]& P_0 \ar[r]& P_1 \ar[r]& P_2 \ar[r]& \cdots,} 
$
where  each $P_i$  is a finitely generated projective $R$-module. 
\end{enumerate}
\end{corollary}

\begin{corollary}
Under the assumption that $R$ is a generically Gorenstein ring, we have the equality of G-dimension; 
$$
\mathrm{G\text{-}dim} _R M = \sup \{ n \in \Z \ | \ \Ext _R^n (M, R) \not= 0 \ \}, 
$$ 
for a finitely generated $R$-module $M$. 
\end{corollary}

Jorgensen and \c{S}ega \cite{JS} gave an example of a module over a non-Gorenstein Artinian ring for which the implication $(2) \Rightarrow (1)$ in Corollary \ref{3} fails, hence the generic Gorensteinness assumption in the theorem is indispensable.

The following is a commutative version of Tachikawa conjecture that is also a consequence of Theorem \ref{1}. 
It should be noted that this has been proved by Avramov, Buchweitz and \c{S}ega \cite{ABS}. 
 
\begin{corollary}\label{4}
Let  $R$  be a Cohen-Macaulay ring with canonical module $\omega$. 
Furthermore assume that  $R$ is a generically Gorenstein ring. 
If  $\Ext _R ^i (\omega ,  R) =0$  for all $i >0$, then $R$ is Gorenstein. 
\end{corollary}

\begin{corollary}\label{5}
Assume that the ring  $R$ is a generically Gorenstein ring. 
Let  $X$  be a complex of finitely generated projective modules. 

{
\begin{enumerate}
\item 
If $H(X)$ is bounded above, i.e. $X  \in D ^- (R)$, then there is an isomorphism  $X^* \cong \RHom _R(X, R)$ in the derived category $D(R)$
\item 
If  $H(X)$ and $H(X^*)$ are bounded above, i.e. $X, X^* \in D ^- (R)$, 
then we have the isomorphism in the derived category: 
$$
X \cong \RHom _R (\RHom _R (X, R), R). 
$$
\end{enumerate}}
\end{corollary}

\begin{corollary}\label{6}
Assume that the ring  $R$ is a generically Gorenstein ring. 
Let  $X$  be a complex of finitely generated projective modules. 
If all the cohomology modules $H^i(X) \ (i \in \Z)$ have dimension at most $\ell$  as $R$-modules,  
then so are  the modules  $H^i (X^*) \ (i \in \Z)$.  
In particular,  $X$  has cohomology modules of finite length if and only if  so does $X^*$. 
\end{corollary}

\vspace{6pt}
\noindent
{\bf Acknowledgement}: 
The author is grateful to Dr.~Shinya Kumashiro for pointing out several errors in the initial version of the paper and giving him a lot of valuable comments.  
The author also thanks Osamu Iyama, Ryo Takahashi, Luchezar Avramov, Srikanth Iyengar, Henning Krause and Bernhard Keller for their useful comments. 
{Finally he thanks the anonymous referee for his careful reading of the manuscript and his many insightful suggestions.} 

This research was supported by JSPS Grant-in-Aid for Scientific Research 19K0344801.

\vspace{6pt}
\section{Preliminary observation for  complexes} 

Throughout this paper we assume that  $R$ is a commutative Noetherian ring. 
We denote by $\mod (R)$ the abelian category of finitely generated $R$-modules and $R$-module homomorphisms. 
Furthermore we denote  by $\proj (R)$  the additive subcategory of  $\mod (R)$,  which consists of all finitely generated projective $R$-modules.

We denote by $\C = C (\proj (R))$  the additive category of complexes over  $\proj (R)$ and chain homomorphisms.  
We also denote by $\K = K (\proj (R))$ the homotopy category consisting of complexes over $\proj (R)$. 
Note by recalling the definition that  objects of  $\C$  and $\K$  are complexes consisting of finitely generated projective modules, which we denote cohomologically such as 
$$
X= \left[ 
\xymatrix{ 
\cdots \ar[r] & X ^{i-1} \ar[r]^(0.6){d_X^{i-1}}  & X^i \ar[r]^(0.4){d_X^i} &  X^{i+1} \ar[r] & \cdots \\
}\right],  
$$
where  each $X^i$ belongs to  $\proj (R)$. 
All cohomology modules $H^i(X) \ (i \in \Z)$ are necessarily finitely generated $R$-modules for  $X \in \C$.  
A morphism  $X \to Y$ in $\C$ is a cochain homomorphism, 
while a morphism $X \to Y$ in $\K$  is a homotopy equivalence class of a chain homomorphism from $X$ to $Y$, i.e.
$$
\Hom  _{\K} (X, Y) = \Hom _{\C} (X, Y)  / \{\text{chain homotopy}\}. 
$$
Both of $\Hom  _{\C} (X, Y)$ and $\Hom _{\K} (X, Y)$ have natural  $R$-module structures. 
However they are not necessarily finitely generated $R$-modules in general.  
\par\noindent
Eg. Consider the endomorphisms of the complex 
$\left[ \xymatrix{ \cdots \ar[r]& R \ar[r]^0 &R \ar[r]^0 &R \ar[r]&\cdots} \right]$. 

Note that a complex $X \in \C$ is the zero object as an object of  $\K$  if and only if it is a split exact sequence as a long exact sequence, which is called a null complex. 
(It is also known as a contractible complex.) 
Every complex $X \in \C$  has a direct sum decomposition in  $\C$ such as  $X = X' \oplus N$,  where  $N$  is a null complex and  $X'$  contains no null complex as a direct summand.\footnote{In general, let $ f: P \to Q$  be an $R$-homomorphism between finitely generated projective $R$-modules. 
Then there are direct decompositions $P = P' \oplus P''$, $Q =Q' \oplus Q''$ under which $f$ is described as  $\begin{pmatrix} f'& 0 \\ 0 & f'' \end{pmatrix}$ where  $f'': P'' \to Q''$  is an isomorphism, and $f' : P' \to Q'$ does not contain any direct summands that induces isomorphism. 
This is possible since $P$ and $Q$ are finitely generated. 
If we are given a complex $X= [\xymatrix{ \cdots \ar[r] & X^n \ar[r]^{d^n_X} &X^{n+1} \ar[r] & \cdots}]$  in $\C$, applying this argument to each ${d^n_X} $ we have desired decomposition. }
We should note that such a decomposition is not unique in general.

It is clear and well-known that a chain homomorphism $f$ in $\C$  factors through a null complex if and only if  $f$ is null homotopic. 
Therefore the category  $\K$  is a residue category of $\C$ by the ideal generated by the object set consisting of all null complexes. 
It is easy to verify that  $\C$ is a Frobenius category with null complexes as relatively projective and injective objects. 
In such a sense $\K$ has a structure of triangulated category. 
Recall that the shift functor $X \mapsto X[1]$  is defined as 
$X[1]^n= X^{n+1}$ {and} $d_{X[1]}^n = - d_X^{n+1}$. 
Furthermore there is a triangle  $X \to Y \to Z \to X[1]$  in $\K$  if and only if there is an exact sequence in $\C$ of the form; 
$$
\xymatrix@C=32pt{ 
0 \ar[r] & X \ar[r]  & Y \oplus N \ar[r] & Z \ar[r] & 0, }
$$
where $N$ is a null complex. 
One can find such description of triangles in Happel \cite[Chapter 1]{H}. 
The more general references for complexes and triangulated categories are Weibel \cite{W} and {Gelfand-Manin }\cite{MG}.

A remarkable advantage of  $\K$  is that it possesses a duality. 
For $X \in \K$ we are able to define the dual complex by
$$ 
X^* = \Hom _R (X, R), \quad d_{X^*}^n = {(-1)^{n+1}} \Hom _R (d_X^{-n}, R). 
$$
Note that  $X^*$  is again an object of  $\K$, since the dual of a finitely generated projective module is finitely generated projective. 
It is easy to see that the {duality}  functor 
$$
( - ) ^* \ : \ \K \longrightarrow \K^{\text{op}}, \quad X \mapsto X^*
$$
is a triangle functor between triangulated categories. 
Since  $X^{**}$  is naturally isomorphic to $X$, it actually yields the duality on $\K$.

\begin{notation}\label{c}
For a complex  $X \in \K$, 
we denote by  $C(X)$  
{the cokernel of the differential mapping  $d_X [-1] : X[-1] \to X$ as underlying graded $R$-module.
So  $C (X) = \oplus _{i \in \Z} C^i(X)$ where $C^i(X) = \Cok (X^{i-1} \to X^i)$.}
Similarly the cocycle  $Z(X)= \bigoplus _{i \in \Z}Z^i(X)$  is the kernel of  $d_X$ and the coboundary $B(X)= \bigoplus _{i \in \Z}B^i(X)$  is the image of   $d_X$. 

As  $C (X) = X /B(X)$ there is a short exact sequence of graded $R$-modules such as
\begin{equation}\label{HCB}
\xymatrix{0 \ar[r] & H(X) \ar[r] & C (X) \ar[r] & B (X)[1]  \ar[r] & 0. }
\end{equation}
\end{notation}

\vspace{12pt}

Let  $X$ be a complex in $\K$ and let  $M$ be an $R$-module.  
We denote by  $K(\Mod (R))$  the homotopy category of all complexes of any $R$-modules, and we regard  $M$ as a complex concentrated in degree zero. 
Recall that $\Hom _R(X, M)$ is the Hom complex and an element of the cohomology modules of this complex is nothing but the homotopy class of a chain map from $X$  to  $M$, i.e.
$$
H^{-i} (\Hom _R (X, M)) = \Hom _{K(\Mod ( R ))} (X[i], M). 
$$

\begin{definition}\label{rho}
Let  $X \in \K$, $M$  an $R$-module and  $i \in \Z$. 
As noted above, each element  $[f] \in H^{-i}(\Hom _R (X, M))$  is a homotopy class of a chain map  $f : X[i] \to M$, thus it induces a unique $R$-module homomorphism 
$H^0(f) : H^0(X[i]) = H^i(X) \to H^0(M) = M$,  hence an element $H^0(f) \in \Hom _R( H^i(X), M)$. 
We define an $R$-module homomorphism 
$$
\rho ^i_{X M} : H^{-i} (\Hom _R (X, M)) \longrightarrow \Hom _R (H^i(X), M) \ \ 
$$
by  $\rho ^i _{X M}([f]) = H^0 (f)$.\footnote{{This is the standard K\"unneth map; see \cite[IV.6]{CE}.}}

\end{definition}

\begin{theorem}\label{exact}
Under the circumstances in Definition \ref{rho}, there is an exact sequence of $R$-modules; 
{\small 
$$
\xymatrix{ 
0 \to \Ext _R ^1 (C^{i+1}(X), M) \to H^{-i} (\Hom _R (X, M))
\ar[r]^(0.52){\rho ^i_{X M}} & \Hom _R (H^i(X), M) \to \Ext_R ^2 (C^{i+1}(X), M), \\
}
$$
}
for each  $i \in \Z$. 
\end{theorem} 

\begin{proof}
We see from the exact sequence (\ref{HCB}) that there exists an exact sequence; 
{\small 
\begin{equation}\label{HCBdual}
\xymatrix@C=10pt{
0 \to \Hom _R (B^{i+1} (X), M) \ar[r] & \Hom _R(C^i (X), M) \ar[r] & \Hom _R (H^{i} (X), M) \ar[r] & \Ext _R ^1(B^{i+1}(X), M),  \\
}
\end{equation}
}
where we should note that  $\Ext _R ^1(B^{i+1}(X), M) \cong \Ext _R ^2 (C^{i+1}(X), M)$, since $X^{i+1} / B^{i+1}(X) \cong C^{i+1}(X)$ and $X^{i+1}$  is projective. 

Note that  $\Hom _{R} (X, M)^{-i} = \Hom _{R} (X^{i}, M)$  for all $i \in \Z$. 
Thus, from  the exact sequence  
$X^{i-1} \to  X^{i} \to C^{i}( X )  \to 0$,   it follows that 
$0 \to  \Hom _{R} (C^{i}( X ), M) \to \Hom _{R}(X, M)^{-i}  \to \Hom _{R}(X, M)^{-i+1}$  is exact, hence we have an isomorphism    
$$
\lambda : Z ^{-i} (\Hom _{R}(X, M)) \to \Hom _{R}(C^{i}( X ), M), 
$$
 where the left hand side is the $(-i)$th cocycle module of the complex  $\Hom _{R}(X, M)$. 
On the other hand, the $(-i)$th coboundary  $B^{-i}(\Hom _{R}( X , M))$  is the image the mapping 
$\Hom _{R}(X^{i+1}, M)  \to \Hom _{R}(X^{i}, M)$. 
From the exact sequence  $0 \to B^{i+1}( X ) \to X ^{i+1} \to C^{i+1}( X ) \to 0$, we have an exact sequence 
$$
\xymatrix{
\Hom _{R} (X^{i+1},  M) \ar[r]^(0.45){\nu} & \Hom _{R}(B^{i+1}( X ), M)  \ar[r] & \Ext _{R}^{1} (C^{i+1}( X ) , M) \ar[r] & 0.}  
$$ 
Since there is a commutative diagram  
$$
\xymatrix@R=24pt@C=42pt@M=4pt{
X^{i}  \ar@{->>}[dr] \ar[r]^{d_{X}^{i}}  &  X ^{i+1}  \\ 
    &  B^{i+1}( X ),   \ar@{^{(}->}[u]  \\
}
$$
we have a commutative diagram 
$$
\xymatrix{
\Hom _{R}(X^{i+1}, M)   \ar[d]^{\nu} \ar[r]^{(d_{X}^{i})^{*}} &  \Hom _{R}(X^{i}, M)  \\
\Hom _{R}(B^{i+1}( X ), M),   \ar@{^{(}->}[ur]  &    \\
}
$$
from which we see that the mapping $\nu$  has the image  $B^{-i}(\Hom _{R}( X , M))$.
Hence we have an exact sequence 
$$
\xymatrix{0 \ar[r]  & B^{-i}(\Hom _{R}( X , M)) \ar[r] & \Hom _{R}(B^{i+1}( X ), M) \ar[r] & \Ext _{R}^{1} (C^{i+1}( X ) , M) \ar[r] & 0.}  
$$
Combine all the mappings together, we finally have a commutative diagram with exact rows and columns:
{\small
$$
\xymatrix@C=12pt{
0  \ar[r] & B^{-i}(\Hom _{R}( X , M) ) \ar[d] \ar[r] &  Z^{-i}(\Hom _{R}( X , M)) \ar[d]^{\lambda} \ar[r] & H^{-i}(\Hom _{R}( X , M)) \ar[d]^{\rho} \ar[r] &  0 \\
 0 \ar[r] & \Hom _{R}(B^{i+1}( X ), M) \ar[d] \ar[r] & \Hom _{R}(C^{i}( X ), M) \ar[r] & \Hom _{R}(H^{i}( X ), M)  \ar[r] &  \Ext _{R}^{2}(C^{i+1}( X ), M) \\
   & \Ext _{R}^{1}(C^{i+1}( X ), M) \ar[d] &&&  \\
   & 0 &&& \\
}
$$}
Since  $\lambda$  is an isomorphism, we have the desired exact sequence by the snake lemma.  
\end{proof}


\begin{example}\label{transpose}
Let  $M$  be a finitely generated $R$-module and let  
$$
\xymatrix@C=32pt{
P_1 \ar[r]^f  & P_0 \ar[r] & M \ar[r] & 0}
$$
be a projective presentation of $M$, where each  $P_i$  are finitely generated and projective. 
Recall that the transpose $\transpose (M)$  of $M$ is defined to be the cokernel of the dual mapping  $f^* = \Hom _R(f, R)$  of  $f$.  

Now, set the complex $X$ to be $\left[\xymatrix{ 0 \ar[r]& P_0 ^* \ar[r]^{f^*} \ar[r]& P_{1}^* \ar[r]& 0}\right]$. 
Then we have that $C^1 (X) = \transpose (M)$ and 
$X^* = \left[\xymatrix{ 0 \ar[r] & P_{1} \ar[r]^{f} & P_0 \ar[r]& 0}\right]$. 
Therefore  $H^0(X^*) = M$  and  $H^0(X)^* = M^{**}$ in this case. 
It is easily verified that the mapping  $\rho _{X R}^0$ is {isomorphic to} the canonical mapping $M \to M^{**}$. 
Thus applying Theorem \ref{exact}, we have an exact sequence 
$$
\xymatrix{ 
0 \ar[r] & \Ext ^1 (\transpose (M), R) \ar[r] & M \ar[r] & M^{**} \ar[r] & \Ext^2(\transpose (M), R),  \\
}
$$
as shown in \cite[Chapter 2]{AB}.  
\end{example}

\vspace{6pt}
\section{*Torsion-free and *reflexive complexes}

\begin{definition}\label{deftorfree}
Let   $X \in \K$. 
We denote by  $X^{*}$  the $R$-dual complex  $\Hom _{R}(X, R)$. 
As we remarked in Definition \ref{rho},  we have a natural mapping 
$$
\rho ^{i}_{X R} :  H^{-i} ( X^{*} )  \to  H^{i}( X ) ^{*}
$$ 
for all $i \in \Z$. 
We say that  the complex  $X$ is {\bf *torsion-free} if  $\rho ^{i}_{X R}$  are injective mappings for all $i \in \Z$.  
Likewise, we say that $X$ is {\bf  *reflexive} if  $\rho ^{i}_{X R}$  are isomorphisms for all $i \in \Z$. 
\end{definition}

\begin{lemma}\label{restatement}
Let  $X$  be a complex in  $\K$. 
\begin{enumerate}
\item 
$X$ is *torsion-free if and only if it satisfies the following condition: \medskip 
\begin{itemize}
\item[$(*)$] For any chain map $f : X \to R[i]$ with $i \in \Z$, if  $H(f) =0$ then $f=0$ as a morphism in $\K$. 
\end{itemize} \medskip 
\item
Assume that $X$ satisfies the condition $(*)$. 
Then  $X$ is *reflexive if and only if it satisfies the following condition: \medskip 
\begin{itemize}
\item[$(**)$] If $a : H^{-i}( X ) \to R$ is an $R$-module homomorphism where  $i \in \Z$, 
then there is a chain map  $f : X \to R[i]$  such that   $H^{-i}( f ) = a$. 
\end{itemize}
\end{enumerate}
\end{lemma} 

\begin{proof}
This is just a restatement of the definition. 
\end{proof}

\begin{remark}\hfil
\begin{enumerate}
\item
Let  $X$ be *torsion-free (resp. *reflexive). 
Then so are any shifted complexes  $X[i]$  for  $i \in \Z$. 
Any direct summands of $X$  are also *torsion-free (resp. *reflexive).   
\item
Any direct sums of *torsion-free complexes are *torsion-free. 
(As we will see in Section 5, the category $\K$ admits certain kind of infinite direct sums. 
This remark says that if  $\{X_{i} \ | \ i \in I\}$ is a set of *torsion-free complexes and if $X = \coprod  _{i \in I} X_i$ exists in $\K$,  then  $X$  is also *torsion-free.   The proof is clear from Lemma \ref{restatement}(1))
\item
Any direct sums of finite number of *reflexive complexes are *reflexive. 
\end{enumerate}
\end{remark}

The following is straightforward from Theorem \ref{exact}. 

\begin{theorem}\label{torsion-free criterion}
Let  $X \in \K$. 
\begin{enumerate}
\item 
$X$ is *torsion-free if and only if  $\Ext ^1 _R(C(X), R)=0$. \medskip 
\item
If  $\Ext ^1 _R(C(X), R)= \Ext ^2 _R(C(X), R)=0$, then $X$ is *reflexive. 
\end{enumerate}
\end{theorem} 


\begin{corollary}\label{cor}
If  $R$  is a Gorenstein ring of dimension zero, then every complex $X \in \K$ is *reflexive (and hence *torsion-free).
\end{corollary}

\begin{proof}
In this case $\Ext_R^i( - ,R)=0$ for all $i>0$. 
\end{proof}

\begin{example}\label{module case}
Let  $M$ be a finitely generated $R$-module and  let 
$$
\xymatrix@C=32pt{\cdots \ar[r] & P_2 \ar[r] & P_1 \ar[r] & P_0 \ar[r] & M \ar[r] & 0} 
$$
be a projective resolution of $M$ with  $P_i \in \proj (R)$  for all $i >0$.  
\begin{enumerate}
\item
Setting 
$$
X = \left[ \xymatrix@C=32pt{\cdots \ar[r] & P_2 \ar[r] & P_1 \ar[r] & P_0 \ar[r] & 0} \right] \in \K, 
$$
we can easily see that the following three conditions are equivalent: \medskip
\begin{enumerate}
\item[(i)]
$\Ext_R^i(M, R) = 0$  for all $i>0$. \vspace{4pt}
\item[(ii)]  
$X$  is *torsion-free. \vspace{4pt}
\item[(iii)] 
$X$  is *reflexive.
\end{enumerate} \medskip 
\item 
Let  $n > 0$ be an integer. 
Considering the truncation of $X$, we set 
$$
X_{(n)} = \left[ \xymatrix@C=32pt{ 0 \ar[r] & P_n \ar[r]  & \cdots \ar[r] & P_1 \ar[r] & P_0 \ar[r] & 0} \right] \in \K.  
$$
Then  $X_{(n)}$  is *torsion-free if and only if 
$\Ext _R ^i (M, R) =0$  for $1 \leq i \leq n$, 
while  $X_{(n)}$  is *reflexive if and only if 
$\Ext _R ^i (M, R) =0$  for $1 \leq i \leq n+1$. 
\end{enumerate}

\end{example}

\begin{proposition}\label{triangle argument}
Let
$$
\xymatrix{
X \ar[r]^{a} &  Y \ar[r]^(0.45){b}   &   Z  \ar[r]^(0.4){c}  &  X[1] \\
}
$$  
be a triangle in $\K$. 
\begin{enumerate}
\item
Suppose that $H(b)^{*}: H(Z)^{*} \to H(Y)^{*}$ is injective.
If $X$ and  $Z$ are  *torsion-free, then so is $Y$. 
\item
Suppose that the sequence 
$\xymatrix{H(Z)^{*} \ar[r]^{H(b)^{*}} & H(Y)^{*} \ar[r]^{H(a)^{*}} & H(X)^{*}}$ is exact.
If $Z$ is *reflexive and if  $Y$  is *torsion-free, then $X$ is *torsion-free.
\item
Suppose that the sequence 
$\xymatrix{H(Z)^{*} \ar[r]^{H(b)^{*}} & H(Y)^{*} \ar[r]^{H(a)^{*}} & H(X)^{*}}$ is exact.
And assume that  $X$ and $Z$ are *reflexive and that  $Y$ is *torsion-free. 
Then $Y$ is *reflexive. 
\end{enumerate}
\end{proposition}

\begin{proof}
(1) Let  $f: Y  \to R[i]$  be a chain map with $i \in \Z$. 
Assume  $H(f) =0$. 
Then  $H(fa)= H(f)H(a)=0$. 
Since  $X$  is *torsion-free, it follows that  $fa = 0$  in  $\K$. 
Then there is a morphism  $g : Z \to R[i]$ such that  $f = gb$. 
Thus we have  $0 = H(f) = H(g)H(b)=H(b)^{*}(H(g))$  and since  $H(b)^{*}$  is injective, it follows    $H(g)=0$. 
However, since  $Z$ is *torsion-free, we have  $g=0$. 
Therefore  $f = gb =0$. 

(2) 
Let  $f: X \to R[i]$  be a chain map for  $i \in \Z$ and we assume that  $H(f) =0$. 
Then,  $H (f\cdot c[-1]) = 0$, and it follows that $f\cdot c[-1] =0$, since  $Z$  is *torsion-free. 
Hence there is a morphism  $g : Y \to R[i]$  with  $f = ga$. 
Note that there is a commutative diagram of graded $R$-modules with an exact row: 
$$
\xymatrix{ 
 H(X) \ar[r]^{H(a)}  &  H(Y) \ar[d]^{H(g)}  \ar[r]^{H(b)}  &  H(Z)  \\
 & H(R[i])=R[i]  &   \\
}
$$
Since  $H(a)^*(H(g)) = H(g) H(a) = H(f) = 0$, it follows from the  assumption that $H(g)$  induces a graded $R$-module homomorphism $\epsilon : H (Z) \to H(R[i])$ with  $H(g) = \epsilon H(b)$. 
Since  $Z$  is *reflexive, there is a chain map  $h : Z \to R[i]$ such that $H(h) = \epsilon$. 
Then, we have  $H(g-hb) = H(g) - H(h)H(b)=0$. 
Since  $Y$  is *torsion-free, it follows that  $g = hb$. 
Thus  $f = ga = hba=0$ as desired.

(3)
Let  $\alpha : H(Y) \to R$  be any element of  $H (Y)^*$. 
Since  $\alpha  H(a) \in H(X)^*$ and since  $X$  is *reflexive, there is a morphism  $f : X \to R$  such that $\alpha H(a) = H(f)$. 
Then we have 
$H(f\cdot c[-1]) = \alpha H(a)H(c[-1])=0$. 
Thus it follows from the *torsion-free property of $Z$ that $f\cdot c[-1] =0$. 
Then there is a morphism  $g : Y \to R$ with  $f = ga$. 
Therefore we have 
$\alpha H(a) = H(ga) = H(g)H(a)$, or equivalently  $(\alpha - H(g))H(a)=0$.
By the exact sequence 
 $\xymatrix{H(Z)^{*} \ar[r]^{H(b)^{*}} & H(Y)^{*} \ar[r]^{H(a)^{*}}  &  H(X)^{*}}$, we find an element  $\beta \in H(Z)^*$  satisfying   $\alpha - H(g) = \beta H(b)$. 
Since $Z$  is *reflexive, we have  $\beta = H (h)$  for some morphism  $h :Z \to R$. 
Thus we have  $\alpha = H (g) + \beta H(b) = H(g+hb)$.  
\end{proof}

\vspace{6pt}
\section{Complexes over a generically Gorenstein ring}

{Recall} that a finitely generated $R$-module $M$  is called {\bf torsionless}  if it satisfies one of the following equivalent conditions: 
(See also Example \ref{transpose}.) 
\begin{enumerate}
\item 
$M$  is a submodule of a free $R$-module.  \medskip
\item
The natural mapping  $M \to M^{**}$  is injective.  \medskip
\item
$\Ext _R ^1 (\transpose M, R) =0$.   \medskip
\end{enumerate}

On the other hand an $R$-module $M$ is said to be {\bf torsion-free} if the natural mapping  $M \to S^{-1}M$ is injective, where  $S$ is the multiplicatively closed subset  $R \backslash \bigcup _{\p \in \Ass (R)}\p$ consisting of all non-zero divisors of  $R$. 
Note that every torsionless module is torsion-free. 

Recall that a Noetherian commutative ring $R$  is said to be {\bf generically Gorenstein} if every localization  $R_{\p}$  for $\p \in \Ass( R )$ is a Gorenstein local ring, or equivalently the total quotient ring of  $R$  is a Gorenstein ring of dimension zero. 
The following lemma is well-known.

\begin{lemma}\label{lemma torsion-free}
Let  $R$  be a generically Gorenstein ring. 
Then a finitely generated $R$-module $M$ is torsionless if and only if $M$ is torsion-free. 
\end{lemma}

\begin{proof}
We have only to prove the \lq if \rq  part of the lemma.   
Let   $S =  R \backslash \bigcup _{\p \in \Ass (R)}\p$.  
There is a commutative diagram; 
$$
\xymatrix@C=52pt{ 
 M \ar[r]^(0.35){\alpha}  \ar[d]_{\beta}  &   \Hom _R (\Hom _R (M, R), R) \ar[d]  \\
 S^{-1}M \ar[r]^(0.25){S^{-1}\alpha}  &  \Hom _{S^{-1}R}(\Hom _{S^{-1}R}({S^{-1}} M, S^{-1}R), S^{-1}R),  \\
} 
$$
where the vertical arrows are the mapping induced by the localization {at} $S$. 
If $M$ is torsion-free, then $\beta$ is injective. 
Since  $S^{-1}R$  is a Gorenstein ring of dimension zero, 
$S^{-1}\alpha$ is an isomorphism. 
As a result, it follows that $\alpha$ is injective, hence  $M$ is torsionless. 
\end{proof}

\begin{theorem}\label{theorem torsion-free}
Let  $R$  be a generically Gorenstein ring. 
Then the following two conditions are equivalent for  $X \in \K$: \medskip  
\begin{enumerate}
\item 
$X$  is *torsion-free. \medskip 
\item
Each cohomology module $H^i (X^* )$  is a torsion-free $R$-module for $i \in \Z$. \medskip
\end{enumerate}
\end{theorem}

\begin{proof}
$(1) \Rightarrow (2)$: 
Before the proof we recall that $N^*$  is torsionless for any finitely generated $R$-module $N$.
(If  $R^m \to N$ is a surjective mapping of $R$-modules,  then we have an injection  $N^*$  to a free module  $(R^m)^*$.)  

By definition  $\rho ^i_{X R} : H^{-i}(X^*) \to H ^{i}(X) ^*$  is injective. 
Since  $H^i (X)$  is a finitely generated $R$-module,   $H ^{i}(X) ^*$ is torsionless. 
This forces  $H^{-i}(X^*)$  to also be torsionless, and hence torsion-free. 

$(2) \Rightarrow (1)$: 
Let  $S = R \backslash \bigcup_{\p \in \Ass (R )}\p$  as in Lemma \ref{lemma torsion-free}. 
Now let  $f: X  \to R[i]$  be a chain map with $i \in \Z$ and assume  $H(f) =0$. We want to show that  $f =0$ as an element of  $H^i ( X^{*})$. 

Note that $S^{-1}f : S^{-1}X \to S^{-1}R[i]$  is a chain map with $H (S^{-1}f)=0$. 
Since  $S^{-1}R$  is a Gorenstein ring of dimension zero, we have from Corollary \ref{cor} that  $S^{-1}X$  is *torsion-free as a complex over  $S^{-1}R$, hence $S^{-1}f =0$  in  $\Ko(S^{-1}R)$.
This means that  $S^{-1}f = 0$  as an element of 
$H(\Hom _{S^{-1}R}(S^{-1}X, S^{-1}R[i]))$. 
Since each term of $X$ is a finitely generated $R$-module, we note that there is  a natural isomorphism
$\Hom _{S^{-1}R}(S^{-1}X, S^{-1}R[i])\cong S^{-1}\Hom _{R}(X, R[i])$, hence 
$$
H(\Hom _{S^{-1}R}(S^{-1}X, S^{-1}R[i]))\cong S^{-1}H(\Hom _{R}(X, R[i])) = S^{-1}H^i (X^*). 
$$
This shows that there is an element $s \in S$ with  $sf =0$ as an element of 
$H^i(X^{*})$. 
Since we assumed that  $H^i(X^{*})$  is a torsion-free $R$-module, we must have  $f =0$ as an element of $H^i(X^{*})$. 
\end{proof}

\begin{remark}
The implication $(1) \Rightarrow (2)$ in the theorem is generally true without the assumption of generic Gorensteinness.
But it is not the case for $(2) \Rightarrow (1)$. 

For example, let  $(R, \m, k)$ be a local ring with $\dim R >0$  and  $\depth R =0$. 
Note in this case that every  $k$-vector space is torsionless, hence torsion-free,  as an $R$-module, 
since $k$ is isomorphic to a submodule in $R$. 
Now let  $X$ be an $R$-free resolution of  $k$. 
Then it follows that $H^i (X^*) \cong \Ext _R ^i (k, R)$ is a torsion-free $R$-module for each $i$, hence $X$ satisfies the condition $(2)$. 
On the other hand, we note that  $H^i (X)^* \not= 0$ only if $i =0$. 
Hence the condition $(1)$ forces   $\Ext _R ^i (k, R) = 0$ for all $i >0$, which is an equivalent condition for $R$ to be   a Gorenstein ring of dimension zero. 
Therefore  $X$ does not satisfy the condition $(1)$. 
\end{remark}

Note that a finitely generated module  $M$ over a commutative Noetherian ring $R$  is said to be reflexive if the natural mapping  $M \to M^{**}$  is an isomorphism.

Recall that a commutative Noetherian ring $R$ is said to be {\bf Gorenstein in depth one} if each  $R_{\p}$  is a Gorenstein ring for all the prime ideals  $\p$ satisfying  $\depth R_{\p} \leqq 1$. 

First we remark the following (perhaps well-known) lemma. 

\begin{lemma}\label{lemma ref}
Assume that  $R$  is Gorenstein in depth one. 
\begin{enumerate}
\item
If  $M$  is a finitely generated $R$-module, then  $M^*$  is a reflexive $R$-module. \medskip
\item 
Let  $M \subseteq N$  be a submodule of a finitely generated $R$-module which is equal in depth one, i.e. $M_{\p} = N _{\p}$  if  $\depth R_{\p} \leqq 1$. 
Furthermore assume that both  $M$ and $N$ are reflexive. 
Then $M = N$.  \medskip 
\end{enumerate}
\end{lemma}

\begin{proof}
(1) 
Since  $M^*$  is a torsionless module, the natural mapping $\alpha: M^* \to M^{***}$ is injective. 
Set  $C$  to be the cokernel of this map, i.e. $C = {\Cok} (\alpha)$. 
Then, by the assumption, we have  $C_{\p} = 0$  if  $\depth R_{\p} \leqq 1$. 
(Note that $M^{*}_{\p}$ are torsion-free, hence MCM's over Gorenstein rings $R_{\p}$ for those $\p$, hence  $\alpha _{\p}$  are isomorphisms.)
To prove  $C = 0$, let us assume that  $C \not= 0$ and take a minimal prime ideal $\p$  in $\Supp (C)$. 
Then, by the above, we must have  $\depth R_{\p} \geqq 2$. 
Note that there is an exact sequence of $R_{\p}$-modules 
$0 \to M _{\p}^* \to M_{\p}^{***} \to C_{\p} \to 0$, where $C_{\p}$ is a non-zero $R_{\p}$-module of finite length. 
Remark here that both $M _{\p}^*$ and $M_{\p}^{***}$  are second syzygy modules over $R_{\p}$. 
Since  $\depth R_{\p}\geqq 2$, it follows that such second syzygy modules have depth at least {2}. 
Noticing that $\depth \, C_{\p}=0$, we see that this contradicts the depth lemma  (see \cite[Proposition 1.2.9]{BH} {or \cite[1.2.6]{Av}}). 

\par
(2)
Setting   $C = N/M$, we want  to show  $C =0$. 
By the  assumption, if  $\depth R_{\p} \leqq 1$, then  $C_{\p} =0$. 
Thus every prime  $\p$  in $\Supp (C)$  satisfies $\depth R_{\p} \geqq 2$. 
Assuming $C \not= 0$, we take a minimal prime ideal in $\Supp (C)$. 
Then there is an exact sequence of $R_{\p}$-modules $0 \to M_{\p} \to N_{\p} \to C_{\p} \to 0$, where $C_{\p}$ is of finite length. 
Since $M_{\p}$  (resp. $N_{\p}$ ) is a reflexive $R_{\p}$-module, 
the depth of  $M_{\p}$ (resp. $N_{\p}$) is at least {2}. 
This contradicts the depth lemma again. 
\end{proof}

\begin{theorem}\label{theorem ref}
Suppose that $R$  is Gorenstein in depth one and let  $X$  be a complex  in $\K$. 
Then the following two conditions are equivalent: 
\begin{enumerate}
\item 
$X$  is *reflexive. \medskip 
\item
Each cohomology module $H^i (X^* )$  is a reflexive $R$-module for $i \in \Z$. 
\end{enumerate}
\end{theorem}

\begin{proof}
$(1) \Rightarrow (2)$: 
Since  $X$  is *reflexive, we have an isomorphism  $\rho ^i_{X R} : H^{-i}(X^*) \to H^i(X)^*$  for all $i \in \Z$. 
Note that each  $H^i(X)$  is a finitely generated $R$-module. 
It thus follows from Lemma \ref{lemma ref}(1) that  $H^i(X)^*$, hence  $H^{-i}(X^*)$ as well, is a reflexive $R$-module.

$(2) \Rightarrow (1)$: 
We want to show that the natural mapping $\rho ^i_{XR} : H^{-i}(X^*) \to H^i(X)^*$ is an isomorphism for each $i \in Z$. 
We know, from Theorem \ref{theorem torsion-free}, that $X$ is *torsion-free, hence all $\rho ^i_{XR}$ are injective. 
Thus, applying Lemma \ref{lemma ref}(2), we have only to show that  $(\rho ^i_{XR})_{\p}$ are isomorphisms for prime ideals $\p$ with  $\depth R_{\p} \leqq 1$. 
Therefore the proof is reduced to the case where the ring $R$  is a Gorenstein local ring of dimension at most one. 
Henceforth we assume $R$ is such a ring. 
In this case, we have   $\Ext _R ^2 (C(X), R) =0$, 
 thus  it results from Theorem \ref{exact} that $\rho ^i_{XR} : H^{-i}(X^*) \to H^i(X)^*$ is surjective for each $i \in \Z$. 
Since we know already that this is injective, each $\rho ^{i}_{XR}$ is an isomorphism. 
\end{proof}

\vspace{6pt}
\section{Split complexes and $\F$}

We note that $\K$  admits finite direct sums, and moreover some kind of infinite direct sums can be possibly taken inside $\K$. 
For example, let  $\{ X_j \ | \ j \in J\}$  be a set of complexes in  $\K$ and assume that  $X^i = \bigoplus _{j \in J} X_j^i$  is  a finitely generated $R$-module for each  $i \in \Z$. 
In such a case  the direct sum  $X = \coprod _{j \in J} X_j$  (or the coproduct in $\K$) is well-defined so that its $i$th component is $X^i$. 
Note in this case that the direct sum  coincides with the direct product $\prod _{j \in J} X_j$, as we see in the next lemma.  

The direct sum  $\coprod _{i \in \Z } R[i]$  is one of such typical examples of infinite direct sums, actually it is a complex of the form 
$\left[ \xymatrix{ \cdots \ar[r]^0 & R \ar[r]^{0} & R \ar[r]^(0.45){0} & R \ar[r]^(0.4){0}  & \cdots} \right]$ that belongs to $\K$. 
  
\begin{lemma}\label{sum-product lemma}
Let  $\{ X_j \ | \ j \in J \}$  be a set of complexes in  $\K$. 
Assume that,  for each $i \in \Z$,  there is a finite subset $J_i \subseteq J$ such that  $X_j^i \not= 0$ only if $j \in J_i$.  
Then  the coproduct $X = \coprod  _{j \in J} X_j$ exists in $\K$. 
Moreover in this case, the coproduct is a product in $\K$, i.e.
$
X = \prod _{j \in J} X_j
$. 
Hence there is an isomorphism of $R$-modules  
$$
\Hom _{\K} (Y, \  \coprod _{j \in J} X_j) \cong \prod _{j \in J} \Hom _{\K} (Y, X_j)
$$
for  all  $Y \in \K$. 
\end{lemma}
 
\begin{proof}
Let  $\Mod (R)$  be the abelian category consisting of all (not necessarily finitely generated) $R$-modules and we denote by $K (\Mod (R))$ the homotopy category of all complexes over  $\Mod (R)$. 
Now regarding   $\{ X_j \ | \ j \in J \}$  as an object set in $K (\Mod (R))$,  we see that the coproduct  $X$ in  $K(\Mod (R))$ is given as  $X ^i = \bigoplus _{j \in J}X^i_j$ with differentials defined diagonally by each $d_{X_j}^i$. 
Similarly the product in $K (\Mod (R))$ is given as  $\prod _{j \in J} X_j^i$. 
Now the assumption of the lemma assures that each $X^i$  is finitely generated, hence the coproduct $X$  in $K(\Mod (R))$ lies in its full subcategory $\K$. 
This shows that $X$ is in fact a coproduct in the category $\K$. 

Moreover, under the assumption in the lemma we have the equality 
$\bigoplus _{j \in J}X^i_j  =  \prod _{j \in J} X_j^i$ as $R$-modules for all $i \in \Z$.
Hence the last half of the lemma follows. 
\end{proof}

\begin{definition}\label{smallest additive sub}
Given an $X \in \K$, we define  $\mathrm{Add} (X)$  as the smallest additive subcategory of  $\K$ containing  $X$ that is closed under the shift functor and admits possibly infinite coproducts.     
Equivalently  $\mathrm{Add} (X)$ is the intersection of  all the full subcategories  $\U$ satisfying the following conditions:   
\begin{itemize}
\item[(i)]
$\U$  is  closed under isomorphism and $X \in \U$. \medskip
\item[(ii)]
If  $Y \in \U$ then $Y[i] \in \U$  for all $i \in \Z$. \medskip
\item[(iii)]
If  $Z$ is a direct summand of  $Y \in \U$ then  $Z \in \U$. \medskip
\item[(iv)]
Let $\{ Y_j \ | \ j \in J\}$ be a set of objects in  $\U$ and assume that the coproduct  $\coprod _{j \in J} Y_j$  in $\K$ exists. 
Then  $\coprod _{j \in J} Y_j \in \U$.
\end{itemize}
\end{definition}

(Note that $0$ is an object of $\U$ by (iii) and that all null complexes belong to $\U$ by (i).) 

In the rest of the paper we are particularly interested in  $\mathrm{Add}  (R)$, where  $R$  is regarded as a complex concentrated in degree $0$. 

If  the complex 
$$
X = \left[ \xymatrix@C=32pt{ \cdots \ar[r]^{d^{-2}_X} & X^{-1} \ar[r]^(0.6){d^{-1}_{X}} &  X^{0} \ar[r]^{d^{0}_{X}}  &  X^1 \ar[r]^{d^{1}_{X}} & \cdots} \right] , 
$$
satisfies the equalities  $d_{X}^{{i}}  = 0$  for all $i \in \Z$, then  $X$ belongs to $\F$, since $X$  is a direct sum $\coprod  _{i \in \Z} X^i [-i]$ with each  $X^i$  being a projective $R$-module. 
Such a complex $X$  is characterized by the condition that  $X \cong H(X)$ in $\C$, where we regard the graded $R$-module $H(X)$ as a complex with zero differentials.

Recall that  a complex $X \in \C$ is called {\bf split} if there is a graded $R$-module homomorphism $s : X \to X[-1]$  satisfying  $d_X s d_X = d_X$. 
(Cf. \cite[Definition (1.4.1)]{W}.) 
To state the following well-known lemma,  we recall the notation $C(X) = \mathrm{Coker} (d_X)$ and  $B(X) = \mathrm{Im} (d_X)$,  for a complex $X \in \C$,  as in Notation \ref{c}.

\begin{lemma}\label{characterization of split}
The following conditions are equivalent for  $X \in \C$. 
\begin{enumerate}
\item   $X$ is split. \medskip
\item  There is a direct sum decomposition  $X = X' \oplus N$  in  $\C$  where  $d_{X'} = 0$ and  $N$ is a null complex. \medskip
\item  $C(X) = \bigoplus _{i \in \Z} C^i (X)$  is a projective $R$-module {as underlying $R$-module}. \medskip
\item  The natural inclusion map  $B(X)= \bigoplus _{i \in \Z} B^i(X)  \hookrightarrow X = \bigoplus _{i \in \Z} X^i$  is a split monomorphism as graded $R$-modules. 
\end{enumerate}
\end{lemma}

\begin{proof}
The implications  $(2) \Rightarrow (1) \Rightarrow (4) \Rightarrow (3)$ are well-known and easily proved. 
We have only to show $(3) \Rightarrow (2)$.

If  $C(X)$  is projective, then the natural exact sequences of graded $R$-modules 
$$
\xymatrix@C=16pt{0 \ar[r]& B(X) \ar[r]& X \ar[r]& C(X) \ar[r]& 0,}
\quad  \xymatrix@C=16pt{0 \ar[r]& H(X) \ar[r]& C(X) \ar[r]&B(X)[1]\ar[r]& 0}
$$   
are splitting. 
Therefore each  $X^i$  decomposes to  $X_0^i \oplus X_1^i \oplus X_2 ^i$  where  $X_0^i \cong H^i(X)$ and $X_1^i \cong B^i(X)$, $X_2 ^i \cong B^{i+1}(X)$  for  $i \in \Z$, and the differential map  $d_X^i$ yields an isomorphism  $X_2^i \to X_1^{i+1}$, while it is zero on $X_0^i \oplus X_1^i$. 
Thus, part  $X_1 \oplus X_2$ of $X$ defines a null subcomplex $N$.  
Therefore, setting  $X' = X_0$ with zero differentials,  we have a direct sum decomposition  $X = X' \oplus N$. 
\end{proof}

As a result of the equivalence $(1) \Leftrightarrow (2)$ in the lemma, we see that all the split complexes in $\K$   belong  to  $\F$. 
We can show that the uniqueness of the direct sum decomposition in the meaning of $(2)$ in the lemma holds for a split complex. 

\begin{lemma}\label{uniqueness of decomposition}
Let  $X$  be a split complex belonging to  $\C$. 
Assume there are decompositions  $X = X_1 \oplus N_1 = X_2 \oplus N_2$ where  $d_{X_i} =0$ and $N_i$ is a null complex for $i =1, 2$. 
Then we have isomorphisms $X_1 \cong X_2$,   $N_1 \cong N_2$  in $\C$. 
\end{lemma}

\begin{proof}
Write the natural injection $X_1 \hookrightarrow X = X_2 \oplus N_2$  as  $\binom{a}{b}$ where $a : X_1 \to X_2$  and $b: X_1 \to N_2$. 
Similarly write the natural projection  $X_2 \oplus N_2 = X \twoheadrightarrow X_1$  as $(c, d)$ with $c : X_2 \to X_1$  and  $d : N_2 \to X_1$. 
Then we have $1 _{X_1} = ca + db$. 
Since the morphism $db$ factors through a null complex, it is null homotopic. 
Hence it follows from the next remark that  $db =0$ as a morphism in $\C$. 
Thus $ca = 1_{X_1}$. 
In the same way as this one can show  $ac = 1_{X_2}$. 
Hence  $a : X_1 \to X_2$  is an isomorphism in $\C$. 

To show $N_1 \cong N_2$  in $\C$ we remark that, for a null complex $N$, we have  $Z(N) \cong N/Z(N) [-1]$ as graded $R$-modules and $N$ is isomorphic to the mapping cone of the identity mapping on  $N/Z(N)$. 
Since  $Z(X_i) = X_i$ for  $i=1,2$, we have  $X/Z(X) = (X_i \oplus N_i)/(Z(X_i) \oplus Z(N_i) ) \cong  N _i / Z(N_i)$. 
Therefore $N_1/Z(N_1) \cong  N_2 /Z(N_2)$  as graded $R$-modules. 
Since both $N_1$ and $N_2$  are null complexes, we have  an isomorphism  $N_1 \cong N_2$ in $\C$, as remarked above.      
\end{proof}

\begin{remark} 
Let  $f: X \to Y$ be a morphism in $\C$, where we assume that  $d_X = d_Y =0$. 
If  $f$  is null homotopic, then $f=0$ in $\C$.  
In fact, this follows from that $f = d_Y h - h d_X =0$  for a homotopy $h$. 
\end{remark}

By a similar proof to the lemma above we can also show the following lemma.

\begin{lemma}
Let  $X$ and  $Y$  be complexes in $\C$ such that  $d_X=0$. 
If  $X$ is a direct summand of  $Y$  in $\K$, then it is also a direct summand of $Y$ in $\C$. 
\end{lemma}

\begin{proof}
Assume there are morphisms  $f : X \to Y$  and  $g : Y \to X$  in $\C$  such that $gf$ is chain homotopic to the identity morphism     $1_X$ on  $X$. 
Then it follows from the remark above that $1_X -gf = 0$ as a morphism in $\C$. 
\end{proof}

\begin{proposition}\label{coproduct}
Let  $\{ X_j \ | \ j \in J\}$ be a set of complexes in $\C$  such that  $d_{X_j}=0$ for all $j \in J$. 
Assume that the coproduct $\coprod _{j \in J} X_j$  in $\K$ exists. 
Then,  for any $i \in \Z$, there is a finite subset $J_i \subseteq J$ such that $X_j^i \not= 0$  only if    $j \in J_i$. 
In this case, the coproduct is an ordinary direct sum of  complexes. 
Hence  $\coprod _{j \in J} X_j$  has zero differentials, and it is a split complex as well.  
\end{proposition}

\begin{proof}
Set  $P = \coprod_{j \in J} X_j$. 
By definition we have an isomorphism 
$$
\Hom _{\K} (P,  - ) \cong \prod _{j \in J} \Hom _{\K}(X_j , - ) \cong \Hom _{K(\Mod(R))} (\bigoplus _{j \in J} X_j , - ) |_{\K}
$$
as functors on  $\K$, where  $\bigoplus _{j \in J} X_j$  denotes the coproduct in  $K (\Mod (R))$. 
Therefore there is a morphism  $\bigoplus _{j \in J} X_j \to P$  in  $K (\Mod (R))$,  by which any finite direct {sum}  $\bigoplus _{k=1}^r X_{j_k}$  is a direct summand of  $P$ in the category $\K$. 
Then it follows from the previous lemma that  any such finite direct sums $\bigoplus _{k=1}^r X_{j_k}$ are direct summands of $P$ in the category $\C$. 
In particular,   for each $i \in \Z$,  any finite direct sum $\bigoplus _{k=1}^r X_{j_k}^i$ of $R$-modules is a direct summand of $P^i$.  

Under such a circumstance, for any $i \in \Z$, we claim that $X_j ^i = 0$ for almost all  $j \in J$ (i.e. for all $j \in J$  except a finite number of them).

In fact, if  $R$  is an integral domain, then this is true by rank argument, since $P^i$ is finitely generated and hence it has a finite rank.  
For general cases, set $\mathrm{Min} (R) = \{ \p _1, \ldots , \p _\ell\}$  and it follows from the domain case that $X_j^i  \otimes _R R/\p_i =0$ for almost all  $j \in J$.  
Then we have
$$
X_j^i \subseteq \p_1 X_j^i \subseteq \p _1 ^2 X_j^i \subseteq \cdots \subseteq \p _1^N X_j ^i \subseteq \p_1^N \p_2 X_j^i\subseteq \cdots \subseteq (\p_1 \p_2 \cdots \p_\ell)^N X_j^i
$$
for almost all $j \in J$ and any $N>0$. 
Taking $N$ enough large so that $(\p_1 \p_2 \cdots \p_\ell)^N=0$, we have $X_j^i=0$ for all such $j \in J$.
\end{proof}

Now we are able to state a main result of this section.

\begin{theorem}\label{characterization of F}
The following conditions are equivalent for  $X \in \K$. 

\begin{enumerate}
\item   $X$ belongs to $\F$. \medskip
\item  $X$ is a split complex. \medskip 
\item 
The natural mapping 
$$
H : \Hom _{\K }(X, Y) \longrightarrow  \Hom _{\mathrm{graded} R\mathrm{-mod} }(H(X), H(Y))
$$
which sends $f$  to $H(f)$  is injective for all $Y \in \K$.  \medskip
\item 
The natural mapping $H$  in the condition  $(3)$  is bijective for all $Y \in \K$. 
\end{enumerate}
\end{theorem}

\begin{proof}
We have shown the implication  $(2) \Rightarrow (1)$ in Lemma \ref{characterization of split}. 

\vspace{6pt}

$(1) \Rightarrow (2)$: 
Let  $\U$  be the subcategory of $\K$ consisting of all split complexes. 
Note that $R \in \U$ and that $\U$   is closed under shift functor,  and taking direct summands. 
If we prove that  $\U$ is closed under taking coproducts in $\K$, then  $\F \subseteq \U$ by Definition \ref{smallest additive sub} and the proof will be finished.  

Let  $\{ X_j \ | \ j \in J\}$ be a set of complexes in $\U$. 
By Lemma \ref{uniqueness of decomposition} each $X_j$ is uniquely decomposed into $X_j' \oplus N_j$  with  $d_{X_j'} = 0$ and  {a null complex} $N_j$.   
Since  $X_j \cong X_j'$  in $\K$, replacing  $X_j$  with  $X_j'$  we may assume $d_{X_j} =0$ for all $j \in J$. 
If the coproduct $\coprod _{j \in J} X_j$  exists in $\K$,  then it follows from the previous proposition it is split again hence belongs to $\U$.

\vspace{6pt}

$(2) \Rightarrow (4)$: 
As in the proof above we may assume that  $d_X=0$, hence  $X =H(X)$. 

Let  $f : X \to Y$ be a morphism in $\K$ and assume that  $H(f)=0$.  
Then the image of $f^i$  is contained in the coboundary $B ^i(Y)$ for $i \in \Z$. 
Since  $X^i$ is a projective module, there is an  $h^i : X^i \to Y^{i-1}$  with  $f^i = d_Y^{i-1} \cdot h^i$. 
Thus $\{ h^i\  | \ i \in \Z\}$  gives a homotopy, and we have  $f=0$  as a morphism in $\K$.   

To show the surjectivity of $H$, let  $a : H(X) \to H(Y)$  be a graded $R$-module homomorphism. 
Then each $a ^i : H^i(X) = X^i \to H^i(Y)$ is lifted to an $R$-module mapping  from  $X^i$ to the cocycle module $Z ^i (Y)$. 
These lifted maps define a chain map $f: X \to Y$ with $H(f) = a$.

\vspace{6pt}

$(4) \Rightarrow (3)$: Obvious.  

\vspace{6pt}

$(3) \Rightarrow (2)$: 
 Let $M$ be a finitely generated $R$-module and let $P \in \K$  be a projective resolution of  $M$. 
 Then note that the equality $\Hom _{\K} (X[i], P) = H^{-i}(\Hom _R (X, M))$ holds. 
To show this we introduce the notion of silly truncation 
$$
\sigma _{\leqq 1} (X[i]) := [\xymatrix{\cdots \ar[r] & X^0 \ar[r] & X^1 \ar[r] & \cdots \ar[r] & X^i \ar[r] & X^{i+1} \ar[r] & 0 }] 
$$
where $X^{i+1}$ sits in the $(+1)$st position. 
Note that this is K-projective. 
Then, since $P_i=0$ for $i > 0$, we have 
\begin{eqnarray*}
\Hom  _{\K} (X[i], P) &=& 
\Hom  _{\K} (\sigma _{\leqq 1} (X[i]), P) \\
&=& 
H^0(\Hom  _{R} (\sigma _{\leqq 1} (X[i]), M)) \\
&=& 
H^0(\Hom  _{R} (X[i], M)).
\end{eqnarray*}
Therefore  the mapping defined by taking cohomology modules 
 $H: \Hom _{\K} (X[i], P) \to \Hom _{\mathrm{graded} R-\mathrm{mod}} (H(X[i]), H(P)) =\Hom_R (H^i(X), M)$ 
 is just the same as $\rho _{X M}^i$  defined in Definition \ref{rho}. 
Thus the condition $(3)$  implies that $\rho _{X M}^i$  is injective for all $i \in \Z$  and for all $M \in \mod (R)$. 
It then follows from Theorem \ref{exact} that  
$\Ext ^1_R (C(X), M) =0$ for any $M \in \mod (R)$, and therefore  $C(X)$ is a projective $R$-module.   
Thus  $X$  is split by Lemma \ref{characterization of split}.
\end{proof}

As a result of Theorem \ref{characterization of F} {and Lemma \ref{characterization of split}}, we have the following corollary. 

\begin{corollary}
Every complex  $F$  in  $\F$  is decomposed as $F \cong \coprod _{j \in \Z} H^j (F)[-j]$ where  $H^j (F) \in \proj (R)$ for all $j \in \Z$. 
(Note that  $F^* = \coprod _{j \in \Z} H^j(F)^*[j]$ in this case. )
Moreover every complex in $\F$  is *reflexive. 
\end{corollary}

\begin{proposition}\label{injection proposition}
Let  $X, F \in \K$. 
Assume  that $F$ belongs to $\F$ and that  $X$ is *torsion-free (resp. *reflexive).  
Then the mapping 
$$
H \ : \ \Hom _{\K} (X, F) \longrightarrow \Hom _{\mathrm{graded} R\mathrm{-mod} }(H(X), H(F)) \ ; \ f \mapsto H(f) 
$$
is injective  (resp. bijective). 
\end{proposition}

\begin{proof}
We may take $F$  as  it satisfies $d _F =0$, hence  $F = H(F)$. 
Then, as remarked above,  $F = \coprod _{j \in \Z} F^j[-j]$  with  $F^j \in \proj (R)$ and this coproduct is also a product. 
Therefore, 
\begin{eqnarray*}
&\Hom _{\K} (X, F) = \prod _{j \in \Z} \Hom _{\K} (X, F^j[-j]), \ \ \text{and} \phantom{SA} \\
&\Hom _{\mathrm{graded} R\mathrm{-mod} }(H(X), H(F)) = \prod_{j \in \Z} (H^j(X), F^j[-j]).
\end{eqnarray*} 
According as  $X$ is *torsion-free or  *reflexive, we have that 
$H :   \Hom _{\K} (X, F^j [{j}]) \longrightarrow  \Hom _{\mathrm{graded} R\mathrm{-mod}}(H(X), F^j [{j}])$ is injective or bijective for each  $i, j  \in \Z$. 
The proposition follows from this observation. 
\end{proof}

The following theorem is one of the crucial results on *torsion-free complexes, on which the proof of the main Theorem \ref{1} will deeply rely. 
See Sections 10 and 12.

\begin{theorem}\label{S^{-1}f=0}
Assume that  $X \in \K$  is *torsion-free and that  $F \in \F$. 
Let  $f \in \Hom _{\K}(X, F)$. 
Setting  $S = R \backslash \bigcup _{\p \in \Ass (R)} \p$, 
if  $S^{-1}f=0$  as a morphism $S^{-1}X \to S^{-1}F$  in  $\Ko (S^{-1}R)$, then we have that $f=0$ as a morphism in $\K$. 
\end{theorem}

\begin{proof}
If  $S^{-1}f =0$  then $H(S^{-1}f)=0$  as an $S^{-1}R$-module homomorphism  $H(S^{-1}X) \to H(S^{-1}F)$. 
Thus  we see that  $S^{-1}H(f)=0$ as a mapping  $S^{-1}H(X) \to S^{-1}H(F)$.    
Since  $H(F)$  is a projective $R$-module, any elements of $S$ act on $H(F)$ as non zero divisors. 
It thus follows that  $H(f)=0$ as a mapping  $H(X) \to H(F)$. 
Then from Proposition \ref{injection proposition} we have  $f =0$. 
\end{proof}

\begin{corollary}\label{torsionless cor}
If $X \in \K$ is *torsion-free and  $F \in \F$, then 
$\Hom _{\K} (X, F)$  is a torsion-free $R$-module. 
\end{corollary}

\begin{proof}
There is a commutative diagram of  $R$-modules 
$$
\xymatrix{
\Hom _{\K} (X, F)  \ar[r]^{\alpha} \ar[rd]^{\gamma}  & S^{-1} \Hom _{\K} (X, F) \ar[d]^{\beta} \\
  & \Hom _{\Ko (S^{-1}R)}(S^{-1}X, S^{-1}F),  \\
}$$
where  $\alpha$  is a localization mapping by $S$ and  $\gamma$ is a natural mapping that sends $f$  to $S^{-1}f$. 
Note that  $\beta ({f}/{s}) = {\gamma (f)}/{s}$  for $f \in \Hom _{\K}(X, F)$  and $s \in S$. 
We have shown in Theorem \ref{S^{-1}f=0} that  $\gamma$  is injective. 
Thus  $\alpha$ is also injective, and hence $\Hom_{\K}(X, F)$  is a torsion-free $R$-module.     
\end{proof}

\begin{remark}
In the proof of the corollary, we should note that the natural mapping 
$$
\beta :  S^{-1}\Hom _{\K} (X, F) \longrightarrow \Hom _{\Ko(S^{-1}R)}(S^{-1}X, S^{-1}F)
$$
is not necessarily an isomorphism. 
For example, setting  $X = F = \coprod _{i \in \Z} R[-i]$, we have $\Hom _{\K} (X, F) = \prod _{i \in \Z} R$ and $\Hom _{\mathscr K (S^{-1}R)} (S^{-1}X, S^{-1}F) = \prod _{i \in \Z} S^{-1}R$.   
\end{remark}

\vspace{6pt}
\section{The stable category of $\K$}

The main objective of this paper is to consider the nature of complexes in $\K$ up to $\F$-summands, which we  call the stable theory  after the paper \cite{AB}. 

\begin{definition}\label{stable}
We denote by  $\uK$  the factor category  $\K$  modulo the subcategory $\F$: 
$$
\uK = \K / \F
$$ 
\end{definition}

We call $\uK$ the {\bf stable category} of  $\K$. 

The objects of $\uK$  are the same as $\K$, while the morphism set is given by
$$
\Hom _{\uK} (X, Y) = \Hom _{\K} (X, Y) / \F( X, Y), 
$$
for $X, Y \in \uK$, where  $\F (X,Y)$  is the $R$-submodule of $\Hom _{\K} (X, Y)$  consisting of all morphisms factoring through objects of $\F$. 
The object sets of $\K$ and $\uK$ are identical, but for an object  $X \in \K$, 
to discriminate it with an object in $\uK$, we often write  $\underline{X}$ for the corresponding object in $\uK$. 
Similarly we denote by $\underline{f}$  the corresponding morphism in $\uK$ for a given $f$ in $\K$.

Since  $\F$  is stable under the action of shift functor in $\K$, it should be noted that $\uK$ admits the shift functor so that 
$\underline{X[1]} = \underline{X}[1]$  for  $X \in \uK$. 
However $\uK$ is not a triangulated category, but merely an additive  $R$-linear category with the shift functor that is an auto-functor on it. 
($\uK$ is not triangulated, by which we mean that there is no triangle structure on  $\uK$ that makes the natural functor $\K \to \uK$ a triangle functor. 
This is true, since  $\F$  is not closed under triangles in $\K$.)   
  
First of all we remark on the commutativity of a diagram in $\uK$. 

\begin{lemma}\label{commutativity}
Let  $f : X \to Z$, $g : X \to Y$, $h : Y\to Z$. 
Then  $\underline{f} = \underline{h}  \underline{g}$  in $\uK$ if and only if 
there is  a commutative diagram in $\K$ of the following form: 
$$
\xymatrix@C=40pt@R=40pt{Y \oplus F  \ar[rr]^{(h \ a)} & & Z \\  
& X  \ar[lu]^{\binom{g}{b}} \ar[ru]_f &\\ } 
$$  
where  $F \in \F$. 
\end{lemma}

\begin{proof}
If  $f - hg$  factors through $F \in \F$, then there are  $a : F \to Z$  and  $b : X \to F$  that satisfy the equality  $f = hg + ab$.   
The converse is trivial since  $\underline{a}=\underline{b}=0$.   
\end{proof}

Note from this lemma that  $\underline{X} =0$ for  $X \in \K$  if and only if  $X \in \F$. 
In fact if  $1_{\underline{X}} = 0$,   then setting  $X=Z$, $Y=0$ and  $f = 1_X$ in the lemma, we see that  $X$ is a direct summand of  $F \in \F$  and hence  $X \in \F$. 
More generally we should note the following corollary holds.

\begin{corollary}
Let  $X, Y \in \K$. 
Then  $\underline{X} \cong \underline{Y}$  in $\uK$ 
if and only if  $X \oplus F \cong  Y \oplus F' $ in $\K$  for some  $F, F' \in \F$. 
\end{corollary}

\begin{proof}
If $\underline{g} : \underline{X} \to \underline{Y}$ is an isomorphism whose inverse morphism is  $\underline{h}$, then 
it follows from Lemma \ref{commutativity} that  $X$  is a direct summand of $Y \oplus F$  in $\K$ for some  $F \in \F$. 
Therefore  there exists an isomorphism $Y \oplus F \to X\oplus F'$ in which the restricted map  $Y \to X$  is given by  $h$. 
We have to show that  $F' \in \F$. 
Since  $\underline{F}=0$, we have an isomorphism  $\underline{Y} \to \underline{X} \oplus \underline{F'}$  in which  $\underline{h} : \underline{Y} \to \underline{X} $ is also an isomorphism. 
Taking  the inverse of the isomorphism 
$$
\begin{pmatrix} \underline{h} \\ \underline{a}  \end{pmatrix} :  \underline{Y} \to \underline{X} \oplus \underline{F'}, 
$$
we have 
$$
\begin{pmatrix} \underline{b} & \underline{c}  \end{pmatrix} :   \underline{X} \oplus \underline{F'} \to\underline{Y} 
$$
such that 
$$
\begin{pmatrix} \underline{h}\underline{b} & \underline{h}\underline{c} \\
 \underline{a}\underline{b} & \underline{a}\underline{c} \\
 \end{pmatrix}
 = \begin{pmatrix} 1_{\underline{X} } & 0 \\ 0 & 1_{\underline{F'}}  \end{pmatrix}. 
$$
Since $\underline{h}\underline{b} = 1_{\underline{X} }$ where $\underline{h}$ is an isomorphism, $\underline{b} $ is also an isomorphism. 
Then, since $\underline{a}\underline{b} =0$, we have $\underline{a}=0$. 
Therefore $1_{\underline{F'}}  =  \underline{a}\underline{c} =0$, and thus $\underline{F'} \cong 0$,
hence  $F' \in \F$. 
\end{proof}

\begin{remark}
Recall from  Theorem \ref{characterization of F} that   $X \in \K$  belongs to  $\F$  if and only if  $X$  is a split complex. 
Hence,  setting  $\mathcal{S}$  to be the full subcategory of  $\C$  that consists of all split complexes, we can also describe the stable category as  $\underline{\K} = \C / \mathcal{S}$.  
Therefore one can also prove that 
$\underline{X} \cong \underline{Y}$  in $\uK$ 
if and only if  $X \oplus T \cong  Y \oplus T' $ in $\C$  for some  $T, T' \in \mathcal{S}$. 
\end{remark}

\begin{definition}
Let  $f :  X\to Y$   be a morphism in $\K$. 
We say that  $f$  is {\bf cohomologically surjective} if  the cohomology mapping $H(f): H(X)  \to  H(Y)$ is surjective. 

We also define the complex  $Cone (f)  \in \K$  by the triangle 
$$\xymatrix@C=36pt{ Cone (f) [-1] \ar[r] & X \ar[r]^f  &  Y \ar[r] & Cone (f) }$$ 
in $\K$, which is actually the mapping cone of the chain map  $f$. 
\end{definition}

In general, for given morphisms  $f, g : X \to Y$ in $\K$,  that $\underline{f}= \underline{g}$ in $\uK$  does not mean $\underline{Cone (f)}\cong \underline{Cone (g)}$  in $\uK$. 

{For an example of this, let $R$ be a local ring and $x\in R$ a non-zero divisor. 
Set $X = Y = R$ and consider the morphisms  $f, g \in \Hom _{\K} (X, Y)$ defined by 
$f(a) =xa, g(a)=0$  for $a \in X=R$. 
Since  $R$ is a split complex as an object of $\K$, we have $\underline{f} = \underline{g}= 0$. 
In this case, $Cone (f) = [ \xymatrix{0 \ar[r] & R \ar[r]^x & R \ar[r] & 0}]$ is the Koszul complex and $\underline{Cone (f)} \not\cong 0$, while 
 $Cone (g) = [ \xymatrix{0 \ar[r] & R \ar[r]^0 & R \ar[r] & 0}]$ is a split complex hence 
 $\underline{Cone (g)} \cong 0$. 
Note in this example that $f$ and $g$ are not cohomologically surjective. }

\begin{theorem}\label{cone}
Let $f : X \to Y$  and  $f' : X' \to Y$  be morphisms in  $\K$. 
Assume that both $f$  and  $f'$  are cohomologically surjective. 
Further assume that  $\underline{X} \cong \underline{X'}$ in $\uK$   and that  $\underline{f}$ corresponds to  $\underline{f'}$ under the isomorphism 
$\Hom _{\uK}(\underline{X}, \underline{Y}) \cong   \Hom _{\uK}(\underline{X'}, \underline{Y})$. 
Then we have an isomorphism  $\underline{Cone (f)} \cong \underline{Cone ({f'})}$  in $\uK$. 
\end{theorem}

\begin{proof}
As the first step of the proof we prove the following isomorphism: 
\begin{equation}\label{claim}
\underline{Cone (f)} \cong  \underline{Cone (f \ a)} \quad \text{for any }\ F \in \F \ \text{and}  \ (f \ a) : X \oplus F \to Y.   \end{equation}
In fact, there is a commutative diagram in $\K$ whose rows and columns are triangles: 
$$
\xymatrix@C=32pt@R=32pt{
& F[-1]  \ar[d]^u \ar@{=}[r]  & F[-1] \ar[d]^0 & \\
Y[-1] \ar@{=}[d] \ar[r] & Cone (f) [-1] \ar[d] \ar[r]^(0.6)v & X \ar[d]^{\binom{1}{0}} \ar[r]^f &Y \ar@{=}[d] \\
Y[-1] \ar[r] & Cone (f \  a) [-1] \ar[d] \ar[r] & X \oplus F  \ar[d]^{(0\ 1)} \ar[r]^(0.6){(f \  a)} & Y \\
& F \ar@{=}[r] & F. & \\
}
$$
Since  $H(f)$ is surjective, note in this diagram  that  $H(v)$  is injective. 
Then that  $vu =0$, and hence  $H(v)H(u)=0$, forces $H(u)=0$. 
Thus by Theorem \ref{characterization of F} we have  $u=0$,   which shows  an isomorphism 
$Cone(f \  a) [-1] \cong Cone (f) [-1]  \oplus F$, and hence  (\ref{claim})  is proved. 

As the second step of the proof, we prove the theorem in the case of  $X = X'$. 
In this case we have  $f' = f + ab$  for $a : F \to Y$  and  $b : X \to F$  with  $F \in \F$, by virtue of {Lemma \ref{commutativity}}.  
Then there is a commutative diagram in $\K$ 
$$
\xymatrix@C=48pt@R=36pt{
X \oplus F \ar[d]_{
{\tiny \begin{pmatrix}1& 0 \\ b & 1\\ \end{pmatrix}}
} \ar[r]^{(f' \ a)} &  Y \ar@{=}[d] \\
X \oplus F \ar[r]^{(f \ a)} & Y. \\  
}
$$ 
Since the left vertical arrow is an isomorphism,  we have $Cone (f \ a) \cong  Cone (f' \ a)$ in $\K$, hence  
$\underline{Cone (f) } \cong \underline{Cone(f ')}$  by using (\ref{claim}). 

Now consider the general case of the theorem. 
Since  $\underline{X} \cong \underline{X'}$, there is an isomorphism  $g : X \oplus F \to X' \oplus F'$  for some  $F, F' \in \F$, and by the assumption we must have  $\underline{f} = \underline{f'} \cdot \underline{g}$. 
 Consider the morphisms  $(f \ 0) : X \oplus F  \to  Y$ and $(f' \ 0) : X' \oplus F'  \to  Y$, and we note that  they are cohomologically surjective. 
 On the other hand, since  $\underline{F}=\underline{F'}=0$,  we have equalities
 $$
 \underline{(f \ \  0)} =  \underline{f}= \underline{f'}\cdot \underline{g} =  \underline{(f' \ \ 0)\cdot g}. 
 $$ 
Thus it follows from the second step of this proof that  
$\underline{Cone (f \  0)} \cong   \underline{Cone ({(f' \ 0)\cdot g})}$. 
Note here that  $Cone {((f'\ 0) \cdot g) }\cong Cone (f' \ 0)$  in  $\K$, since  $g$  is an isomorphism in $\K$.    
Hence the isomorphism  $\underline{Cone (f)} \cong   \underline{Cone (f')}$  follows from (\ref{claim}).    
\end{proof}

\vspace{6pt}
\section{$\F$-approximations}

We are able to show that  the subcategory  $\F$  of  $\K$  is functorially finite in the sense of Auslander. 
(Cf.  Auslander \cite{A}.) 
For this we begin with recalling the definition of right approximations.  

\begin{definition}
Let  $X  \in \K $. 
A morphism  $p : F \to  X$ in  $\K $   is called a {\bf right $\F$-approximation} of  $X$  if 
 $F \in \F $  and  $\Hom_{\K }(G, p) : \Hom _{\K}(G, F) \to \Hom _{\K }(G, X)$    is surjective for any  $G \in \F $. 
\end{definition}

We should remark that the shift functor preserves the right $\F$-approximation property, i.e. $p : F \to X$  is a right $\F$-approximation if and only if so is $p[n] : F[n] \to X[n]$ for any $n \in \Z$.

\begin{lemma}\label{right approximation}
Let  $X \in \K$ and  $F \in \F$. 
Then   a morphism  $p : F \to  X$ in  $\K$  is a right  $\F$-approximation if and only if 
$p$  is cohomologically surjective. 
Moreover,  there  always exists a right $\F$-approximation of  $X$  for any  $X \in \K$. 
\end{lemma}

\begin{proof}
If  $p : F \to X$  is a right $\F$-approximation then 
$H^i(p)= \Hom _{\K}(R[-i], p)$ is surjective, since  $R[-i] \in \F$  for $i \in \Z$.

Conversely assume that  $H(p)$  is surjective, and let  $g : G  \to X$  be a morphism in $\K$ with  $G \in \F$. 
Then  $H(g) : H( G ) \to  H(X)$  factors through  $H(p)$, since  $H(G)$  is a graded projective $R$-module : 
$$
\xymatrix{ 
H(G) \hspace{6pt} \ar[d]_{\alpha} \ar[dr]^{H(g)}  &    \\
H(F)  \ar[r]^(0.45){H(p)}  &  H(X) \\
}
$$
Then,  by Theorem \ref{characterization of F},  there is a morphism  $a : G \to F$  such that  $H(a) = \alpha$, and since  $H(g) = H(pa)$, we have  $g = pa$. 

For the existence of right $\F$-approximation of  $X$, one has only to take a graded projective  $R$-module  $F$  which maps surjectively onto  $H(X)$.  
Then it follows from Theorem \ref{characterization of F} that this mapping is lifted to a chain homomorphism  $F \to X$ which is in fact a right $\F$-approximation of $X$. 
\end{proof}

If $p : F \to X$  is a right $\F$-approximation, then as we have shown in Theorem \ref{cone}, the mapping cone  $Cone (p)$  is uniquely determined as an object of $\uK$.

\begin{definition}\label{omega}
Let  $X \in \K$ and $p : F \to X$  be a right $\F$-approximation of  $X$.
We define $\Omega (X)$ (or simply denoted $\Omega X$) by the equality 
$$
\Omega (X) = Cone (p)[-1], 
$$
which is uniquely determined in the stable category 
$\uK$  by Theorem \ref{cone}.
Actually, $\Omega$  yields a functor  $\uK  \to \uK $ as follows: 

Let  $a : X \to Y$  be a morphism in $\K $.  
If $p_X : F_{X} \to X$  and  $p_Y : F_{Y}\to Y$  are right $\F $-approximations, then,  since  $ap_X$ factors through $p_Y$, we have the following commutative diagram, and as a result the morphism  $b : \Omega (X) \to \Omega (Y)$ is induced. 
$$
\xymatrix@C=48pt@R=36pt{ 
\Omega (X)  \ar[r]^{q_X} \ar[d]_b &  F_{X} \ar@{-->}[dl]^e \ar[r]^{p_X} \ar[d] & X \ar@{-->}[dl]^c \ar[r]^{\omega _X} \ar[d]_a &  \Omega (X) [1] \ar[d]^{b[1]}  \\ 
\Omega (Y) \ar[r]^{q_Y}  &  F_{Y}  \ar[r]^{p_Y}  &  Y  \ar[r]^{\omega _Y} & \Omega (Y) [1] \\
}
$$
If  $a$ factors through an object in $\F$, then it factors through $p_Y$, i.e. there is $c : X \to F_Y$ such that  $a = p_Y c$. 
Then we have  $b[1]\omega _X = \omega _Ya=\omega _Yp_Yc=0$, hence there is a morphism $e : F_X \to \Omega (Y)$  with  $b[1] = e[1]q_X[1]$ or $b=eq_X$. 
Thus $b$ factors through an object in $\F$. 
In such a way we see that the mapping 
$$
\Hom  _{\uK} (X,  Y)  \to   \Hom _{\uK} ( \Omega (X), \Omega (Y))
\quad ; \ \underline{a} \ \mapsto \underline{b} 
$$
is well-defined, hence we can define  $\Omega (\underline{a}) = \underline{b}$ for morphisms. 
We call $\Omega$ the {\bf syzygy functor} on  $\uK$. 
\end{definition}

\begin{definition}
Let  $X  \in \K $. 
A morphism  $q  : X  \to  G$ in  $\K $   is called a {\bf left  $\F$-approximation} of  $X$  if 
 $G \in \F $  and  $\Hom_{\K}(q, F) : \Hom _{\K}(G, F) \to \Hom _{\K}( X, F)$   are surjective mappings for all  $F \in \F $. 
\end{definition}

Recall that the dual complex  $X^{*} = \Hom _{R }(X, R)$  is again a complex belonging to  $\K$  and $X^{**} \cong X$. 
Also note that $X^* \in \F$  if and only if  $X \in \F $.

\begin{lemma}\label{left approximation}
Let  $X \in \K$ and  $G \in \F $. 
Then   a morphism  $q : X \to  G$ in  $\K$  is a left  $\F $-approximation 
if and only if 
the dual  $q^{*} : G^{*} \to X^{*}$  is a right $\F $-approximation, 
and the latter  is equivalent to that $q^{*}$ is cohomologically surjective. 
In particular,  there  exists a left  $\F $-approximation of  $X$  for any  $X \in \K$. 
\end{lemma}

\begin{proof}
Assume that  $q^{*} : G^{*} \to X^{*}$  is a right $\F $-approximation. 
Let  $a : X \to F$ be a morphism in  $\K$  with  $F \in \F $. 
Then  $a ^{*} : F^{*} \to X^{*}$  is a morphism in  $\K$, hence it factors through  $q^{*}$. 
As a result, $a$ factors through  $q$, hence  $q$  is a left $\F $-approximation. The converse is proved similarly.  
\end{proof}

\begin{remark}\label{remark approximation}
Suppose we have a commutative diagram in $\K$ with  $F, F' \in \F$:  
$$
\xymatrix{ 
F \ar[r]^{p} \ar[d]_f &X    \\
F'  \ar[ur]_{p'}  &  \\
}
$$
In such a case if  $p$  is right $\F$-approximation, then so is $p'$. 
In fact if  $H(p)=H(p')H(f)$ is surjective, then so is $H(p')$. 

Similarly if a diagram 
$$
\xymatrix{ 
X \ar[r]^{q} \ar[dr]_{q'} & G    \\
& G'  \ar[u]  \\
}
$$
is commutative with  $G, G' \in \F$ and if  $q$ is a left $\F$-approximation, then $q'$ is a left $\F$-approximation as well.   
\end{remark}

\begin{corollary}\label{Gor dim 0}
Assume that  $R$  is a Gorenstein ring of dimension zero, and  let 
$$
\xymatrix{Y \ar[r]^q & F \ar[r]^p  & X \ar[r] & Y[1]}, 
$$
be a triangle in $\K$ where  $F \in \F$. 
Then,  $p$  is a right $\F$-approximation if and only if  $q$  is a left $\F$-approximation. 
\end{corollary}

\begin{proof}
If $p$ is a right $\F$-approximation then $H(p)$  is a surjective $R$-module homomorphism. 
Then $H(p)^*$  is an injective homomorphism. 
Since $R$ itself is an injective $R$-module, noting that the equality  $H(p ^*) = H(p)^*$ holds, we see from 
the triangle  $\xymatrix{X^* \ar[r]^{p^*} & F^* \ar[r]^{q^*}  & Y^* \ar[r] & X^*[1]}$  that  $H(q^*)$  is surjective. 
Hence $q$  is a left $\F$-approximation by  Lemma \ref{left approximation}. 
The converse is proved in a similar manner. 
\end{proof}

\begin{definition}\label{omega inverse}
Let  $X \in \ \K$ and $q  : X  \to G$  be a left  $\F$-approximation of  $X$.
Embed  $q$ into a triangle  $\xymatrix{X \ar[r]^{q} & G \ar[r] & Z \ar[r]  & X[1]}$. 
We denote  the resulted $Z$  by $\Sigma (X)$ (or simply  $\Sigma X$). 
It follows from Lemma \ref{left approximation} that  
$$
\Sigma (X) = \Omega (X^{*})^{*},
$$  
which is uniquely determined as an object in  the stable  category  $\uK$. 
Actually, $\Sigma$  yields a well-defined functor  $\uK \to \uK$  in a similar manner to  the case of $\Omega$. 
We call $\Sigma$ the {\bf cosyzygy functor} on $\uK$. 
\end{definition}

\begin{remark}\label{remark on omega}
If  $R$  is a Gorenstein ring of dimension zero, then Corollary \ref{Gor dim 0} says that  
$\Sigma$ is actually the {inverse} of $\Omega$ as a functor on $\uK$.  
\end{remark}

\begin{proposition}\label{remark on omega by localization}
Let $S$ be a multiplicative closed subset of  $R$. 
Then the functor  $S^{-1} : \K \to \Ko (S^{-1}R)$  is defined naturally by taking the localization by  $S$. 
\begin{enumerate}
\item
If  $p : F \to X$  is a right $\F$-approximation in $\K$, then $S^{-1}p : S^{-1}F \to S^{-1}X$ is a right $\mathrm{Add}(S^{-1}R)$-approximation in $\Ko (S^{-1}R)$. \vspace{6pt} 
\item
If  $q : Y \to G$ is a left $\F$-approximation in $\K$, then $S^{-1}q : S^{-1}Y \to S^{-1}G$ is a left  $\mathrm{Add}(S^{-1}R)$-approximation in $\Ko (S^{-1}R)$. \vspace{6pt}
\item 
Let  $\Omega_{S^{-1}R}$  and  $\Sigma _{S^{-1}R}$  be the syzygy and cosyzygy functors on $\underline{\Ko(S^{-1}R)}$. 
Then the following squares are commutative: 
$$
\xymatrix{
\uK \ar[r]^{\Omega} \ar[d]_{S^{-1}} & \uK \ar[d]^{S^{-1}}  \quad  & \uK \ar[r]^{\Sigma} \ar[d]_{S^{-1}} & \uK \ar[d]^{S^{-1}}  \\
\underline{\Ko (S^{-1}R)} \ar[r]^(0.45){\Omega_{S^{-1}R} } & \underline{\Ko (S^{-1}R)}, \quad & 
\underline{\Ko (S^{-1}R)} \ar[r]^{\Sigma_{S^{-1}R} } & \underline{\Ko (S^{-1}R)} .  
}$$
\end{enumerate}
\end{proposition}

\begin{proof}
(1) If  $p$  is cohomologically surjective, then so is  $S^{-1}p$. 
\par\noindent 
(2) is clear from the fact that  $S^{-1} \Hom _R(q, R) \cong  \Hom _{S^{-1}R} (S^{-1}q, S^{-1}R)$  and  Lemma \ref{left approximation}. 
\par\noindent
(3) follows from (1) and (2).  
\end{proof}

In general,  $\Sigma $  is not necessarily the quasi-inverse of  $\Omega$, but we see that $\Sigma$  is a left adjoint to $\Omega$.

\begin{theorem}\label{adjoint}
As functors  from  $\uK$  to itself,  $(\Sigma, \Omega)$  is an adjoint pair, i.e. 
there are functorial isomorphisms
$$
\Hom  _{\uK} (\Sigma X,  Y)  \cong  \Hom _{\uK} ( X, \Omega Y ),
 $$
 for all $X, Y \in \uK $.
\end{theorem}

\begin{proof}
To prove the theorem let  $a : \Sigma X \to Y$  be a morphism in $\K$. 
 Then it induces the following commutative diagram:  
$$
\xymatrix@C=48pt@R=36pt{ 
X \ar[d]_b \ar[r]^q  &  G_{X} \ar[d] \ar[r] &  \Sigma X \ar[d]_a \ar[r]  & X[1] \ar[d] \\ 
\Omega Y \ar[r] &  F_{Y} \ar[r]^p &  Y  \ar[r]  & \Omega Y [1],  
}
$$
where  $p$  is a right $\F$-approximation and  $q$ is a left $\F$-approximation. 
    Then, by the same reason in Definition \ref{omega} above, we see that 
$$
\Hom  _{\uK} (\Sigma X,  Y)  \to  \Hom _{\uK} ( X, \Omega Y)\quad ; \ \underline{a} \ \mapsto \underline{b} 
$$
is well-defined. 
Conversely, given a morphism $b  : X \to \Omega Y$, one can easily find an $a : \Sigma  X \to  Y$ that makes the diagram commutative. 
It thus gives the inverse to the above mapping:  
$$
\Hom _{\uK} ( X, \Omega Y) \to \Hom_{\uK} (\Sigma  X,  Y) \quad ; \ \underline{b} \ \mapsto \underline{a} 
$$
\end{proof}

We should  notice that the similar arguments to ours in this section and also the similar content to Theorem \ref{characterization of F} can be found  in those papers of J.D.Christensen \cite[Proposition 8.1]{C}  and Krause-Kussin \cite[Lemma 2.5]{KK}.

\begin{example}\label{exomega}
Let  $M$  be a finitely generated $R$-module and let  
$$
\xymatrix{
\cdots \ar[r] & P_n \ar[r]^(0.4){u_n} & P_{n-1} \ar[r] & \cdots \ar[r] & P_1 \ar[r]^{u_1} & P_0 \ar[r]^{u_0} & M \ar[r] &  0 \\
}
$$
be a projective resolution of $M$ where each $P_i$ are finitely generated. 
Now, set the complex $X$ to be $\left[ \xymatrix{0 \ar[r] & P_1 \ar[r]^{u_1} & P_{0} \ar[r] & 0}\right]$. 
Then the right $Add (R)$-approximation of $X$ is given as 
$$
\xymatrix{F_X \ar[d] _{p_X} &{=}  &[ 0 \ar[r] & P_2 \ar[d]^{u_2}  \ar[r]^0 & P_0 \ar[d]^1  \ar[r] & 0 ] \\
 X &{=}  &[0 \ar[r] & P_1 \ar[r] ^{u_1} & P_0 \ar[r] & 0]. }  
$$
Hence we have 
$$
\Omega X = \left[ \xymatrix{0 \ar[r] & P_{2} \ar[r]^{u_{2}} & P_{1} \ar[r] & 0}\right]. 
$$
More generally we have 
$$
\Omega ^n X (= \Omega (\Omega^{n-1}X)) = \left[ \xymatrix{0 \ar[r] & P_{n+1} \ar[r]^{u_{n+1}} & P_{n} \ar[r] & 0}\right],  
$$
for $n > 0$. 
On the other hand let  
$$
\xymatrix{
\cdots \ar[r] & Q_n \ar[r]^(0.4){v_n} & Q_{n-1} \ar[r] & \cdots \ar[r] & Q_1 \ar[r]^{v_1} & Q_0 \ar[r]^{v_0} & M^*  \ar[r] &  0 \\
}
$$
be a projective resolution of $M^*$. 
Then one can see that 
$$
\Sigma  X =\left[ \xymatrix{0 \ar[r] & P_{0} \ar[r]^(0.45){w} & Q_{0}^* \ar[r] & 0}\right],  
$$
where $w$ is the composition  
$\xymatrix@C=32pt{P_0 \ar[r]^{u_0} & M \ar[r]^(0.4){\text{natural}} & M^{**} \ar[r]^{v_0^*} & Q_0^*}$. 
For $n>1$, it can be easily seen that 
$$
\Sigma ^{n} X = \left[ \xymatrix{0 \ar[r] & Q_{n-2}^* \ar[r]^{v_{n-1}^*} & Q_{n-1}^* \ar[r] & 0}\right].  
$$
See also Example \ref{ex9.9}. 
\end{example}

\begin{lemma}\label{torsion-free}
Let   $X \in \K$. 
Then   $\Sigma X$  is *torsion-free as an object in $\K$.   
\end{lemma}

\begin{proof}
Suppose we are given  $f \in H^{i}((\Sigma X)^{*})= \Hom _{\K}(\Sigma X, R[i])$  for some  $i \in \Z$  such that  $H(f)=0$. 
We want to show $f =0$. 

There is a triangle;  
$$
\xymatrix@C=36pt{
 G_{X} \ar[r]^(0.4)a & \Sigma  X \ar[d]_f  \ar[r]^b & X[1] \ar[r]^q & G_{X}[1] \\ 
&  R[i],  && }
$$
where $q$ is a left $\F$-approximation. 
Since  $H(fa)=H(f)H(a)=0$, we have  $fa=0$  by Theorem \ref{characterization of F}. 
Therefore there is a morphism $c : X[1] \to R[i]$  such that  $f=cb$. 
Since  $q$  is a left $\F$-approximation and  since  $R[i] \in \F$, 
we have  $c = eq$  for some  $e : G_{X}[1] \to R[i]$. 
Thus  $f = cb = eqb=0$ as desired. 
\end{proof}

\begin{theorem}\label{torsion-free2}
The following conditions are equivalent for  $X \in \K$. 
\begin{enumerate}
\item
$X$  is *torsion-free. \medskip
\item
There are complexes $Y \in \K$ and $F \in \F$  such that $X$  is a direct summand of  $\Sigma Y \oplus F$  in  $\K$.  \medskip 
\item
There is an isomorphism ${X}  \cong  { \Sigma \Omega X}$  in $\uK$. 
\end{enumerate}
\end{theorem}

\begin{proof}
The implication $(2) \Rightarrow (1)$ follows from Lemma \ref{torsion-free}, 
since direct sums (or direct summands) of *torsion-free complexes are *torsion-free as well. 

 The condition (3) means exactly that  $X \oplus F_{1} \cong \Sigma \Omega X \oplus F_{2}$  in  $\K$  for some  $F_{1}, F_{2} \in \F$. 
 Hence  (3)  implies  (2). 
 
It remains to prove $(1) \Rightarrow (3)$. 
 Let  
 $$
 \xymatrix@C=36pt{ 
 \Omega X  \ar[r]^a  &  F  \ar[r]^p &  X \ar[r] & \Omega X[1]  \\ 
 \Omega X  \ar@{=}[u] \ar[r]^q  &  G   \ar[u]^{f}  \ar[r] &  \Sigma \Omega X \ar[r] \ar[u]^{\pi}  & \Omega X[1]  \ar@{=}[u] \\ }
$$ 
be  a commutative diagram in  $\K$  whose rows are triangles in $\K$,  where  $p$ (resp. $q$)  is a right (resp. left) $\F$-approximation and the morphism $f$  is induced by the definition of  left $\F$-approximations. 
Note in this diagram that we can take such diagram in such a way that  $H(f)$  is a surjective graded $R$-module homomorphism. 
In fact, if necessary, we may replace  $q : \Omega X \longrightarrow G$  by  
$\binom{q}{a} : \Omega X \longrightarrow  G \oplus F$. 
Thus we may assume that  $L := Cone (f) [-1]$  belongs to $\F$. 
Note also that  $\pi$  is cohomologically surjective, since both  $f$ and  $p$  are so.  

Since there is a commutative diagram by the octahedron axiom:
$$
\xymatrix{
 & L[1] \ar@{=}[r] & L[1]  \\
\Omega X \ar[r]^a  & F \ar[r] ^p \ar[u]& X \ar[r] \ar[u]& \Omega X [1]  \\ 
\Omega X  \ar@{=}[u] \ar[r]^q  & G  \ar[u]_f \ar[r] ^p & \Sigma\Omega X \ar[u]_{\pi} \ar[r] & \Omega X [1]  \ar@{=}[u]\\ 
 & L \ar[u]  \ar@{=}[r] & L,  \ar[u]  \\
} 
$$
we have the following triangle in $\K$:  
 $$
 \xymatrix@C=36pt{ 
L \ar[r] & \Sigma \Omega X  \ar[r]^(0.6){\pi}  &  X  \ar[r]^(0.45)b  &  L[1].  \\}
$$ 
Since  $H(\pi)$  is surjective, we see that  $H(b)=0$  by the cohomology long exact sequence. 
Then, since we are assuming that $X$  is *torsion-free,  it follows from Proposition \ref{injection proposition} that  $b=0$  as a morphism in $\K$. 
Thus the triangle splits and we have the isomorphism  $\Sigma \Omega X \cong  X \oplus L$  in $\K$  with $L \in \F$. 
\end{proof}

\begin{remark}\label{torsion-free precover}
By the adjoint property proved in Theorem  \ref{adjoint}, there is a natural  counit morphism 
$\pi : \Sigma \Omega X \to X$  for any $X \in \uK$. 
Theorem \ref{torsion-free2} says that this is actually a right *torsion-free approximation of $X$ in the following sense: 
If $\underline{f} : \underline{Y} \to \underline{X}$ is a morphism in $\underline{\K}$ where $Y$ is *torsion-free, then $\underline{f}$ factors through $\pi$.

In fact, $\pi_Y$ in the following commutative diagram in $\underline{\K}$ is an isomorphism: 
$$
\xymatrix{
\Sigma\Omega Y \ar[d]_{\pi_Y} \ar[r]^{\Sigma\Omega f} & \Sigma\Omega X \ar[d]_{\pi}  \\
Y \ar[r]^{f} & X}
$$
\end{remark}

\begin{lemma}\label{reflexive}
Suppose that  $R$  is a generically Gorenstein ring. 
If  $X \in \K$  is *torsion-free, then  $\Sigma X$ is *reflexive. 
\end{lemma}

\begin{proof}
We have shown in Lemma \ref{torsion-free} that  $\Sigma  X$ is *torsion-free. 
To prove that it is *reflexive, let  $\alpha : H^n(\Sigma  X) \to R$ be a  homomorphism of $R$-modules, where $n \in \Z$. 
We want to show that there is a morphism $a : \Sigma X \to R[-n]$  in $\K$  satisfying  $H(a) = \alpha$. 
By definition, there is a triangle 
$$
 \xymatrix@C=36pt{ 
 X \ar[r] ^(0.4){q} &  G_X \ar[r]^(0.4)p &  \Sigma X \ar[r]^r & X[1],  \\}
$$
where $q$  is a left  $\F$-approximation. 
Therefore we have a long exact sequence of cohomology modules; 
$$
\xymatrix@C=36pt{ 
H^n(X) \ar[r]^(0.45){H(q)}  & H^n(G_X) \ar[r]^(0.45){H(p)}  &  H^n(\Sigma X)  \ar[d]_{\alpha}  \ar[r]^{H(r)}  & H^{n+1}(X)  \\
& & R &  \\}
$$
Note that  $G_X \in \F$  is *reflexive, and thus there is a morphism  $b : G_X \to R[-n]$  such that  $H(b) = \alpha H (p)$.
Then it is clear that $H(bq) = \alpha H(p)H(q)=0$. 
Since $X$  is *torsion-free, it follows that  $bq=0$. 
{(See Lemma \ref{restatement}.)}
Thus there is a morphism  $a : \Sigma X \to R[-n]$  with  $b = ap$. 
Note that  $(\alpha -H(a))H(p) = 0$. 
Let  $S$  be the set of all non-zero divisors of $R$. 
Since  $S^{-1}R$  is a Gorenstein ring of dimension zero, 
$$
\xymatrix@C=36pt{ 
S^{-1}X \ar[r]^(0.45){S^{-1}q} & S^{-1}G_X \ar[r]^(0.4){S^{-1}p}  &  S^{-1}\Sigma  X \ar[r]^{S^{-1}r}  & S^{-1}X[1]  \\}
$$
is a triangle in which $S^{-1}q$  is a left $\mathrm{Add}(S^{-1}R)$-approximation and 
$S^{-1}p$ is a right approximation in $\mathscr K (S^{-1}R)$ by Corollary \ref{Gor dim 0}.  
In particular, $S^{-1}H(p) = H(S^{-1}p)$  is a surjective mapping by Lemma \ref{right approximation}. 
Since   $(\alpha -H(a))H(p) = 0$, we see that  $S^{-1} ( \alpha -H(a)) = 0$  as  an element  of  $S^{-1}\left( H(\Sigma X)^*\right)$. 
Noting that the $R$-dual of any finitely generated module is torsion-free,  we  see that  $H(\Sigma X)^*$  is a torsion-free $R$-module. 
Consequently we have that  $\alpha = H(a)$  as  an element  of  $H^n(\Sigma (X))^*$. 
\end{proof}

Combining this lemma with Theorem \ref{torsion-free2} or with Lemma \ref{torsion-free} we obtain the following theorem. 

\begin{theorem}\label{reflexive cor}
Under the assumption that $R$ is generically Gorenstein,  $\Sigma ^{2}X =\Sigma (\Sigma  X)$  is always *reflexive for any  $X \in \K$. 
\end{theorem}

Similarly to Theorem \ref{torsion-free2} one can characterize the *reflexivity property for complexes as follows:  

\begin{corollary}
Assume that $R$ is a generically Gorenstein ring. 
Then the following two conditions for  $X \in \K$ are equivalent: 
\begin{enumerate}
\item 
$X$  is *reflexive.  \vspace{6pt}
\item
$\Sigma ^{2}\Omega^2 X \cong X $  in $\uK$. 
\end{enumerate}
\end{corollary}

\begin{proof}
Theorem \ref{reflexive cor} says that $(2) \Rightarrow (1)$ holds. 
Assume  $X$  is *reflexive. 
Then take a right $\F$-approximation sequence  
$\xymatrix@C=28pt{ \Omega X \ar[r] &F_0 \ar[r]& X \ar[r] &\Omega X[1]}$. 
Since right $\F$-approximations are cohomologically surjective, we have an exact sequence of $R$-modules: 
$$
\xymatrix{0 \ar[r] &H(X)^* \ar[r] &H(F_0)^* \ar[r] &H (\Omega X)^* }
$$
Thus we can apply Proposition \ref{triangle argument}(2), and we see that $\Omega X$  is *torsion-free. 
Then it follows from Theorem \ref{torsion-free2} that $\Sigma \Omega ^2 X = \Sigma \Omega (\Omega X) \cong \Omega X$. 
Thus applying $\Sigma $ to the both sides, we have  $\Sigma ^{2}\Omega ^2 X \cong \Sigma \Omega X$ and the last equals $X$, since $X$ is *torsion-free.
\end{proof}

\vspace{6pt}
\section{Contractions}

\begin{definition}\label{def Kexact}
We say that a finite sequence of morphisms in $\K$;  
\begin{equation}\label{Kexact1}
\xymatrix@C=32pt{ 
0 \ar[r] & X_n  \ar[r]^{q_n} & F_{n-1} \ar[r]^{f_{n-1}} & F_{n-2} \ar[r] & \cdots \ar[r] & F_1 \ar[r]^{f_1} & F_0 \ar[r]^{p_0} & X_0 \ar[r] & 0}  
\end{equation}
is $\K${\bf -exact} if  there are triangles 
$$
\xymatrix@C=32pt{ 
X_{i+1} \ar[r]^{q_{i+1}}  &  F_{i}  \ar[r]^{p_{i}}  &  X_{i}  \ar[r]^(0.4){\omega_i} &  X_{i+1}[1] \\
}
$$
and equalities 
$$
f_i= q_i  \, p_i
$$
for $0 \leq i \leq n-1$. 
The $\K$-exact sequence (\ref{Kexact1}) can be described in a single diagram as 
$$
\xymatrix@C=32pt{
 & & X_{n-1} \ar[dr]^{q _{n-1}}  & X_{n-2} \ar[dr] & & X_1 \ar[dr]^{q_1} & & & \\ 
0 \ar[r] & X_n  \ar[r]^{q_n} & F_{n-1} \ar[u]_{p_{n-1}} \ar[r]^{f_{n-1}} & F_{n-2} \ar[u]_{p_{n-2}} \ar[r] & \cdots \ar[r] & F_1 \ar[u] _{p_1} \ar[r]^{f_1} & F_0 \ar[r]^{p_0} & X_0 \ar[r] & 0.}  
$$

We also call the $\K$-exact sequence (\ref{Kexact1})   a {\bf  partial $\F$-resolution of}  $X_0$  if  $F_i \in \F$  for all $0 \leq i < n$.   
By definition an {\bf $\F$-resolution of  $X_0$ of length $n-1$}  is a partial $\F$-resolution with $X_n = 0$. 
\end{definition}

Note that, in the paper \cite[Notation 3.2]{IY},  we call a $\K$-exact sequence an $(n+1)$-angle in $\K$.
However in the present paper we are interested only in the \lq exactness'  of the sequence, and not in the length $n$. 
For this reason we use the term \lq$\K$-exact sequence' instead of $(n+1)$-angle.

If we are given such a $\K$-exact sequence (\ref{Kexact1}), we have a natural morphism $\widetilde{\omega_n} : X_0 \to X_n[n]$ that is defined by the composition 
$\omega _{n-1}[n-1]  \omega _{n-2}[n-2] \cdots \omega _1[1]  \omega_0$ of the morphisms in the relevant triangles. 
Notice that the morphism $\widetilde{\omega_n} : X_0 \to X_n[n]$  is {determined by  $\omega _i$  in (\ref{Kexact1}), and hence it is uniquely determined by the collection of triangles in (\ref{Kexact1})}, which we call the {\bf connecting morphism} of the $\K$-exact sequence (\ref{Kexact1}).

\begin{theorem/definition}\label{theorem/definition}
Suppose we are given a partial $\F$-resolution; 
\begin{equation}\label{Kexact2}
\xymatrix@C=32pt{
 & & X_{n-1} \ar[dr]^{q _{n-1}}  & X_{n-2} \ar[dr] & & X_1 \ar[dr]^{q_1} & & & \\ 
0 \ar[r] & X_n  \ar[r]^{q_n} & F_{n-1} \ar[u]_{p_{n-1}} \ar[r]^{f_{n-1}} & F_{n-2} \ar[u]_{p_{n-2}} \ar[r] & \cdots \ar[r] & F_1 \ar[u] _{p_1} \ar[r]^{f_1} & F_0 \ar[r]^{p_0} & X_0 \ar[r] & 0.}  
\end{equation}
Furthermore we assume that each  $F_i  \in \F$  contains no null complex as a direct summand for $0 \leq i \leq n-1$. 
Then there is a triangle of the form
\begin{equation}\label{contraction sequence}
\xymatrix@C=32pt{
X_n[n-1] \ar[r]^(0.6){\psi _n} & \widetilde{F} \ar[r]^{\varphi_n} & X_0 \ar[r]^(0.4){\widetilde{\omega_n}} & X_n [n],  
}
\end{equation}
where  $\widetilde{\omega_n}$  is the connecting morphism of the sequence  (\ref{Kexact2}) and the following conditions are satisfied: 
\begin{enumerate}
\item
There is an equality as underlying graded $R$-modules   
$$
\widetilde{F}=
F_{n-1}[n-1] \oplus F_{n-2}[n-2] \oplus \cdots \oplus F_1[1] \oplus F_0.  
$$ 
\item
Let $in_i : F_i[i] \to \widetilde{F}$  and  $pr_i: \widetilde{F} \to F_i[i]$  be respectively a natural injection and a natural projection of graded $R$-modules according to the direct sum decomposition in $(1)$. 
Denoting  by $d_{\widetilde{F}}$ the differential mapping of  $\widetilde{F}$, we have equalities of graded $R$-module homomorphisms; 
$$
\begin{cases}
\  pr_j \,  d_{\widetilde{F}} \, in_i = 0  & \text{for} \ \ 0 \leq i \leq j \leq n-1,  \\   
\  pr_{i-1} \, d_{\widetilde{F}} \, in_i = f_i [i]   & \text{for} \ \  1 \leq i \leq n-1.
\end{cases}
$$  
\item
The natural inclusion $in_0 : F_0 \to \widetilde{F}$ and the natural projection $pr_{n-1} : \widetilde{F} \to F_{n-1} [n-1]$ {are chain maps and they} yield  the morphisms in $\K$  which make the following diagrams in $\K$ commutative;
\begin{equation}\label{phipsi}
\xymatrix@C=32pt{
X_n [n-1] \ar[r]^(0.65){\psi_n} \ar[rd]_{q_n[n-1]} &  \widetilde{F} \ar[d]^{pr_{n-1}}  & \qquad F_0 \ar[d]_{in_0} \ar[rd]^{p_0} \qquad &  \\
& F_{n-1}[n-1]   &  \widetilde{F} \ar[r]_(0.4){\varphi _n}  & \quad X_0 \quad   \\
}
\end{equation}
{\item {\phantom{}}} 
As an object of  $\K$, such a complex  $\widetilde{F}$  is unique up to isomorphism. 
\par
{\rm 
We call $\widetilde{F}$ the  {\bf contraction}  of the partial  $\F$-resolution  (\ref{Kexact2}). 
The triangle (\ref{contraction sequence}) is called the  {\bf  contracted triangle}  of (\ref{Kexact2}).} 
\end{enumerate}
\end{theorem/definition}

\begin{proof}
We prove the theorem by the induction on $n$. 
If $n=1$, then there is a triangle 
$$
\xymatrix@C=32pt{ 
X_1 \ar[r]^{q _1} &F_0 \ar[r]^{p_0} &X_0 \ar[r]^(0.4){\omega _0} & X_1[1].
}
$$ 
Hence, {setting} $\widetilde{F} = F_0$, $\psi_1=q_1$  and  $\varphi _1= p_0$ in this case, {we see} the conditions (1) - (3) are satisfied.

Now we assume  $n > 1$. 
Setting $\widetilde{F'}$ as the contraction of the partial $\F$-resolution  
$$
\xymatrix@C=32pt{ 
0 \ar[r] & X_{n-1} \ar[r]^{q_{n-1}} & F_{n-2} \ar[r]^{f_{n-2}} &F_{n-3} \ar[r]^{f_{n-3}} &\cdots \ar[r] F_1 \ar[r]^{f_1} & F_0 \ar[r]^{p_0} & X_0 \ar[r] & 0, }
$$
we assume that the theorem holds for this partial resolution and  $\widetilde{F'}$. 
Then we have the following octahedron diagram: 

\begin{equation}\label{proven diagram}
\xymatrix@C=40pt@R=36pt{
& \widetilde{F'} \ar[d]^{\iota_{n-1} }\ar@{=}[r] & \widetilde{F'} \ar[d]^{\varphi_{n-1}} & \\ 
X_n [n-1] \ar[r]^{\psi_n}  \ar@{=}[d] & \widetilde{F} \ar[r]^{\varphi_n} \ar[d]^{pr_{n-1}} & X_0 \ar[r]^{\widetilde{\omega_n}}  \ar[d] ^{\widetilde{\omega_{n-1}}} & X_n [n] \ar@{=}[d] \\
X_n [n-1] \ar[r]^{q_n [n-1]}  & F_{n-1} [n-1]  \ar[d]^{\alpha_{n-1}[1]} \ar[r]^{p_{n-1}[n-1]} & X_{n-1}[n-1] \ar[d]^{\psi_{n-1}[1]}  \ar[r]^(0.6){\omega_{n-1}[n-1]}& X_n[n]  \\   
& \widetilde{F'}[1]  \ar@{=}[r] & \widetilde{F'}[1]  &  
}
\end{equation}
where  $\alpha _{n-1} = \psi_{n-1}\, p_{n-1}[n-2]$.  
In fact, the third column and the third row of this diagram are triangles by the induction hypothesis and the $\K$-exactness assumption. 
The second column is a triangle given by setting $\widetilde{F} = Cone (\alpha_{n-1})$.   
The second row gives the desired triangle for the case of $n$. 

\par
Since  $\widetilde{F}$  is the mapping cone of the morphism $\alpha _{n-1}$,  it equals $F_{n-1}[n-1]  \oplus \widetilde{F'}$ as an underlying graded $R$-module, and since  $d_{F_{n-1}}=0$,  the summand $F_{n-1}[n-1]$ is mapped by $\alpha _{n-1}[1]$  into  $\widetilde{F'}[1]$  under the differential  $d_{\widetilde{F}}$. 
This proves  $pr _{n-1} \, d_{\widetilde{F}} \, in_{n-1} = 0$  and the restriction of  $d_{\widetilde{F}}$  {to}  $\widetilde{F'}$  is its own differential $d_{\widetilde{F'}}$. 
It then follows from the induction hypothesis on $\widetilde{F'}$ that  
$pr_j \,  d_{\widetilde{F}} \, in_i = 0$  holds for $0 \leq i \leq j \leq n-1$. 
If $1 \leq i \leq n-2$  then the equality $pr_{i-1} \, d_{\widetilde{F}} \, in_i = pr_{i-1} \, d_{\widetilde{F'}} \, in_i = f_i [i]$ holds by the induction hypothesis. 
To prove that  $pr_{n-2} \, d_{\widetilde{F}} \, in_{n-1} = f_{n-1} [n-1]$, we note that $d_{\widetilde{F}} \, in _{n-1} = \alpha _{n-1} [1]$ and that $pr'_{n-2} : \widetilde{F'} \to F_{n-2}[n-2]$  is a chain map by the induction hypothesis. 
We consider the following diagram: 
\begin{equation} 
\xymatrix@C=46pt@R=46pt{
F_{n-1}[n-1] \ar[r]^{\alpha_{n-1}[1] }   \ar[d] _{p_{n-1}[n-1]} &  \widetilde{F'} [1]  \ar[d]^{pr'_{n-2}[1]} \\ 
X_{n-1}[n-1] \ar[r]_{q_{n-1}[n-1]} \ar[ur]_{\psi_{n-1}[1]}   & F_{n-2}[n-1]  \\ 
}
\end{equation}
The upper left triangle is commutative by the definition of  $\alpha _{n-1}$, and the lower right triangle is also commutative by the induction hypothesis for $\widetilde{F'}$. 
Therefore the square above is a commutative diagram, and thus we have 
$pr' _{n-2}[1] \, \alpha_{n-1} [1] = q_{n-1}[n-1] \, p_{n-1}[n-1] = f_{n-1}[n-1]$. 
This proves    $pr_{n-2} \, d_{\widetilde{F}} \, in_{n-1} = f_{n-1} [n-1]$, and the conditions (1) and (2) are  proved.  

We can see from (2) that  $d_{\widetilde{F}} \, in _0 =0$ and  $pr_{n-1} \, d_{\widetilde{F}}=0$, hence  $in_0$ and $pr_{n-1}$  are chain maps. 
The commutativity of the left triangle in (\ref{phipsi})  follows from the diagram (\ref{proven diagram}). 
To prove that the right triangle in (\ref{phipsi}) is also commutative, note that  $p_0 = \varphi_{n-1} \, in'_0$ holds by the induction hypothesis, where  $in'_0 : F_0 \to \widetilde{F'}$  is the natural injection. 
Note form the diagram (\ref{proven diagram}) that  $\varphi _{n} \, \iota _{n-1} = \varphi _{n-1}$  and  $\iota_{n-1} \, in' _0 = in_0$. 
Thus we obtain  $\varphi_n \, in _0 = \varphi _n \, \iota _{n-1} \, in'_0 = \varphi_{n-1}\, in'_0 = p_0$. 

Since the connecting morphism $\widetilde{\omega_n}$  is uniquely determined by the $\K$-exact sequence, 
noting that  $\widetilde{F} \cong  Cone (\widetilde{\omega_n}[-1])$  in $\K$, 
we see that  $\widetilde{F}$  is unique up to isomorphism in $\K$. 
\end{proof}

\begin{remark}\label{matrix form psi} 
Another interpretation of the conditions (1)-(3) in the theorem is the following: 
Now returning to the setting of  Theorem \ref{theorem/definition}, the contraction of the $\K$-exact sequence (\ref{Kexact2}) is  $F_{n-1}[n-1]\oplus \cdots \oplus F_1[1] \oplus F_0$  as an underlying graded $R$-module and the differential $d_{\widetilde{F}}$ is given by a matrix of the form; 
\begin{equation}\label{contraction diff}
d_{\widetilde{F}} = 
\begin{pmatrix}
 0 & 0 &    \dots\dots &  0 &0  \\
 f_{n-1}[n-1] & 0 &    \dots\dots & 0 &0 \\
 a_{n-1 \, n-3} & f_{n-2}[n-2] &    \dots\dots & 0 &0 \\
 a_{n-1 \, n-4} & a_{n-2 \, n-4} &    \dots\dots & 0 &0 \\
 \vdots & \vdots  & \ddots & \vdots &\vdots  \\
 a_{n-1 \, 0}& a_{n-2 \, 0}& \dots\dots  & f_1[1] & \phantom{AA}0\phantom{AA} \\
\end{pmatrix},  
\end{equation}
where  each  $a _{i \, j} : F_i[i] \to F_j[j+1]$  is a graded $R$-homomorphism. 
On the other hand, the commutativity of the diagrams  (\ref{phipsi}) says that as underlying graded $R$-module homomorphisms 
$\psi _n$  and  $\varphi _n$  are represented respectively by the following matrices: 
$$
\begin{array}{l}
\psi _n = 
\begin{pmatrix}
q_{n}[n-1] \\ a_{n\, n-2} \\ \vdots \\ a_{n \, 0}
\end{pmatrix}
: X_n [n-1] \longrightarrow F_{n-1}[n-1] \oplus F_{n-2}[n-2] \oplus \cdots \oplus F_1[1] \oplus F_0 \vspace{6pt} \\
\varphi _n = 
\begin{pmatrix}
b_{n-1 \, 0}  & \cdots & b_{1\, 0}  & p_0 \\ 
\end{pmatrix}
: F_{n-1}[n-1] \oplus F_{n-2}[n-2] \oplus \cdots \oplus F_1[1] \oplus F_0  \longrightarrow X_0
\end{array}
$$
for some graded $R$-homomorphisms $a_{n \, i} : X_n[n-1] \to F_i [i]$ and  $b _{i \, 0} : F_i [i] \to X_0$.   
\end{remark}

\begin{definition}\label{morphism of resolutions}
Assume that we have a commutative diagram  
\begin{equation}\label{morphism1}
\xymatrix@C=32pt{ 
0 \ar[r] & X_n  \ar[r]^{q_n^F} & F_{n-1} \ar[r]^{f_{n-1}} &  \cdots \ar[r] & F_1 \ar[r]^{f_1} & F_0 \ar[r]^{p_0^F} & X_0 \ar[r] & 0\\  
0 \ar[r] & Y_n  \ar[r]^{q_n^G} \ar[u]^{t_n} & G_{n-1} \ar[r]^{g_{n-1}} \ar[u]^{s_{n-1}} &  \cdots \ar[r] & G_1 \ar[r]^{g_1} \ar[u]^{s_1} & G_0 \ar[r]^{p_0^G} \ar[u]^{s_0} & Y_0 \ar[r] \ar[u]^{t_0}  & 0, }  
\end{equation}
where the rows are partial $\F$-resolutions. 
We say that the diagram (\ref{morphism1}) gives a morphism between the partial $\F$-resolutions if there are commutative diagrams
$$
\xymatrix@C=32pt{ 
X_{i+1} \ar[r]^{q_{i+1}^F}  &  F_{i}  \ar[r]^{p_{i}^F}  &  X_{i}  \ar[r]^(0.4){\omega_i^F} &  X_{i+1}[1] \\
Y_{i+1} \ar[r]^{q_{i+1}^G} \ar[u]^{t_{i+1}} &  G_{i}  \ar[r]^{p_{i}^G} \ar[u]^{s_{i}} &  Y_{i}  \ar[r]^(0.4){\omega_i^G} \ar[u]^{t_i}& Y_{i+1}[1] \ar[u]^{t_{i+1}{[1]}}, \\
}
$$
where each row is a triangle in $\K$ for  $0 \leq i < n$, and  $f_i= q_i^F p_i^F$,  $g_i= q_i^Gp_i^G$  for  $1\leq i <n$. 

In such a case we have a morphism $\widetilde{s}$  between the contractions with the diagram;   
\begin{equation}\label{morphism of contractions}
\xymatrix@C=32pt{
X_n[n-1] \ar[r]^(0.6){\psi _n^F} & \widetilde{F} \ar[r]^{\varphi_n^F} & X_0 \ar[r]^(0.4){\widetilde{\omega_n^F}} & X_n [n]  \\
Y_n[n-1] \ar[r]^(0.6){\psi _n^G} \ar[u]^{t_{n}[n-1]}& \widetilde{G} \ar[r]^{\varphi_n^G} \ar[u]^{\widetilde{s}} & Y_0 \ar[r]^(0.4){\widetilde{\omega_n^G}} \ar[u]^{t_{0}}& Y_n [n]. \ar[u]^{t_{n}[n]}  
}
\end{equation}
Since the morphisms  $t_0$  and  $t_n$  are given beforehand, 
such a morphism  $\widetilde{s}$  obviously exists so that the diagram (\ref{morphism of contractions})  will be commutative. 
Unfortunately  it is not unique in general.  
\end{definition}

But we can prove the following theorem. 

\begin{theorem}\label{lower triangle matrix}
Under the circumstances in Definition \ref{morphism of resolutions}, we furthermore assume that all the $F_i,  G_i \  ( 1 \leq i <n)$ do not contain  any null complexes as direct summands. 
Then we can take a morphism $\widetilde{s} : \widetilde{G} \to \widetilde{F}$   
so that it is represented by the following type of lower {triangular} matrix as an underlying graded $R$-module homomorphism according to the direct sum decompositions  
  $\widetilde{G} = G _{n-1}[n-1] \oplus G_{n-2}[n-2] \oplus G_0$ and  $\widetilde{F} = F _{n-1}[n-1] \oplus F_{n-2}[n-2] \oplus F_0$: 
    
 $$
\begin{pmatrix}
 s_{n-1}[n-1] & 0 &  0 &   \dots\dots & 0  \\
 *  & s_{n-2}[n-2] &  0 &    \dots\dots &0 \\
 *  & *  & s_{n-3}[n-3]  &    \dots\dots & 0 \\
 \vdots & \vdots  & \vdots & \ddots &\vdots  \\
 * & * & * &  \dots\dots  & s_0 \\
\end{pmatrix}  
$$

In this case the following diagram is commutative: 
$$
\xymatrix@C=48pt{
\widetilde{F}  \ar[r] ^(0.4){pr^F_{n-1}} & F_{n-1}[n-1] \\
\widetilde{G}  \ar[r] ^(0.4){pr^G_{n-1}} \ar[u]^{\widetilde{s}} & G_{n-1}[n-1] \ar[u]^{s_{n-1}[n-1]} }
$$ 
\end{theorem}

\begin{proof}
The last part of the theorem follows from the first, since  $pr _{n-1}^F$ and  $pr _{n-1}^G$ are represented by matrices of the form 
$(1 \ 0 \   \cdots \ 0)$. 

We prove the first part by induction on $n$.

If $n=1$, then $\widetilde{F}=F_0$, $\widetilde{G}=G_0$ and $\widetilde{s} = s_0$. 
Hence the statement is true.

Assume $n \geq 2$. 
Under the same notation as in the proof of Theorem  \ref{theorem/definition}, 
$\widetilde{F}$ and $\widetilde{G}$  are the mapping cones of the morphisms   $\alpha _{n-1}^F = \psi _{n-1} ^F p_{n-1}^F[n-2]$ and 
 $\alpha _{n-1}^G = \psi _{n-1} ^G p_{n-1}^G[n-2]$ respectively. 
By the induction hypothesis for $n-1$  we may assume that we have such an $\widetilde{s'}$ that is represented by a lower {triangular} matrix and that makes the following diagram commutative.  
$$
\xymatrix@C=48pt{
X_{n-1}[n-2] \ar[r]^(0.6){\psi _{n-1}^F} & \widetilde{F'} \ar[r]^{\varphi_{n-1}^F} & X_0 \ar[r]^(0.4){\widetilde{\omega_{n-1}^F}} & X_{n-1} [{n-1}]  \\
Y_{n-1}[n-2] \ar[r]^(0.6){\psi _{n-1}^G} \ar[u]^{t_{n-1}[n-2]}& \widetilde{G'} \ar[r]^{\varphi_{n-1}^G} \ar[u]^{\widetilde{s'}} & Y_0 \ar[r]^(0.4){\widetilde{\omega_{n-1}^G}} \ar[u]^{t_{0}}& Y_{n-1} [n-1]. \ar[u]_{t_{n-1}[n-1]}  
}
$$
Since there is a triangle  $\xymatrix@C=24pt{X_n \ar[r] & F_{n-1} \ar[r]^{p_{n-1}^F} & X_{n-1} \ar[r] & X_n [1]} $ in $\K$,  
we see that  $X_n$ is isomorphic to the mapping cone of  $p_{n-1}^F[-1]$ i.e.
$F_{n-1} \oplus X_{n-1} [-1]$  is its underlying graded $R$-module and it has the differential  
$d _{X_n} = \begin{pmatrix} 0& 0 \\ p_{n-1}^F  & d_{X_{n-1}[-1]}\end{pmatrix}$. 
Similarly  $Y_n \cong G_{n-1}  \oplus Y_{n-1} [-1]$ as an underlying graded $R$-module with 
$d _{Y_n} = \begin{pmatrix} 0& 0 \\ p_{n-1}^G & d_{Y_{n-1}[-1]}\end{pmatrix}$. 
Note that  $X_{n-1}[-1]$  is a subcomplex of this mapping cone, and  $F_{n-1}$  is a quotient of it. 
{So the commutative diagram with rows being triangles; 
$$\xymatrix{
X_{n-1}[-1]  \ar[r] &X_n \ar[r] & F_{n-1} \ar[r] & X_{n-1} \\ 
Y_{n-1}[-1] \ar[u]^{t_{n-1}[-1]} \ar[r] & Y_n \ar[r] \ar[u]^{t_n} & G_{n-1} \ar[r] \ar[u]^{s_{n-1}} & Y_{n-1}  \ar[u]^{t_{n-1}}  
}$$
}
{is represented by a commutative diagram of exact sequences; 
$$\xymatrix{
0 \ar[r] & X_{n-1}[-1]  \ar[r] &X_n \ar[r] & F_{n-1} \ar[r] & 0 \\ 
0 \ar[r] & Y_{n-1}[-1] \ar[u]^{t_{n-1}[-1]} \ar[r] & Y_n \ar[r] \ar[u]^{t_n} & G_{n-1} \ar[r] \ar[u]^{s_{n-1}} & 0  
}$$
}

Since  $t_n$  maps the subcomplex  ${Y}_{n-1}[-1]$ into  the subcomplex  ${X}_{n-1}[-1]$, 
we can see that  $t_n$  is represented by a matrix 
$$
\xymatrix@C=90pt{
G_{n-1} \oplus Y_{n-1} [-1]  \ar[r]^{
\begin{pmatrix}
s_{n-1} & 0 \\ u & t_{n-1}[-1]   
\end{pmatrix}
}
& F_{n-1} \oplus X_{n-1}[-1], 
}
$$  
where  $u : G_{n-1} \longrightarrow X_{n-1}[-1]$  is a graded $R$-homomorphism. 
Identifying those relevant complexes under such isomorphisms, we also see from the inductive construction of $\widetilde{F}$ in the proof of Theorem \ref{theorem/definition}  that  the morphism $\psi _n ^F : X_n[n-1] \longrightarrow \widetilde{F}$  is given by the chain map 
$$
\xymatrix@C=72pt{
F_{n-1}[n-1] \oplus X_{n-1}[n-2]  \ar[r]^(0.6){
\begin{pmatrix}
1 & 0 \\ 0 & \psi _{n-1}^F   
\end{pmatrix}
}&  F_{n-1}[n-1] \oplus \widetilde{F'}.  
}
$$
{See the diagram (\ref{proven diagram}).}
Similarly  $\varphi_n ^F : \widetilde{F} \longrightarrow X_0$ is represented by 
$$
\xymatrix@C=72pt{
F_{n-1}[n-1] \oplus \widetilde{F'}  \ar[r]^(0.6){
\begin{pmatrix}
0 & \varphi _{n-1}^F \\    
\end{pmatrix}}
&  X_0.  
}
$$
Finally it is easy to see that the following diagram is commutative: 
$$
\xymatrix@C=60pt@R=60pt{
F_{n-1}[n-1] \oplus X_{n-1} [n-2] \ar[r]^(0.6){
\tiny{\begin{pmatrix}
1 & 0 \\ 0 & \psi _{n-1}^F   
\end{pmatrix}
} } 
& F_{n-1} [n-1] \oplus \widetilde{F'}  \ar[r]^(0.6){
\tiny{\begin{pmatrix}
0 & \varphi _{n-1}^F \\    
\end{pmatrix}}}  
& X_0  \\  
G_{n-1}[n-1] \oplus Y_{n-1} [n-2] \ar[r]_(0.6){
\tiny{\begin{pmatrix}
1 & 0 \\ 0 & \psi _{n-1}^G   
\end{pmatrix}
} } \ar[u]^{
\tiny{\begin{pmatrix}
s_{n-1}[n-1] & 0 \\ u[n-1] & t_{n-1}[n-2]  
\end{pmatrix}}}
& G_{n-1} [n-1] \oplus \widetilde{G'}  \ar[r]_(0.6){
\tiny{\begin{pmatrix}
0 & \varphi _{n-1}^G \\    
\end{pmatrix}}} \ar[u]^{
\tiny{\begin{pmatrix}
s_{n-1}[n-1] & 0 \\ \psi_{n-1}^Fu[n-1]  & \widetilde{s'}  
\end{pmatrix}}
}
& Y_0  \ar[u]^{t_0} \\  
}
$$
Therefore we can take the matrix  
$\begin{pmatrix}
s_{n-1}[n-1] & 0 \\ \psi_{n-1}^Fu[n-1]  & \widetilde{s'}  
\end{pmatrix}$  as $\widetilde{s}$. 
Since $\widetilde{s'}$ is taken to be a lower {triangular} matrix by the induction hypothesis, so is $\widetilde{s}$. 
\end{proof}

\vspace{6pt}
\section{Remarks on partial $\F$-resolutions}

\begin{definition}\label{degenerate}
$\ $
\par\noindent
(1) We say that a partial $\F$-resolution (\ref{Kexact2}) is {\bf split} if 
each  $q_i$  in Definition \ref{def Kexact}  has a left inverse, i.e. $q_i$  is a split monomorphism,  for all $1 \leq i \leq n$. 
This is equivalent to $\omega _i =0$  for all $0 \leq i < n$, with the notation in Definition \ref{def Kexact}. 
{See \cite[Lemma 1.4]{H}. } 

\medskip 

\par\noindent
(2) We say that a partial $\F$-resolution (\ref{Kexact2}) is {\bf degenerate} if one can choose  the differential  $d_{\widetilde{F}}$  {such that} $pr_{j} {\circ} d_{\widetilde{F}} {\circ}  in_i = 0$ {for all $1 \leq i, j \leq n-1$ with $j \not= i-1$} under the notation of Theorem \ref{theorem/definition}. 
This is equivalent to saying that one can take the differential of the form     
\begin{equation}\label{diff degenerate}
d_{\widetilde{F}} = 
\begin{pmatrix}
0                & 0 &  0 &   \dots\dots & 0 &0  \\
f_{n-1}[n-1] & 0 & 0 &    \dots\dots &  0 &0  \\
0 & f_{n-2}[n-2] & 0 &    \dots\dots & 0 &0 \\
0 & 0 & f_{n-3}[n-3] &    \dots\dots & 0 &0 \\
\vdots & \vdots & \vdots  & \ddots & \vdots &\vdots  \\
0 &0&0& \dots\dots  & f_1[1] & \phantom{AA}0\phantom{AA} \\
\end{pmatrix}
\end{equation}
as a {graded} $R$-module mapping from $\widetilde{F}=F_{n-1}[n-1] \oplus F_{n-2}[n-2] \oplus \cdots \oplus F_1[1] \oplus F_0 
$ to $\widetilde{F}[1]$. 
Note in this case that we have an equality 
$$
\widetilde{F} = \coprod  _{i \in \Z} \left[
\xymatrix{
0 \ar[r] & F_{n-1} ^i \ar[r]^(0.4){f_{n-1}^i} &  F_{n-2}^i \ar[r] & \cdots \ar[r] & F_1^i \ar[r]^{f_1^i}  \ar[r] & F_0^i \ar[r] & 0} 
 \right] [-i]. 
$$
\end{definition}

The following proposition will be necessary in {a} later argument of this paper. 

\begin{proposition}
Let 
\begin{equation}\label{split resolution}
\xymatrix@C=32pt{ 
0 \ar[r] & F_{n-1} \ar[r]^{f_{n-1}} & F_{n-2} \ar[r] & \cdots \ar[r] & F_1 \ar[r]^{f_1} & F_0 \ar[r]^{p_0} & X_0 \ar[r] & 0}  
\end{equation}
be an $\F$-resolution of length $n-1$ {where each $F_i$ has no null complex as a direct summand}  and  let  $\widetilde{F}$  be its contraction. 
Assume that $n \geq 2$  and $f_{n-1}$  has a left inverse in $\K$. 
Then  $X_0 \cong \widetilde{F}$ and the morphism  $pr_{n-1} : \widetilde{F} \longrightarrow  F_{n-1}[n-1]$  is zero in  $\K$. 
\end{proposition}

\begin{proof}
The isomorphism  $X_0 \cong \widetilde{F}$ follows from the contraction sequence (\ref{contraction sequence}) in Theorem \ref{theorem/definition} by setting   $X_n=0$. 
Note that $pr_{n-1}$ is represented by the matrix 
$$
\begin{pmatrix} 1 & 0 & \cdots & 0 \end{pmatrix} : F_{n-1}[n-1] \oplus \cdots \oplus F_1[1] \oplus F_0 \longrightarrow F_{n-1}[n-1]. 
$$
Let  $v$  be a left inverse of  $f_{n-1}$, i.e. $v : F_{n-2} \to F_{n-1}$ such that  $v f _{n-1}= 1_{F_{n-1}}$ and set 
$\widetilde{v} : \widetilde{F}[1]  \to F_{n-1}[n-1]$ as a graded $R$-homomorphism given by the matrix 
$$
\begin{pmatrix} 0 & v[n-1] & 0 & \cdots & 0 \end{pmatrix} : F_{n-1}[n] \oplus F_{n-2}[n-1] \oplus \cdots \oplus F_0 \longrightarrow F_{n-1}[n-1]. 
$$
Then, since the differential $d_{\widetilde{F}}$  is represented by the matrix (\ref{contraction diff}), 
it is easy to see that   $pr_{n-1} = \widetilde{v} d_{\widetilde{F}}$. 
Hence  $pr_{n-1}$  is null homotopic.   
\end{proof}

\begin{corollary}\label{split pr}
Assume that the $\F$-resolution (\ref{split resolution}) is split and  $n \geq 2$. 
Then $pr_{n-1} = 0$  in $\K$.  
\end{corollary}

\begin{lemma}\label{split case}
We assume that the following conditions are satisfied for the partial $\F$-resolution (\ref{Kexact2}) : 
\begin{enumerate}
\item
$X_0, X_n$  belong to $\F$ and they have no null complexes as direct summands. 
\item
As a sequence of  graded $R$-modules, the sequence  
$$
\xymatrix@C=32pt{ 
0 \ar[r] & X_n  \ar[r]^{q_n} & F_{n-1} \ar[r]^{f_{n-1}} & F_{n-2} \ar[r] & \cdots \ar[r] & F_1 \ar[r]^{f_1} & F_0 \ar[r]^{p_0} & X_0 \ar[r] & 0}  
$$
is exact. 
\end{enumerate}
Then  the partial $\F$-resolution is degenerate. 
The contracted triangle (\ref{contraction sequence}) is realized by the morphisms represented by the following  of underlying graded $R$-module homomorphisms: 
$$
\begin{array}{l}
\psi _n = 
\begin{pmatrix}
q_{n}[n-1] \\ 0 \\ \vdots \\ 0
\end{pmatrix}
: X_n [n-1] \longrightarrow \widetilde{F} = F_{n-1}[n-1] \oplus F_{n-2}[n-2] \oplus \cdots \oplus F_1[1] \oplus F_0 \vspace{12pt} \\
\varphi _n = 
\begin{pmatrix}
0 & \cdots & 0 & p_0 \\ 
\end{pmatrix}
: \widetilde{F} = F_{n-1}[n-1] \oplus F_{n-2}[n-2] \oplus \cdots \oplus F_1[1] \oplus F_0  \longrightarrow X_0
\end{array}
$$
\end{lemma}

\begin{proof}
Set  $d_{\widetilde{F}}$  as in (\ref{diff degenerate}) and we see by a straightforward  computation that 
$d_{\widetilde{F}}^2 =0$,  $d_{\widetilde{F}} \psi _n = 0$  and  $\varphi _n d_{\widetilde{F}}= 0$ where  $\psi_n$ and $\varphi _n$ are given as in the lemma. 
Therefore the matrices given in the lemma define chain homomorphisms. 
It is then easy to see that the sequence 
$$
\xymatrix@C=32pt{
X_n[n-1] \ar[r]^(0.6){\psi_n} &\widetilde{F} \ar[r]^{\varphi_n} & X_0 \ar[r]^(0.4){0} & X_n [n]   
}$$
is a triangle in $\K$. 
{(By the condition (2), the mapping $p_0 : F_0 \to X_0$ is surjective as an underlying graded module homomorphism, hence it is a split epimorphism in $\F$. 
Thus, in the triangle $$\xymatrix{X_1 \ar[r] ^{q_1} & F_0 \ar[r]^{p_0}  & X_0 \ar[r]^{\omega_0}  & X_1[1]}$$ 
we have  $\omega _0 =0$. 
Therefore $\widetilde{\omega _n} =0$ in the contracted triangle above. )}
\end{proof}

\begin{corollary}
If a partial $\F$-resolution is split, then it is degenerate. 
\end{corollary}

We should note that all partial $\F$-resolutions of length $n \leq 2$ are degenerate. 
In fact, if  $n=2$  then the $\K$-exact sequence is 
$\xymatrix@C=24pt{ 0 \ar[r] &X_2 \ar[r] & F_1 \ar[r]^{f_1} & F_0 \ar[r] & X \ar[r]  & 0}$ where one can take $F_1$ and $F_2$ have no null complexes as summands, 
and $\widetilde{F}$  is the mapping cone of   $f_1$, therefore the sequence is degenerate.

Note also that, even if there is a $\K$-exact sequence 
$$
\xymatrix@C=32pt{
0 \ar[r] & F_n  \ar[r]^{f_n}  & F_{n-1} \ar[r] & \cdots \ar[r] &  F_1 \ar[r]^{f_1} & F_0 \ar[r] & X \ar[r] & 0}
$$
for assigned  $F_i \in \F$ and $f_i$, the rightmost complex  $X$ is not uniquely determined.

In fact, $X$ depends not only on $f_i$ but also on  $p_i$, $q_i$ with  $f _i =p_i q_i$ as  in Definition \ref{theorem/definition}. 

{
For example consider the $\F$-resolution 
$$
\xymatrix@C=35pt{0 \ar[r] & F_2=R \ar[r]^{\tiny{\begin{pmatrix}  a\\ b \end{pmatrix}}} & F_1 = R\oplus R \ar[r]^{\tiny{\begin{pmatrix}  0 & 0 \\ -b  & a \end{pmatrix}}} & F_0 = R[1] \oplus R  \ar[r] & X_0 \ar[r] &0, }
$$
where $a, b \in R$. 
Then, under the notation of (\ref{theorem/definition}),  $q_1$ is described as 
$$
\xymatrix{X_1 \ar[d]^{q_1}  &=  & [ 0 \ar[r] \ar[d]  & R \ar[r]^{\tiny{\begin{pmatrix}  a\\ b \end{pmatrix}}} \ar[d]^c  & R^2 \ar[r] \ar[d]^{\tiny{\begin{pmatrix}  -b & a  \end{pmatrix}}} & 0] \\ 
F_0 &= & [0 \ar[r] & R \ar[r]^0 & R \ar[r] & 0], }
$$
where one can take any element of $R$ as $c$. 
If $c=0$, then 
$$X_0= R[1] \oplus [\xymatrix@C=30pt{0 \ar[r] & R \ar[r]^{\tiny{\begin{pmatrix}  a\\ b \end{pmatrix}}} & R^2 \ar[r]^(0.55){\tiny{\begin{pmatrix}  -b  & a \end{pmatrix}}} & R \ar[r] & 0}]. 
$$
If $c=1$, then 
    $$X_0= [\xymatrix@C=30pt{0 \ar[r] & R^2 \ar[r]^(0.55){\tiny{\begin{pmatrix}  -b  & a \end{pmatrix}}} & R \ar[r] & 0}]. \quad {\qed} $$
} 

\vspace{16pt}

For $X \in \K $ and for an  integer  $n > 0$, 
we define the $n$-th syzygy and cosyzygy by  induction on $n$; 
$$
\Omega ^0 X = \Sigma ^0 X = X, \quad 
\Omega ^{n} X = \Omega (\Omega ^{n-1}X), \quad 
\Sigma ^{n} X = \Sigma (\Sigma ^{n-1}X).
$$
Recall from Definitions \ref{omega} and \ref{omega inverse} that  $\Omega ^nX$ and $\Sigma^{n}X$  are uniquely determined as objects in $\uK$, or in other words they are  unique  up to $\F$-summands as objects in $\K$.   
Actually they define the functors  $\Omega^n, \Sigma^{n} : \uK \to \uK$, and Theorem \ref{adjoint} assures  that 
$(\Sigma ^{n}, \Omega ^{n})$ is an adjoint pair for each  $n>0$. 

\begin{definition}\label{def of omega}
Let  $X \in \K$ and take a right $\F$-approximation $p _0 : F_0 \to X$. 
We embed $p_0$  into a triangle to get the first syzygy $\Omega X$; 
$$
\xymatrix@C=32pt{ 
\Omega X \ar[r]^{q _1} & F_0 \ar[r]^{p _0}  &  X \ar[r]^(0.4){\omega _1^X} & \Omega X[1]. } 
$$
Similarly but as for the dual version to this, we have a triangle for any $Y \in \K$;  
$$
\xymatrix@C=32pt{ 
\Sigma Y [-1] \ar[r]^(0.6){\omega _{-1}^Y} & Y \ar[r]^{{q ^{0}}}  &  G_0 \ar[r]^(0.4){{p ^{0}}}  & \Sigma Y,}
$$
where  {$q^{0}$}  is a left $\F$-approximation. 
In such a way we have morphisms  $\omega _1^X$  and  $\omega _{-1}^Y$. 
Now let  $n$ be a positive integer. 
We define inductively 
$$
\omega _n ^X =  {\omega _{n-1}^{\Omega X[1]}\omega _1^X} \ : \ X \longrightarrow  \Omega ^n X [n],  \quad 
\omega _{-n} ^Y =   \omega _{-1}^{Y} \omega _{-n+1}^{\Sigma Y[-1]} \ : \ \Sigma ^{n}Y [-n] \longrightarrow  Y .
$$
\end{definition}

Let  $X$  be an arbitrary object in $\K$. 
Note from the definition of $\Omega ^i$  that there are triangles 
$$
\xymatrix@C=32pt{ 
\Omega ^{i+1} X \ar[r]^(0.6){q _{i+1}} & F_i \ar[r]^{p _i} & \Omega ^{i}X \ar[r]^(0.4){\omega _1^{\Omega^i X}} &  \Omega ^{i+1}X[1],} 
$$ 
where $F_i \in \F$ and  $p_i$ is a right $\F$-approximation of $\Omega ^i X$ for all $i \geq 0$.
Hence, when $n$ is a positive integer, we have a partial $\F$-resolution of the form; 
\begin{equation}\label{resolution}
\xymatrix@C=32pt{ 
0 \ar[r] &  \Omega ^n X  \ar[r]^{q _n} &  F_{n-1} \ar[r]^{f_{n-1}} & \cdots \ar[r] & F_1 \ar[r]^{f_1} &  F_0 \ar[r]^{p_0}  & X \ar[r]  & 0, } 
\end{equation}
where $F_i \in \F$ and  $f_i = q _i p _i$ for $0 \leq i < n$. 
We note here that we may assume that all the $F_i \ (0 \leq i < n)$ have zero differentials, because we can take them up to    isomorphisms in  $\uK$. 
(Cf. Theorem \ref{characterization of F}.)
{Therefore from now on we assume that the $F_i$'s have zero differentials.} 
Hence   $H(F_i)= F_i$ for all $0 \leq i < n$. 
Note also that $\omega _n ^X$ defined in Definition \ref{def of omega}  is the connecting morphism of the partial $\F$-resolution (\ref{resolution}).

The following theorem is a {partial} restatement of Theorem \ref{theorem/definition}. 

\begin{theorem}\label{omega triangle}
Under the circumstances above, there is a triangle in $\K$; 
$$
\xymatrix@C=32pt{ 
\Omega ^nX [n-1] \ar[r]^(0.65){\psi_n} &  \widetilde{F} \ar[r]^{\varphi_n}  &  X \ar[r]^(0.35){\omega _n^X} & \Omega ^nX [n],}
$$
where the morphisms $\psi _n : \Omega ^n X[n-1] \to \widetilde{F}$ 
and $\varphi  _n : \widetilde{F} \to X$ make the following diagrams commutative:
\begin{equation}\label{psiphi}
\xymatrix@C=32pt{
\Omega ^nX [n-1] \ar[r]^(0.65){\psi_n} \ar[rd]_{q_n[n-1]} &  \widetilde{F} \ar[d]^{pr_{n-1}}  & \qquad F_0 \ar[d]_{in_0} \ar[rd]^{p_0} \qquad &  \\
& F_{n-1}[n-1],  & \widetilde{F} \ar[r]_(0.4){\varphi _n}  & \quad X. \quad   \\
}
\end{equation}
\end{theorem}

\medskip 

\begin{remark} 
We shall make several remarks on  the partial $\F$-resolution  (\ref{resolution}). 
Firstly we see from Lemma \ref{right approximation}  that the following is an exact sequence of graded $R$-modules; 
$$
\xymatrix@C=32pt{ 
0 \ar[r]& H(\Omega ^n X)  \ar[r]^(0.55){H(q _n)} & F_{n-1} \ar[r]^{f_{n-1}} & \cdots \ar[r] &  F_1 \ar[r]^{f_1} & F_0 \ar[r]^(0.4){H(p_0)}   & H(X) \ar[r] & 0,}
$$
which means that there are exact sequences of  $R$-modules
$$
\xymatrix@C=32pt{ 
0 \ar[r] & H^i(\Omega ^n X)  \ar[r]^(0.55){H^i(q _n)} & F_{n-1}^i \ar[r]^{f_{n-1}^i} & \cdots \ar[r] & F_1^i \ar[r]^{f_1^i} & F_0^i \ar[r]^(0.4){H^i(p_0)} & H^i(X) \ar[r] & 0,} 
$$
for all  $i \in \Z$. 
The diagram (\ref{psiphi}) induces the commutative diagram of cohomology modules: 
$$
\xymatrix@C=32pt{
H(\Omega ^nX) [n-1] \ar[r]^(0.65){H(\psi_n)} \ar[rd]_{H(q_n)[n-1]} &  H(\widetilde{F}) \ar[d]^{H(pr_{n-1})}  & \quad F_0 \ar[d]_{H(in_0)} \ar[rd]^{H(p_0)} \qquad &  \\
& F_{n-1}[n-1],  & H(\widetilde{F}) \ar[r]_(0.4){H(\varphi _n)}  & \quad H(X). \quad   \\
}
$$
Since  $H(q_n)$  is injective, so is $H(\psi _n)$.  
Similarly $H(\varphi_n)$  is surjective, as $H(p_0)$  is surjective. 
As a consequence, it follows from the contracted triangle in Theorem \ref{omega triangle} that there is an exact sequence of graded $R$-modules; 
$$
\xymatrix@C=32pt{ 
0 \ar[r] &H(\Omega ^nX) [n-1] \ar[r]^(0.65){H(\psi_n)} &  H(\widetilde{F}) \ar[r]^{H(\varphi_n)}  &  H(X) \ar[r] &0}.
$$
\end{remark} 

\begin{example}
Let  $M$ be a finitely generated $R$-module and 
\begin{equation}\label{projective resolution}
{Y} =\left[  \xymatrix@C=24pt{
\cdots \ar[r] & P_n \ar[r]^{u_n} & P_{n-1} \ar[r] & \cdots \ar[r] & P_1 \ar[r]^{u_1} & P_0 \ar[r] &0} \right]
\end{equation}
be an $R$-projective resolution of $M$, i.e. ${Y} \in \K$  and there is a quasi-isomorphism  ${Y}\to M$. 
In this case it is obvious that  ${\Omega ^n Y}$  is the truncated complex  \linebreak 
$\left[  \xymatrix@C=24pt{ \cdots \ar[r] & P_{n+1} \ar[r]^{u_n} & P_n \ar[r] &0} \right]$, 
and there is a $\K$-exact sequence 
$$
\xymatrix@C=24pt{
0 \ar[r] &\Omega ^n {Y} \ar[r] & P_{n-1} \ar[r]^{u_{n-1}} & P_{n-2} \ar[r] & \cdots \ar[r] & P_1 \ar[r]^{u_1} & P_0 \ar[r] &{Y} \ar[r] & 0.}
$$
In this case the contraction of this partial $\F$-resolution is the complex 
$$
\xymatrix@C=24pt{
0 \ar[r] & P_{n-1} \ar[r] & P_{n-2} \ar[r] & \cdots \ar[r] & P_1 \ar[r] & P_0 \ar[r] & 0, }
$$
hence it is degenerate. 
\end{example}

Even  for such natural constructions we should remark that there are partial $\F$-resolutions of the form (\ref{resolution}) that are not degenerate. 

\begin{example}\label{ex9.9}
Let $M$ be a finitely generated $R$-module and ${Y}$ a projective resolution of $M$ given as in (\ref{projective resolution}).
We consider a complex of length one; 
$$
{X} = \left[ \xymatrix@C=24pt{0 \ar[r] & P_1 \ar[r]^{u_1} & P_0 \ar[r] & 0} \right]. 
$$
As we remarked in Example \ref{exomega} we see that 
$$
\Omega ^n {X} = \left[ \xymatrix@C=24pt{0 \ar[r] & P_{n+1} \ar[r]^{u_{n+1}} & P_n \ar[r] & 0} \right]. 
$$
In fact,  set
$$
f_{i+1} = \begin{pmatrix}
u_{i+1} & 0 \\
0 & u_{i+3}[1]
\end{pmatrix} \ : F_{i+1} := P_{i+1} \oplus P_{i+3}[1] \longrightarrow F_i := P_i \oplus P_{i+2}[1]
$$
for $i \geq 0$, where  each  $F_i$ is a complex with zero differential mappings. 
Furthermore we set 
$$
p_0 = \begin{pmatrix}
1  & 0 \\
0 & u_{2}[1]
\end{pmatrix} \ :
F_0 = P_0 \oplus P_2[1]  \longrightarrow {X} = P_0 \oplus P_1[1], 
$$
and 
$$
q_n  = \begin{pmatrix}
u_n  & 0 \\
0 & 1 
\end{pmatrix} \ :
\Omega ^n {X}= P_n \oplus P_{n+1}[1] \longrightarrow F_{n-1} = P_{n-1} \oplus P_{n+1}[1].  
$$
Then we have a partial $\F$-resolution  
$$
\xymatrix@C=32pt{ 
0 \ar[r] &  \Omega ^n {X}  \ar[r]^{q _n} &  F_{n-1} \ar[r]^{f_{n-1}} & \cdots \ar[r] & F_1 \ar[r]^{f_1} &  F_0 \ar[r]^{p_0}  & {X} \ar[r]  & 0, } 
$$
as in (\ref{resolution}).

In this example we can observe that if  $n \geq 3$  then the partial $\F$-resolution is never {degenerate}. 

For example, in the case  $n=3$, setting a graded $R$-module homomorphism 
$$
g = \begin{pmatrix}
0 & 0 \\ 1 & 0 \\
\end{pmatrix} : F_2[2] = P_2[2] \oplus P_4[3] \longrightarrow F_0[1] = P_0[1] \oplus P_2[2], 
$$ 
we can see that the differential of $\widetilde{F}$  is given by 
$$
d_{\widetilde{F}} = 
\begin{pmatrix}
0 & 0 &0 \\ f_2[2]  & 0 &0 \\ g & f_1[1] & 0
\end{pmatrix} : 
\widetilde{F} = F_2[2] \oplus F_1[1] \oplus F_0 \longrightarrow 
\widetilde{F} [1] = F_2[3] \oplus F_1[2] \oplus F_0[1],  
$$  
which shows that the sequence is not degenerate. 
{(The reason why this is not degenerate is clear from the following observation: 
As underlying graded modules, $X$ and $\Omega^3X$ do not contain $P_2$ as a direct summand. 
Therefore the $P_2$ component in $\widetilde{F}$ must split off.)}
\end{example}


\begin{definition}\label{def gen split}
We say that a partial $\F$-resolution   
$$
\xymatrix@C=32pt{
 & & X_{n-1} \ar[dr]^{q _{n-1}}  & X_{n-2} \ar[dr] & & X_1 \ar[dr]^{q_1} & & & \\ 
0 \ar[r] & X_n  \ar[r]^{q_n} & F_{n-1} \ar[u]_{p_{n-1}} \ar[r]^{f_{n-1}} & F_{n-2} \ar[u]_{p_{n-2}} \ar[r] & \cdots \ar[r] & F_1 \ar[u] _{p_1} \ar[r]^{f_1} & F_0 \ar[r]^{p_0} & X_0 \ar[r] & 0.}  
$$
is {\bf generically split}  if the localized $\K$-exact sequence 
{\small 
$$
\xymatrix@C=20pt{ 
 & & S^{-1}X_{n-1} \ar[dr]^(0.5){S^{-1}q _{n-1}}  & S^{-1}X_{n-2} \ar[dr] & & S^{-1}X_1 \ar[dr]^{S^{-1}q_1} & & & \\ 
0 \ar[r] & S^{-1}X_n  \ar[r]_{S^{-1}q_n} & S^{-1}F_{n-1} \ar[u]^{S^{-1}p_{n-1}} \ar[r]_{S^{-1}f_{n-1}} & S^{-1}F_{n-2} \ar[u] \ar[r] & \cdots \ar[r] &{S^{-1}F_1 \ar[u]}  \ar[r]_{S^{-1}f_1} & S^{-1}F_0 \ar[r]_{S^{-1}p_0} & S^{-1}X_0 \ar[r] & 0}  
$$}
is split in $\Ko (S^{-1}R)$ in the sense of Definition \ref{degenerate}(1),    where   $S = R \backslash \bigcup _{\p \in \Ass (R)} \p$. 
\end{definition}

\begin{example}
Let  $a$ be an element of $R$.  
Then for the complex  $X_1 := [\  0 \to R \overset{a}\longrightarrow  R \to 0\ ]$ there are triangles; 
$$
\xymatrix@C=30pt{
R \ar[r]^a & R \ar[r]  &X_1  \ar[r] & R[1] , \ \ X_1 \ar[r] & R[1]  \ar[r]^a  &R[1]  \ar[r] & X_1[1].  \\
}
$$
Hence $X_0 := R[1]$ has the following type of finite $\F$-resolution: 
$$
\xymatrix@C=30pt{
0 \ar[r] &R \ar[r]^a & R \ar[r]^0 & R[1] \ar[r]^(0.4)a  &X_0= R[1]  \ar[r] &0.    \\
}
$$
If  $a$  is a non-zero divisor then this resolution is generically split. 
However whenever $a$ is a non-unit, the sequence is not split and not degenerate. 
\end{example}

\begin{example}
Let  $a, b \in R$  and assume that  $a, b$  is a regular sequence on $R$. 
Now let 
$$
X_0 = [ \xymatrix@C=30pt{  0 \ar[r] &R^2 \ar[r] ^{(a\  b)}  & R \ar[r] &0} ], \ \ 
X_1 = [ \xymatrix@C=30pt{  0 \ar[r] &R \ar[r] ^{\binom{b}{-a}}  & R^2  \ar[r] &0} ],  
$$
and note that  $X_0$  is *torsion-free but not *reflexive, while $X_1$  is not *torsion-free.  
One can easily see that there is an $\F$-resolution of  $X_0$;  
$$
\xymatrix@C=32pt{ 
&&X_1 \ar[dr] && \\
0 \ar[r] & R \ar[r]^{\binom{b}{-a}} & R^2  \ar[r]_(0.35){\tiny{\begin{pmatrix} 0 &0 \\ a & b \end{pmatrix}}} \ar[u]_{p_1}  & R[1] \oplus R  \ar[r] ^(0.6){p_0} & X_0  \ar[r] & 0,}  
$$
where $p_0, p_1$ are chain maps defined respectively as 
$$
\xymatrix@C=30pt{
 0 \ar[r] & R \ar[r] ^0 \ar[d]_{\binom{b}{-a}} & R \ar[d]^1 \ar[r] & 0 \quad  &0 \ar[r] & 0 \ar[r] \ar[d] & R^2 \ar[d]^1 \ar[r] & 0 \\
 0 \ar[r] &R^2 \ar[r] _{(a\  b)}  & R \ar[r] &0,  &0 \ar[r] &R \ar[r] _{\binom{b}{-a}}  & R^2 \ar[r] &0.} 
$$
It is easy to see that  $p_0$ and $p_1$ are right $\F$-approximations, hence  $\Omega X_0 = X_1$ and $\Omega ^2 X_0 =0$ in $\uK$. 
We should notice that the $\F$-resolution above {may not be} a split sequence, but {may} generically split. 
\end{example}

\vspace{6pt}
\section{Counit morphism for the adjoint pair $(\Sigma^{n}, \Omega^n)$}

It follows from Theorem \ref{adjoint} that there is an isomorphism 
\begin{equation}\label{Omega nn}
\Hom  _{\uK} (\Sigma ^{n-i}\Omega^n X,  \Omega ^i X)  \cong  \Hom _{\uK} (\Omega^n X , \Omega^n X ),
\end{equation}
for all $X \in \K$  and  $0 \leq i \leq n$.
Thus we can take a morphism in $\K$;  
$$
\pi _X^{(n, i)} : \Sigma^{n-i}\Omega^n X \to \Omega ^i X
$$
which yields a unique element of  $\Hom  _{\uK} (\Sigma ^{n-i}\Omega^n X,  \Omega ^i X)$
that corresponds to the identity on $\Omega ^{n}  X$ in the right hand side of (\ref{Omega nn}).  

If  $i=0$, then $\underline{ \pi _X ^{(n, 0)} }  \in \Hom  _{\uK} (\Sigma ^{n}\Omega^n X,  X)$
is a counit morphism for the adjoint pair $(\Sigma ^{n}, \Omega ^{n})$. 
If $i = n$ then  $\underline{ \pi _X ^{(n, n)} }$ is the identity on $\underline{\Omega ^nX}$. 

Adding an $\F$-summand to  $\Sigma^{n-i}\Omega^n X$ if necessary, we may take the morphism  $\pi _X^{(n, i)}$ 
as cohomologically surjective. 
Under such a circumstance, we make a triangle 
$$
\xymatrix@C=32pt{ 
\Delta^{(n, i)} (X)  \ar[r] & \Sigma^{n-i}\Omega^n X \ar[r]^(0.6){\pi_X^{(n, i)}} & \Omega ^i X \ar[r] & \Delta^{(n, i)} (X)[1] \\
}
$$
and define $\Delta^{(n, i)} (X) \in \K$ by this triangle.

Note that there is a short exact sequence of graded $R$-modules; 
$$
\xymatrix@C=32pt{ 
0 \ar[r] & H(\Delta^{(n, i)} (X))  \ar[r]  & H(\Sigma^{n-i}\Omega^n X)  \ar[r]^(0.6){H(\pi_X^{(n, i)})} &  H(\Omega ^i X  ) \ar[r] & 0,  \\
}
$$
for all $X \in \K$  and $0 \leq i \leq n$. 

Since $\underline{\pi_X^{(n, i)}}$  is uniquely determined as a morphism in $\uK$, 
Theorem \ref{cone} implies the following lemma.

\begin{lemma}
For each $X \in \K$ and positive integers  $0 \leq i \leq n$,  the complex $\Delta^{(n, i)}(X)$ defined above  is uniquely determined as an object of  $\uK$.  {\qed}
\end{lemma}

As in \S 9 we have triangles of the form; 

$$
\xymatrix@C=32pt@R=12pt{
\Omega ^{i+1} X \ar[r]^(0.6){q_{i+1}} &  F_i \ar[r]^(0.4){p_i}  &  \Omega ^{i}X \ar[r]  &   \Omega ^{i+1}X[1], \\
\Sigma ^{i}\Omega^nX \ar[r]^(0.6){q ^{n-i}}  &  G_{n-i-1} \ar[r]^(0.4){p ^{n-i-1}} & \Sigma ^{i+1}\Omega ^n X \ar[r] & \Sigma^{i}\Omega^n X[1],  }
$$ 
where $F_i , \ G_{n-i-1} \in \F$ and  $p_i$ (resp.  $q ^{n-i}$) is a right (resp. left)  $\F$-approximation for all $0 \leq i < n$.

{Since each $q^{i+1}$ is a left $\F$-approximation,} setting  as $v _{n}$  the identity morphism  on   $\Omega ^nX$,  by  induction on $n-i$, we find morphisms $v _i :  \Sigma ^{n-i}\Omega ^n X \to \Omega^{i}X$ and  $a_i : G_{i} \to F_i$  that make the following diagrams commutative: 
$$
\xymatrix@C=40pt@R=32pt{
\Omega^{i+1}X \ar[r]^{q _{i+1}}  & F_i \ar[r]^{p _i}  &  \Omega^{i}X \ar[r] &   \Omega^{i+1}X[1] \\  
\Sigma ^{n-i-1}\Omega^nX  \ar[r]^(0.6){q^{i+1}}  \ar[u]^{v_{i+1}} &   G_{i}  \ar[r]^(0.4){p ^{i}} \ar[u]^{a_{i}} &   \Sigma ^{n-i}\Omega ^n X \ar[r]  \ar[u]^{v_i} &   \Sigma^{n-i-1}\Omega^n X[1] \ar[u]^{v _{i+1}[1]}, \\}
$$
for  $0 \leq i < n$. 
Here we can take such $a _i$ to be surjective graded $R$-module homomorphisms. 
Actually, 
{if we add $F_i$ to $G_i$ and $\Sigma ^{n-i}\Omega^nX$ as a direct summand, then the following diagram is also commutative.
$$
\xymatrix@C=40pt@R=32pt{
\Omega^{i+1}X \ar[r]^{q _{i+1}}  & F_i \ar[r]^{p _i}  &  \Omega^{i}X \ar[r] &   \Omega^{i+1}X[1] \\  
\Sigma ^{n-i-1}\Omega^nX  \ar[r]^(0.5){\tiny{\begin{pmatrix} q^{i+1} \\ 0 \end{pmatrix}}}  \ar[u]^{v_{i+1}} &   G_{i}\oplus F_i  \ar[r]^(0.4){\tiny{\begin{pmatrix} p^{i} & 0 \\ 0 &1 \end{pmatrix}}} \ar[u]^{(a_{i}\ 1)} &   \Sigma ^{n-i}\Omega ^n X  \oplus F_i \ar[r]  \ar[u]^{(v_i\  p_i)} &   \Sigma^{n-i-1}\Omega^n X[1] \ar[u]^{v _{i+1}[1]} \\}
$$
Replacing  $a_i$  by $(a_i \ 1)$, we assume that all $a_i$'s are surjective $R$-module homomorphisms.}

Then, since  $p_i a_i$ is cohomologically surjective, we see that  $v_i$ is also cohomologically surjective. 
Therefore we may take all such $v_i$ to be equal to  $\pi_X ^{(n, i)}$  for  $0 \leq i < n$. 
 
Thus we have a commutative diagram in which the rows are $\K$-exact sequences; 

\begin{equation}\label{FGtotal}
\xymatrix@C=36pt@R=32pt{
0 \ar[r] & \Omega^n X \ar[r]^{q_n}  &  F_{n-1} \ar[r]^{f_{n-1}} & \cdots \ar[r]^{f_1} &  F_0 \ar[r]^{p_0} & X \ar[r] & 0 \\ 
0 \ar[r] &  \Omega^n X \ar[r]^{q^{n}} \ar@{=}[u] & G_{n-1} \ar[r]^{g^{n-1}}  \ar[u]^{a _{n-1}} &  \cdots \ar[r]^{g^{1}}  &  G_{0} \ar[r]^(0.4){p^{0}} \ar[u]^{a _{0}} &  \Sigma^{n}\Omega^nX  \ar[r] \ar[u]^{\pi_X^{(n, 0)}} & 0 , \\
}
\end{equation}
where $F_i {, G_i} \in \F$,  $f_i = q _i p _i$ and  $g^{i} = q ^{i} p ^{i}$  for $1 \leq i < n$. 
This diagram is divided into two  commutative diagrams whose rows are $\K$-exact sequences as well: 

\begin{equation}\label{FGright}
\xymatrix@C=36pt@R=32pt{
0 \ar[r] & \Omega^i X \ar[r]^{q_i}  &  F_{i-1} \ar[r]^{f_{i-1}} & \cdots \ar[r]^{f_1} &  F_0 \ar[r]^{p_0} & X \ar[r] & 0 \\ 
0 \ar[r] &  \Sigma^{n-i}\Omega^n X \ar[r]^{q^{i}} \ar[u]^{\pi_X^{(n, i)}} & G_{i-1} \ar[r]^{g^{i-1}}  \ar[u]^{a _{i-1}} &  \cdots \ar[r]^{g^{1}}  &  G_{0} \ar[r]^(0.4){p^{0}} \ar[u]^{a _{0}} &  \Sigma^{n}\Omega^nX  \ar[r] \ar[u]^{\pi_X^{(n, 0)}} & 0 , \\
}
\end{equation}

\begin{equation}\label{FGleft}
\xymatrix@C=36pt@R=32pt{
0 \ar[r] & \Omega^n X \ar[r]^{q_n}  &  F_{n-1} \ar[r]^{f_{n-1}} & \cdots \ar[r]^{f_{i+1}} &  F_i \ar[r]^{p_i} & \Omega ^i X \ar[r] & 0 \\ 
0 \ar[r] &  \Omega^n X \ar[r]^{q^{n}} \ar@{=}[u] & G_{n-1} \ar[r]^{g^{n-1}}  \ar[u]^{{a _{n-1}}} &  \cdots \ar[r]^{g^{i+1}}  &  G_{i} \ar[r]^(0.4){p^{i}} \ar[u]^{a _{i}} &  \Sigma^{n-i}\Omega^nX  \ar[r] \ar[u]^{\pi_X^{(n, i)}} & 0 , \\
}
\end{equation}

Now set  $L_i = \Ker \ a_i$  the kernel as a graded $R$-module homomorphism for $0 {<}  i \leq n$. 
Since  each $a_i$ is surjective as a graded $R$-module homomorphism, we  see that  $L_i \in \F$.  
Then the successive use of octahedron axiom to the diagram (\ref{FGleft}) will show that  there is a commutative diagram whose columns are triangles and rows are $\K$-exact sequences:
\begin{equation}\label{FGL}
\xymatrix@C=32pt@R=28pt{
0 \ar[r] & \Omega^n X \ar[r]^{q_n} & F_{n-1} \ar[r]^{f_{n-1}} &  \cdots  \ar[r]^{f_{i+1}} & F_i \ar[r]^{p_i} & \Omega ^i X \ar[r] & 0 \\ 
0 \ar[r]  &  \Omega^n X \ar@{=}[u] \ar[r]^{q^{n}}  &  G_{n-1} \ar[u]^{a_{n-1}} \ar[r]^{g^{n-1}} &  \cdots \ar[r]^{g^{1}} & G_{i} \ar[u]^{a_{i}} \ar[r]^(0.4){p^{i}}  & \Sigma^{n-i}\Omega^nX \ar[u]^{\pi_X^{(n, i)}}  \ar[r] & 0  \\
  & 0 \ar[r] & L_{n-1}  \ar[u]^{b^{n-1}} \ar[r]^{\ell ^{n-1}}  &  \cdots \ar[r]^{\ell ^{i+1}}  &  L_{i} \ar[u]^{b^{i}} \ar[r] &  \Delta^{(n, i)} (X) \ar[u] \ar[r] & 0 \\ 
}
\end{equation}

In fact we prove by induction on $n-i$ that the third row of the diagram (\ref{FGL}) is a $\K$-exact sequence. 
If $n-i=1$  then the following octahedron diagram proves this. 
$$
\xymatrix@C=32pt@R=28pt{
   & L_{n-1} \ar[r]^(0.4){{\cong}}   &  \Delta^{(n, n-1)} (X)  &  \\ 
 \Omega ^n X \ar[r]^{q_n} & F_{n-1} \ar[u] \ar[r]^{p_{n-1}} & \Omega ^{n-1} X \ar[u] \ar[r] &  \Omega^{n} X  {[1]} \ar@{=}[d] \\ 
 \Omega ^n X \ar@{=}[u] \ar[r]^{q^{n}}  &  G_{n-1} \ar[u]^{a_{n-1}} \ar[r]^(0.5){p^{n-1}}  & \Sigma \Omega ^n X \ar[u]_{\pi_X^{(n, n-1)}}  \ar[r] &  \Omega ^n  X {[1]} \\
   & L_{n-1} \ar[u]^{b^{n-1}} \ar[r]^(0.4){{\cong}} &  \Delta^{(n, n-1)} (X) \ar[u]  &   }
$$
If  $n - i \geq 2$,  then applying the induction hypothesis, we see that in the diagram 
$$
\xymatrix@C=32pt@R=28pt{
0 \ar[r] & \Omega^n X \ar[r]^{q_n} & F_{n-1} \ar[r]^{f_{n-1}} &  \cdots  \ar[r]^{f_{i+2}} & F_{i+1} \ar[r]^{p_{i+1}} & \Omega ^{i+1} X \ar[r] & 0 \\ 
0 \ar[r]  &  \Omega^n X \ar@{=}[u] \ar[r]^{q^{n}}  &  G_{n-1} \ar[u]^{a_{n-1}} \ar[r]^{g^{n-1}} &  \cdots \ar[r]^{g^{i+2}} & G_{i+1} \ar[u]^{a_{i+1}} \ar[r]^(0.4){p^{i+1}}  & \Sigma^{n-i-1}\Omega^nX \ar[u]^{\pi_X ^{(n, i+1)}} \ar[r] & 0  \\
  & 0 \ar[r] & L_{n-1}  \ar[u]^{b^{n-1}} \ar[r]^{\ell ^{n-1}}  &  \cdots \ar[r]^{\ell ^{i+2}}  &  L_{i+1} \ar[u]^{b^{i+1}} \ar[r] &  \Delta^{(n, i+1)}(X) \ar[u] \ar[r] & 0,  \\ 
}$$
the third row is $\K$-exact. 
On the other hand, there is a commutative diagram where all rows and columns are triangles: 
$$
\xymatrix@C=32pt@R=28pt{
\Omega ^{i+1} X  \ar[r]^{q_{i+1}} & F_i \ar[r]^{p_i} & \Omega ^{i}X \\
\Sigma^{n-i-1}\Omega^nX \ar[r]^(0.7){q^{i+1}} \ar[u]^{\pi_X^{(n, i+1)}} & G_i \ar[r]^(0.4){p^i} \ar[u]^{a_i} & \Sigma^{n-i}\Omega^nX \ar[u] ^{\pi^{(n, i)}_X}\\
\Delta^{(n, i+1)}(X)  \ar[r] \ar[u] & L_{i} \ar[r] \ar[u]^{b_i} & \Delta^{(n, i)}(X) \ar[u]  \\
}
$$
(First we take the top left square with $a_i$ being surjective, then apply the $9$ lemma {(cf. \cite[Exercise 10.2.6, p.378]{W})} to get this commutative diagram.)

In particular we have a $\K$-exact sequence 
$$\xymatrix@C=32pt{
0 \ar[r] & \Delta^{(n, i+1)}(X)  \ar[r] & L_{i} \ar[r]  & \Delta^{(n, i)}(X) \ar[r] & 0  \\
}$$
Combining the sequences above we finally obtain the $\K$-exact sequence:  
\begin{equation}\label{Lsequence}
\xymatrix@C=32pt@R=28pt{
0 \ar[r] & L_{n-1}  \ar[r]^{\ell ^{n-1}}  & L_{n-2} \ar[r]  &  \cdots \ar[r]^{\ell ^{i+1}}  &  L_{i} \ar[r] &  \Delta^{(n, i)} (X)  \ar[r] & 0 \\ 
}
\end{equation}
This proves that all the rows in  the  diagram (\ref{FGL}) are $\K$-exact sequences. 
 
Letting  $\widetilde{L^{(n, i)}}$  be the contraction of the $\F$-resolution (\ref{Lsequence}), we have the isomorphism  $  \Delta^{(n, i)} (X)  \cong \widetilde{L^{(n, i)}}$. 
We have thus proved the following theorem. 

\begin{theorem}\label{Delta has finite res}
Let  $X \in \K$  and $0 \leq i \leq n$. 
Then  $\Delta ^{(n, i)}(X)$  has a finite $\F$-resolution of length  $n-i-1$. 
In this case  $\Delta ^{(n, i)}(X)$ is isomorphic in $\K$ to the contraction of such a finite $\F$-resolution. 
\end{theorem}

\begin{remark}
{　}
\begin{enumerate} 
\item  If $R$  is a Gorenstein ring of dimension zero (i.e. a self-injective algebra), then we can take all the $a_i$  {to be} isomorphisms and hence one can take  $L_i=0$  for all $0 \leq i < n$. 
Thus we have  $\underline{\Delta ^{(n, i)} (X)} =0$ or $ \Delta ^{(n, i)}(X) \in \F$ for all $0 \leq i \leq n$. 
Moreover for any choice of  $a_i$  the sequence (\ref{Lsequence}) is a split sequence in this case. 
{See Corollary \ref{Gor dim 0} and Remark \ref{remark on omega}.} 
\item  
 Let  $S$  be a multiplicatively closed subset of  $R$. 
Then localizing by $S$ and going through the whole procedure again leads to the localization of the sequence (\ref{Lsequence}).
As a consequence,  we observe an isomorphism 
$$
S^{-1}\Delta _R^{(n, i)} (X)   \cong  \Delta ^{(n, i)} _{S^{-1}R}(S^{-1}X), 
$$
in the stable category  $\underline{\Ko (S^{-1}R)}$ for all $0 \leq i < n$, 
 and the localized sequence of  (\ref{Lsequence}) by  $S$  is an $\mathrm{Add}(S^{-1}R)$-resolution of $\Delta ^{(n, i)} _{S^{-1}R}(S^{-1}X)$. 
\end{enumerate} 
\end{remark}

By this remark, if $R$ is a generically Gorenstein ring, then the $\F$-resolution (\ref{Lsequence}) is generically split. 
Thus we have proved the following theorem. 

\begin{theorem}\label{gen split resolution}
Let  $R$  be a generically Gorenstein ring. 
For any  $X \in \K$  and $0 \leq i \leq n$,  
 $\Delta ^{(n, i)}(X)$  has a finite $\F$-resolution of length  $n-i-1$ that is generically split. 
\end{theorem}

\vspace{6pt}
\section{The main theorem and the proof }

The following {lemma} is one of the most essential observations to prove the main theorem. 
 
\begin{lemma}\label{main lemma to prove theorem}
Let    $R$  be a generically Gorenstein ring, and  let $X \in \K $. 
If  $H ( X^{*} ) = 0$, then  $\Omega ^r X$  is *torsion-free for each non-negative integer $r$.  
\end{lemma}

\medskip
To prove {this lemma} we prepare several {preliminary} lemmas. 

\begin{lemma}\label{perp0}
Let  $X$  be a complex in $\K$ and assume that  $H(X^*)=0$. 
Then we have $\Hom _{\K} (X, F)=0$ for all $F \in \F$.
\end{lemma}

\begin{proof}
Note that  $H(X^*)$  is the cohomology module of  the complex $\Hom _R(X, R)$, hence we have the equality  $H(X^*)= \bigoplus _{i \in \Z} \Hom _{\K}(X, R[i])$. 
Thus if   $H(X^*)=0$,  then we see that  $\Hom _{\K}(X, P[i])=0$  for any finitely generated projective $R$-module $P$ and an integer  $i$.
Recall from Theorem \ref{characterization of F} and Proposition \ref{coproduct} that any complex  $F  \in \F$  is  isomorphic to a direct sum  $\bigoplus_{i \in \Z} F^i [-i]$ with  $F^i$  being a {finitely generated} projective $R$-module for each $i \in \Z$.  
On the other hand it follows from Lemma \ref{sum-product lemma} the direct sum is a product in $\K$. 
Hence $\Hom _{\K} (X, F) = \prod _{i \in \Z} \Hom _{\K} (X, F^i [-i]) =0$ as desired.
\end{proof}

\begin{lemma}\label{perp}
Let  $X, Y \in \K$. 
Assume the following conditions: 
\begin{enumerate}
\item
$Y$  has an $\F$-resolution of finite length. \medskip
\item
$H(X^*)=0$. \medskip
\end{enumerate}

Then we have $\Hom _{\K} (X, Y)=0$. 
\end{lemma}

\begin{proof}
This is obvious from the previous lemma and utilizing the induction on the length $\ell$ of the $\F$-resolution of  $Y$. 
In fact, if  $\ell =0$ then $Y \in \F$ hence  $\Hom _{\K} (X, Y)=0$  by Lemma \ref{perp0}. 
If  $\ell >0$ then there is a triangle  $$\xymatrix{Y'  \ar[r] & F_0 \ar[r] & Y \ar[r] & Y'[1]},$$ where  $F_0 \in \F$ and $Y'$  has an $\F$-resolution of length $\ell -1$. 
Thus $\Hom _{\K} (X, Y'[i]) =0$ for all $i \in \Z$  by the induction hypothesis. 
Since there is an exact sequence of $R$-modules;  
$$
\xymatrix{\Hom _{\K}(X, F_0) \ar[r] & \Hom _{\K}(X, Y) \ar[r] & \Hom _{\K}(X, Y'[1]),}  
$$
which results that  $\Hom _{\K}(X, Y)=0$. 
\end{proof}

\medskip 
Now we proceed to the proof of {Lemma} \ref{main lemma to prove theorem}. 
 
In  this proof we assume that  $R$ is generically Gorenstein and  $H(X^*)=0$. 
It is clear that  $X$ is *torsion-free, since $H(X^*) =0 \to H(X)^*$ is injective. 
We shall prove  that  so is $\Omega ^rX$ for  $r \geq 1$.  
 
Let  $n\geq 2$  be an integer. 
We have the following commutative diagram from (\ref{FGL}); 

\begin{equation}\label{FGL0}
\xymatrix@C=32pt@R=28pt{
0 \ar[r] & \Omega^n X \ar[r]^{q_n} & F_{n-1} \ar[r]^{f_{n-1}} &  \cdots  \ar[r]^{f_{1}} & {F_0} \ar[r]^{p_0} & X \ar[r] & 0 \\ 
0 \ar[r]  &  \Omega^n X \ar@{=}[u] \ar[r]^{q^{n}}  &  G_{n-1} \ar[u]^{a_{n-1}} \ar[r]^{g^{n-1}} &  \cdots \ar[r]^{g^{1}} & {G_{0}} \ar[u]^{a_{0}} \ar[r]^(0.4){p^{0}}  & \Sigma^{n}\Omega^nX \ar[u]^{\pi_X^{(n, 0)}}  \ar[r] & 0  \\
  & 0 \ar[r] & L_{n-1}  \ar[u]^{b^{n-1}} \ar[r]^{\ell ^{n-1}}  &  \cdots \ar[r]^{\ell ^{1}}  &  L_{0} \ar[u]^{b^{0}} \ar[r] &  \Delta^{(n, 0)} (X) \ar[u] \ar[r] & 0,  \\ 
}
\end{equation}
where the rows are $\K$-exact sequences and the columns are  triangles.  
Taking the contracted triangles of the rows we obtain the following commutative diagram whose rows and columns are triangles:  

\begin{equation}\label{FGLtilder}
\xymatrix@C=32pt@R=32pt{
 & \widetilde{L}[1]   \ar[r] ^(0.4){\cong} &  \Delta^{(n, 0)} (X)[1]  &   \\ 
\Omega^n X [n-1] \ar[r]^(0.6){\psi^F _n} & \widetilde{F}  \ar[u]^{\lambda}  \ar[r]^(0.4){\varphi _n ^F} & X \ar[u]^{\sigma}  \ar[r]^{\widetilde{\omega_n}^F} & \Omega ^nX [n] \\ 
\Omega^n X [n-1] \ar@{=}[u] \ar[r]^(0.6){\psi ^G_{n}}  &  \widetilde{G} \ar[u]^{\widetilde{a}} \ar[r]^(0.4){\varphi _n ^G}  & \Sigma^{n}\Omega^nX  \ar[u]^{\pi _X^{(n, 0)} } \ar[r]^{\widetilde{\omega_n}^G} & \Omega ^nX [n] \ar@{=}[u] \\ 
 & \widetilde{L} \ar[u]^{\widetilde{b}} \ar[r] ^(0.4){\cong} &  \Delta^{(n, 0)} (X),  \ar[u]^{\tau}  &   
}
\end{equation}
where $\widetilde{a}$ and $\widetilde{b}$ are the induced {morphisms} from $\{ a_i \}$  and  $\{ b_i \}$ respectively. 
See Definition \ref{morphism of resolutions} and Theorem \ref{lower triangle matrix}. 
(In fact, to see the  second column is a triangle, note all the other columns and rows are triangles and apply the octahedron axiom.)
We know from Theorem \ref{Delta has finite res} that  $\Delta^{(n, 0)} (X)[1]$  has a finite $\F$-resolution. 
Hence it follows from Lemma \ref{perp}  that $\sigma$  in the diagram is zero. 
Thus $\lambda$  is also zero in the diagram by the commutativity of the upper square. 
This means that the second and the third columns are split triangles, hence  $\widetilde{a}$  and  $\pi _X^{(n, 0)}$  have right inverses. 
Notice from this that   $\widetilde{L}$,   and hence  $\Delta^{(n, 0)} (X)$ as well,  is  *torsion-free, since it is a direct summand of  $\Sigma^{n}\Omega ^nX$.  
{See Lemma \ref{torsion-free} and Theorem \ref{reflexive cor}.} 

{Taking all $F_i$'s and $G_i$'s to have no null complex as direct summands, we note} that the following diagram is commutative (cf. Theorem \ref{lower triangle matrix}).

\begin{equation}\label{FGLproj}
\xymatrix@C=48pt@R=32pt{
 \widetilde{F}  \ar[r]^(0.4){pr_{n-1} ^F} & F_{n-1} [n-1] \\ 
 \widetilde{G}  \ar[u]^{\widetilde{a}} \ar[r]^(0.4){pr_{n-1} ^G} & G_{n-1} [n-1] \ar[u]_{a_{n-1}[n-1]}  \\ 
 \widetilde{L}  \ar[u]^{\widetilde{b}} \ar[r]^(0.4){pr_{n-1} ^L} & L_{n-1} [n-1] \ar[u]_{b_{n-1}[n-1]} \\ 
}
\end{equation}
 
We shall now prove that  $pr _{n-1} ^L =0$   in $\K$. 

To prove this we note {from Theorem \ref{gen split resolution}} that  $\widetilde{L}$  has an $\F$-resolution of the form (\ref{Lsequence}) that is generically split. 
Therefore  we see from Corollary \ref{split pr} that $S^{-1}   (pr_{n-1} ^L) =0$ in ${\mathscr{K}} (S^{-1}R)$, where  $S$  is the set of all non-zero divisors in $R$ as before.   
Since  $\widetilde{L}$ is *torsion-free and $L_{n-1} \in \F$, it follows from Theorem \ref{S^{-1}f=0} that $pr _{n-1} ^L =0$   in $\K$ as desired. 

Then we have   $pr_{n-1} ^G  \widetilde{b} =0$ by the commutativity of the diagram (\ref{FGLproj}).  
Hence there is a morphism $e : \widetilde{F}  \to  G_{n-1}[n-1] $  in $\K$ such that  $e \ \widetilde{a} = pr_{n-1} ^G$. 
  
Let  $\widetilde{F'}$  and  $\widetilde{G'}$  be the contractions of the partial $\F$-resolutions {appearing} in the following diagram, thus $\widetilde{F'}  = F_{n-2} [n-2] \oplus \cdots \oplus F_0 \subseteq  \widetilde{F} =F_{n-1}[n-1] \oplus F_{n-2}[n-2] \oplus \cdots \oplus F_0$ and the same for $\widetilde{G'}$. 

\begin{equation}\label{FGn-1}
\xymatrix@C=32pt@R=28pt{
0 \ar[r] & \Omega^{n-1} X \ar[r]^{q_{n-1}} & F_{n-2} \ar[r]^{f_{n-2}} &  \cdots  \ar[r]^{f_{1}} & F_0 \ar[r]^{p_0} & X \ar[r] & 0 \\ 
0 \ar[r]  &  \Sigma \Omega^{n} X \ar[u]^{\pi_X^{(n, n-1)}}  \ar[r]^{q^{n-1}}  &  G_{n-2} \ar[u]^{a_{n-2}} \ar[r]^{g^{n-2}} &  \cdots \ar[r]^{g^{1}} & G_{0} \ar[u]^{a_{0}} \ar[r]^(0.4){p^{0}}  & \Sigma^{n}\Omega^nX \ar[u]^{\pi_X^{(n, 0)}}  \ar[r] & 0  \\
}
\end{equation}
We notice that $\widetilde{F'}$ and $\widetilde{G'}$ are subcomplexes  of  $\widetilde{F}$ and $\widetilde{G}$ respectively. 
Recall that $\widetilde{a}$ is a splitting epimorphism in $\K$ and it is represented by a lower {triangular} matrix whose diagonal entries are  $a_{n-1}, \cdots, a_0$  which are all split epimorphisms of graded $R$-modules. 
Thus  we have  that $\widetilde{F'} = \widetilde{a} (\widetilde{G'})$. 

There is a diagram whose rows are triangles and squares are commutative; 

\begin{equation}\label{FGproj}
\xymatrix@C=48pt@R=32pt{
\widetilde{F'} \ar[r] &  \widetilde{F}  \ar[r]^(0.4){pr_{n-1} ^F} \ar[dr]^e & F_{n-1} [n-1] \\ 
\widetilde{G'} \ar[r] \ar[u]^{\widetilde{a'}} & \widetilde{G}  \ar[u]^{\widetilde{a}} \ar[r]^(0.4){pr_{n-1} ^G} & G_{n-1} [n-1] \ar[u]_{a_{n-1}[n-1]}  \\ 
 }
\end{equation}
    
Since   $pr _{n-1}^G ( \widetilde{G'}) =0$, we have 
$e (\widetilde{F'}) = e \widetilde{a} (\widetilde{G'})  = pr _{n-1}^G ( \widetilde{G'}) = 0$. 
Thus $e$  induces a morphism  $f : F_{n-1}[n-1] \to  G_{n-1}[n-1]$  such that  $e = f \ pr_{n-1}^F$. 
Hence it holds that  $f \ pr_{n-1} ^F \widetilde{a} = pr _{n-1} ^G$. 

Now recalling in the diagram  (\ref{FGL0})  that  $q_n [n-1] = pr_{n-1}^F \psi _{n}^F$  and   
$q^n [n-1] = pr_{n-1}^G \psi _{n}^G$  by Theorem and Definition \ref{theorem/definition},     
we have equalities;  
$$
q^n [n-1] = pr _{n-1} ^G \psi _{n}^G  
= f \ pr_{n-1}^F \widetilde{a} \ \psi _n ^G 
=f \ pr_{n-1}^F \psi _n ^F
= f \ q_n [n-1], 
$$ 
{where we use the commutative diagram (\ref{FGLtilder}) for the third equality.} 

This shows the commutativity of the following diagram in which the rows are triangles: 
$$
\xymatrix@C=48pt@R=32pt{
\Omega ^n X [n-1] \ar@{=}[d] \ar[r]^{q_n[n-1]} & F_{n-1}[n-1] \ar[d]^f  \ar[r] & \Omega ^{n-1}X [n-1] \\ 
\Omega ^n X [n-1] \ar[r]^{q^n[n-1]} & G_{n-1}[n-1] \ar[r] & \Sigma \Omega ^n X [n-1] \\ 
}
$$

Recall that  $q^n[n-1]$ is a left $\F$-approximation. 
Then it is easy to see from Remark \ref{remark approximation} that  $q_n[n-1]$ is also a left $\F$-approximation. 
As a consequence of this we have   $\Omega ^{n-1} X [n-1] \cong  \Sigma \Omega^n X [n-1] $  in  $\uK$   
{(cf. Theorem \ref{cone})}. 
Thus  $\Omega ^{n-1} X $  is *torsion-free by Theorem \ref{torsion-free2}. 
Since  $n$  is any integer not less than $2$,  this completes the proof of {Lemma} \ref{main lemma to prove theorem}. 
\qed 

\medskip 

{Lemma} \ref{main lemma to prove theorem} can be strengthened as in the following form. 

\begin{theorem}\label{Omega reflexive}
Let  $R$   be  a generically Gorenstein ring. 
Assume  $H(X^*)=0$ for  $X \in \K$. 
Then  $\Omega ^{r} X$  is *reflexive for any non-negative integer $r$. 
\end{theorem}

\begin{proof}
Let  $r$  be a non-negative integer. 
Note from the definition that there is a triangle  
$$
\xymatrix@C=48pt@R=32pt{
\Omega ^{r+1} X \ar[r]^q & F_r  \ar[r]^{p}  & \Omega ^r X \ar[r] & \Omega ^{r+1} X[1],  \\
}
$$
where  $p$  is a right  $\F$-approximation. 
Hence the sequence  of graded $R$-modules;  
$$
\xymatrix@C=48pt@R=32pt{
0 \ar[r] & H(\Omega ^{r+1} X) \ar[r]^{H(q)} & H(F_r)  \ar[r]^{H(p)}  & H(\Omega ^r X) \ar[r] & 0  \\
}
$$
is exact. 
Given a graded $R$-module homomorphism  $\alpha : H (\Omega ^rX) \to R[i]$ for some $i \in \Z$, we find a morphism  $b : F_r \to R[i]$ in $\K$ with $H(b) = \alpha H(p)$,  since  $F_r$  is *reflexive. 
Then  we have $H(bq) = H(b) H(q) = \alpha H(pq)=0$. 
As we have shown in {Lemma} \ref{main lemma to prove theorem},  $\Omega ^{r+1}X$  is *torsion-free, we have that $bq =0$. 
{(See Lemma \ref{restatement} and also Theorem \ref{characterization of F}.)}
Then  there is an $a : \Omega ^r X \to R[i]$  that satisfies  $b= ap$.        
Since  $H(p)$  is a surjection, it thus follows that  $H(a) = \alpha$. 
Then one can apply Lemma \ref{restatement}(2) to conclude that  $\Omega ^{r}X$  is *reflexive. 
\end{proof}

\begin{proposition}\label{ext vanish}
Let   $Y \in \K$. 
Assume the following conditions are satisfied: 
\begin{enumerate}
\item 
$Y$  is *torsion-free. \medskip 
\item
$\Omega Y$  is *reflexive. 
\end{enumerate}  
Then we have  $\Ext_R^1(H(Y), R) =0$.  
\end{proposition}

\begin{proof}
From the definition of $\Omega Y$  there is a triangle in $\K$; 
$$
\xymatrix@C=32pt{
\Omega Y \ar[r]^q & F \ar[r]^{p} & Y \ar[r]^(0.4){\omega} &\Omega Y [1], 
}
$$
where  $p$  is a right $\F$-approximation. 
Hence there is an exact sequence of graded $R$-modules; 
$$
\xymatrix@C=32pt{
0 \ar[r] & H(\Omega Y) \ar[r]^{H(q)} & H(F) \ar[r]^{H(p)} & H(Y) \ar[r] & 0,  
}
$$
where  $H(F)$  is a projective graded $R$-module. 
Thus we have an exact sequence 
$$
\xymatrix@C=32pt{
0 \ar[r] & H(Y)^* \ar[r]^{H(p)^*}  & H(F)^* \ar[r]^{H(q)^*} & H(\Omega Y)^*  \ar[r] & \Ext _R ^1(H(Y), R) \ar[r] & 0.  
}
$$
On the other hand we also have a triangle; 
$$
\xymatrix@C=32pt{
Y ^* \ar[r]^{p^*}  & F ^* \ar[r]^{q^*}  & (\Omega Y)^* \ar[r]^{\omega^*[1]} & Y^* [1]. 
}
$$
Therefore we have the following commutative diagram of $R$-modules whose rows are exact sequences of graded $R$-modules: 
{\small $$
\xymatrix@C=30pt@R=36pt{
H((\Omega Y)^*)[-1] \ar[r]^(0.61){H(\omega^*)}  & H(Y ^*) \ar[r]^{H(p^*)} \ar[d]^{\rho _{Y R}}  & H(F ^*) \ar[r]^{H(q^*)}  \ar[d]^{=} & H((\Omega Y)^*) \ar[r]^{H(\omega^*[1])}  \ar[d]^{\rho_{\Omega Y R}} & H(Y^* [1]) \\ 
0 \ar[r] & H(Y)^* \ar[r]^{H(p)^*}  & H(F)^* \ar[r]^{H(q)^*} & H(\Omega Y)^*  \ar[r] & \Ext _R ^1(H(Y), R) \to 0.  
}
$$}
Since  $Y$ is *torsion-free, $\rho _{Y R}$ is injective and hence  so is $H(p^*)$. 
It thus follows that  $H(\omega ^*)=0$. 
Then we must have  that  $H(q^*)$  is surjective. 
Since  $\Omega Y$  is *reflexive,  $\rho _{\Omega Y R}$  is bijective. 
As a result we have that  $H(q)^*$  is surjective as well as $H(q^*)$. 
Thus it is concluded from the exactness of the second row that  $\Ext _R ^1(H(Y), R) =0$. 
\end{proof}

Combining Proposition \ref{Omega reflexive} with Proposition \ref{ext vanish}, we obtain the following proposition that is a key for the proof of Theorem \ref{last main theorem}.

\begin{proposition}\label{ext cor}
Let  $R$  be a generically Gorenstein ring, and assume  that  $H(X^*) =0$  for  $X \in \K$. 
Then we have   $$\Ext _R ^r (H(X), R) =0,$$  for  all $r>0$. 
\end{proposition}

\begin{proof}
Recall from the argument after Theorem \ref{omega triangle} that one can take a partial $\F$-resolution of $X$     
$$
\xymatrix@C=32pt{ 
0 \ar[r] &  \Omega ^r X  \ar[r]^{q _r} &  F_{r-1} \ar[r]^{f_{r-1}} & \cdots \ar[r] & F_1 \ar[r]^{f_1} &  F_0 \ar[r]^{p_0}  & X \ar[r]  & 0, } 
$$
{where each $F_i$'s contains no null complex as a direct summand. 
Then} it induces an exact sequence of graded $R$-modules
{\small $$
\xymatrix@C=26pt{ 
0 \ar[r] & H( \Omega ^r X )  \ar[r]^(0.45){H(q _r)} &  H(F_{r-1}) \ar[r]^(0.6){H(f_{r-1})} & \cdots \ar[r] & H(F_1) \ar[r]^(0.45){H(f_1)} &  H(F_0) \ar[r]^(0.45){H(p_0)}  & H( X ) \ar[r]  & 0,  } 
$$}
{with  $H(F_i) = F_i$  for all $0 \leq i \leq r-1$. }
Since $R$  is generically Gorenstein and  $H(X^*)=0$, it follows from Theorem \ref{Omega reflexive} that  $\Omega ^r X$  is *reflexive for each $r>0$. 
Note also that  $X$ is *torsion-free {by Lemma \ref{main lemma to prove theorem}}.  
Then,  by Proposition \ref{ext vanish}, we have  $\Ext _R ^1 (H(\Omega ^{r-1}X), R) = 0$  for all $r>0$.  
Thus it follows from the long exact sequence above of graded $R$-modules that  $\Ext _R ^r (H(X), R) =0$ for $r>0$. 
\end{proof}

The following is the main theorem of this paper, which we can now prove as a result of the previous theorems and propositions. 

\begin{theorem}\label{last main theorem}
Let    $R$  be a generically Gorenstein ring, and  let $X \in \K $. 
Then,  $H ( X ) = 0$  if and only if  $H ( X^{*} ) = 0$.  
\end{theorem}

\begin{proof}
We have only to prove that  if   $H ( X^* ) = 0$  then $H ( X ) = 0$ under the assumption that  $R$  is generically Gorenstein.  
The other implication follows from this by the duality $X^{**} \cong X$. 
Thus in this proof we assume that  $H(X^*)=0$ and our aim is to show that $H(X)=0$.   

\medskip  
\noindent 
(1st step): 
We may assume that   $(R, \m, k)$  is a local ring, which is generically Gorenstein. 
Furthermore we may assume that $\dim R >0$. 

\medskip 
Note that  $H(X)=0$ if and only if $H(X_{\m})=H(X)_{\m}=0$ for all maximal {ideals} $\m$ of $R$, and that $\Hom_R(X, R)_{\m} = \Hom _{R_{\m}}(X_{\m}, R_{\m})$.  
It is also obvious that if $R$ is generically Gorenstein, then so are all $R_{\m}$. 
The first half is clear from these observations. 

If $\dim R =0$ then $R$ is a Gorenstein ring by the generic Gorenstein assumption and the theorem is trivial in this case, since $R$ is an injective $R$-module. 
Hence we may avoid this case. 

\medskip 
\noindent 
(2nd step): 
We may assume that  $\m H ( X ) =0$. 

To show this, let $x \in \m$ and consider the Koszul complex  $$K(x) = \left[ \xymatrix{0 \ar[r] &  R \ar[r]^x & R \ar[r] & 0}\right].$$ 
Set $X' = X \otimes _R K(x)$, and since there is a triangle  
$\xymatrix{X \ar[r]^x & X \ar[r] & X' \ar[r] & X[1]}$ in $\K$, we have the equivalence  $H(X)=0 \ \Leftrightarrow  \ H(X')=0$ by Nakayama Lemma. 
Taking the dual of the triangle above, we have a triangle 
$\xymatrix{(X')^* \ar[r] & X^* \ar[r]^x & X^* \ar[r] & (X')^*[1]}$, thus there is an isomorphism $(X')^* [1]\cong X^* \otimes _R K(x)$. 
Therefore we see that $H(X^*)=0 \ \Leftrightarrow  \ H({X'}^*)=0$ as well. 
 
 Now take a generating set $x_{1}, \ldots , x_{m}$ of the maximal ideal $\m$, and consider the Koszul complex 
  $X'' = X \otimes _{R} K(x_{1}, \ldots , x_{m}) = X \otimes _R K(x_1) \otimes _R \cdots \otimes _R K(x_m)$. 
 Then we have the equivalences  $H(X'') = 0 \ \Leftrightarrow  \ H( X ) = 0$, and also  $H( {X ''} ^{*} ) =0 \ \Leftrightarrow  \ H( X^{*}) =0$. 
Thus it is enough to show that  $H({X''}^*) = 0$ implies $H(X'') =0$.   
It is also clear that  for any element  $x \in \m$, the multiplication map on $X''$ is trivial in $\K$, hence  $\m H(X'') =0$. 

\medskip 
\noindent 
(3rd step): 
Now assume that $H(X) \not= 0$. 
Then there is an integer  $i$  with  $H^i (X) \not=0$. 
By the second step of this proof,  $H^i(X)$  is a non-trivial $k$-module, where  $k=R/\m$. 
On the other hand we have shown in Proposition \ref{ext cor} that  $\Ext _R ^r (H(X), R) =0$  for all $r >0$ under the condition that   $H(X^*)=0$. 
Therefore we have 
$$
\Ext_R^r (k, R) = 0  \quad \text{for all}\ \ r > 0.  
$$
This requires that  $R$  is a Gorenstein ring of dimension zero, which is not the case by the first step. 
Hence  $H(X)=0$  and the proof of the theorem is completed. 
 \end{proof}

\vspace{6pt}
\section{Applications}

Recall that a chain homomorphism $f$  between complexes is called a quasi-isomorphism if the cohomology mapping  $H(f)$  is an isomorphism of modules. 

\begin{theorem}[Corollary \ref{2}]\label{2'}
Assume that the ring  $R$ is a generically Gorenstein ring. 
Let  $f : X \to Y$  be a morphism in $\K$. 
Then, $f$  is a quasi-isomorphism if and only if  the $R$-dual  $f^* : Y^* \to X^*$  is  a quasi-isomorphism. 
\end{theorem}

\begin{proof}
In fact, let  $\xymatrix{X \ar[r]^{f} & Y \ar[r] & Z \ar[r] & X[1]}$ be a triangle in $\K$. 
Then $f$  is a quasi-isomorphism if and only if  $H(Z)=0$. 
We have shown in Theorem \ref{last main theorem}   that  $H(Z) = 0$ if and only if  $H(Z^*)=0$.  
Since $\xymatrix{Z^*[-1] \ar[r] & Y^* \ar[r]^{f^*} & X^* \ar[r] & Z^*}$ is a triangle,  Theorem \ref{2'} follows. 
\end{proof}

Now we recall the definition of totally reflexive modules. 
A finitely generated module $M$ over a commutative Noetherian ring  $R$  is called 
a totally reflexive module or a module of G-dimension zero if 
$\Ext _R ^i(M, R) = \Ext _R^i (\transpose M, R ) =0$ for all  $i>0$. 
{See \cite[(3.8)]{AB} or \cite{JS}. }
This is equivalent to the  following three conditions; 
\begin{enumerate}
\item[(i)] $M$ is reflexive, i.e. the natural mapping  $M \to M^{**}$  is bijective. \medskip
\item[(ii)] $\Ext _R^i(M, R) =0$  for all $i >0$. \medskip
\item[(iii)] $\Ext _R^i(M^*, R) =0$  for all $i >0$. \medskip
\end{enumerate}
See \cite{AB} for the detail {on} totally reflexive modules.   
The following theorem says that  condition $\rm{(ii)}$ is sufficient for totally reflexivity if the ring $R$ is generically Gorenstein.

\begin{theorem}[Corollary \ref{3}]\label{3'}
Assume that the ring  $R$ is a generically Gorenstein ring. 
Let  $M$ be a finitely generated $R$-module. 
Then the following conditions are equivalent: 
\begin{enumerate}
\item
$M$ is a totally reflexive $R$-module. \medskip
\item
$\Ext_R^i (M, R) =0$  for all $i >0$. \medskip
\item
$M$ is an infinite syzygy, i.e. there is an exact sequence of infinite length of the form 
$\xymatrix{0 \ar[r] & M \ar[r]&P_0 \ar[r]&P_1 \ar[r]& P_2 \ar[r]& \cdots,}$
where  each $P_i$  is a finitely generated projective $R$-module. 
\end{enumerate}
\end{theorem}

\begin{proof}
The implications  $(1) \Rightarrow (2)$  and  $(1) \Rightarrow (3)$  are well-known and easily proved. 
{For example see \cite[Proposition (3.8)]{AB}.}

\vspace{6pt}
$(2) \Rightarrow (1)$: 
Take projective resolutions for $M$ and $M^*$ respectively as 
$$
\xymatrix{
\cdots \ar[r]^{f_2} & F_1 \ar[r]^{f_1} & F_0 \ar[r]^{f_0} & M \ar[r] & 0  \ \text{and} \ 
\cdots \ar[r]^(0.6){g_2} & G_1 \ar[r]^{g_1} & G_0 \ar[r]^{g_0} & M^* \ar[r] & 0.  
 \\
}
$$
Then we consider the complex  
$$
X = \left[
\xymatrix@C=36pt{
\cdots \ar[r]^{g_2} & G_1 \ar[r]^{g_1} & G_0 \ar[r]^{f_0^*g_0} & F_0^* \ar[r]^{f_1^*} & F_1^* \ar[r]^{f_2^*} & \cdots 
 \\
}\right], 
$$
which belongs to $\K$,   and acyclic by the condition (2). 
Hence by Theorem \ref{last main theorem} the dual $X^*$  is acyclic as well. 
Since 
$$
X^* = \left[
\xymatrix@C=36pt{
\cdots \ar[r]^{f_2} & F_1 \ar[r]^{f_1} & F_0 \ar[r] & G_0^* \ar[r]^{g_1^*} & G_1^* \ar[r]^{g_2^*} & \cdots 
 \\
}\right], 
$$
is an exact sequence, it follows that  $M \cong M^{**}$  and  $\Ext _R ^i (M^*, R) =0$ for $i >0$. 

\vspace{6pt}
$(3) \Rightarrow (2)$: 
As in (3) we assume that  there is an exact sequence 
$$
\xymatrix{
0 \ar[r] & M \ar[r] & P_0  \ar[r] & P_1 \ar[r] & P_2 \ar[r] & \cdots.    \\ 
}
$$
Then combining this with the projective resolution 
$\xymatrix{ \cdots \ar[r] & F_1 \ar[r]& F_0 \ar[r]& M \ar[r]& 0}$  of $M$, we have an acyclic complex in $\K$ 
$$
Y =\left[ 
\xymatrix{
\cdots \ar[r] & F_1 \ar[r] & F_0 \ar[r] & P_0 \ar[r] & P_1 \ar[r] & P_2 \ar[r] & \cdots  
}\right].
$$
It follows from Theorem \ref{last main theorem} that  $Y^*$ is acyclic again. 
In particular,  the sequence  $\xymatrix{F_0^* \ar[r]& F_1^* \ar[r]& F_2^* \ar[r]&\cdots}$  is exact, and hence 
$\Ext _R^i(M, R) =0$  for $i>0$.  
\end{proof}

Recall that a finitely generated module $M$ has the G-dimension at most $n$, denoted by $\mathrm{G\text{-}dim} _R M \leq n$,  if there is an exact sequence 
of the form  
$$\xymatrix{0 \ar[r]& G_n \ar[r]& G_{n-1} \ar[r]& \cdots \ar[r]&  G_1 \ar[r]& G_0 \ar[r]&  M \ar[r]& 0,}$$
 where all $G_i \ (0 \leq i \leq n)$ are totally reflexive.

\begin{theorem}
Under the assumption that $R$ is a generically Gorenstein ring, we have the equality 
$$
\mathrm{G\text{-}dim} _R M = \sup \{ i \in \Z \ | \ \Ext _R^i (M, R) \not= 0 \ \}, 
$$ 
for a finitely generated $R$-module $M$. 
\end{theorem}

\begin{proof}
Setting   $n = \sup \{ i \in \Z \ | \ \Ext _R^i (M, R) \not= 0 \ \}$, 
we have only to consider the case $n < +\infty$. 
In this case it is easy to see that  $n \leq \mathrm{G\text{-}dim} _R M$. 
(This is just because $\Ext _R ^i (M, R) =0$  for any $i > \mathrm{G\text{-}dim} _R M$.)     
Now we take part of projective resolution of $M$ and get the $n$th syzygy module $\Omega_R^n(M)$, that is,  
$$
\xymatrix{
0 \ar[r] & \Omega _R ^n(M) \ar[r] & P_{n-1} \ar[r] & P_{n-2} \ar[r] & \cdots \ar[r] & P_1 \ar[r] & P_0 \ar[r] & M \ar[r] & 0   
}
$$
is an exact sequence of $R$-modules with  each  $P_i$ being projective. 
Then, since it holds that  $\Ext _R ^i (\Omega _R^n(M), R) = 0$  for $i > 0$, $\Omega _R^n(M)$  is totally reflexive by Theorem \ref{3'}.
Therefore  we have  $\mathrm{G\text{-}dim} _R M \leq n$. 
\end{proof}

Jorgensen and \c{S}ega \cite{JS} gave {an example of a module over a non-Gorenstein Artinian ring for which the implication $(2) \Rightarrow (1)$ in Theoreom \ref{3'} does not hold}, hence the generic Gorensteinness assumption in the theorem is indispensable.

The following is a commutative version of Tachikawa conjecture, which we obtain as a corollary to Theorem \ref{last main theorem}. 
It should be noted that it has been proved by Avramov, Buchweitz and \c{S}ega \cite{ABS}. 
 
\begin{theorem}[Corollary \ref{4}]\label{tachikawa}
Let  $R$  be a Cohen-Macaulay ring with canonical module $\omega$. 
Furthermore assume that  $R$ is a generically Gorenstein ring. 
If  $\Ext _R ^i (\omega ,  R) =0$  for all $i >0$, then $R$ is Gorenstein. 
\end{theorem}

\begin{proof}
Assume  $\Ext _R ^i (\omega ,  R) =0$  for all $i >0$. 
It is enough to show that  $\omega$ is a projective $R$-module. 
We see from Theorem \ref{3'} that  $\omega$  is a totally reflexive $R$-module, and hence there is an exact sequence 
of the form $\xymatrix{0 \ar[r]& \omega  \ar[r]& P_0 \ar[r]& P_1 \ar[r]& P_2 \ar[r]& \cdots}$, where  each $P_i$  is a finitely generated projective $R$-module. 
Setting  $M = \Ker (P_1 \to P_2)$, we note that  $M$  is an maximal Cohen-Macaulay module, since there is an exact sequence  $\xymatrix{0 \ar[r]& M \ar[r]& P_1 \ar[r]& P_2 \ar[r]& \cdots}$. 

Therefore we have $\Ext _R^1 (M, \omega)=0$ by the local duality theorem. 
It however means that a short exact sequence $\xymatrix{0 \ar[r]& \omega \ar[r]& P_0 \ar[r]& M \ar[r]& 0}$ splits, 
and $\omega$ is a direct summand of the projective module $P_0$, and it is projective. 
\end{proof}

{
\begin{remark}
In Theorem \ref{tachikawa} the canonical module has finite Gorenstein dimension, and existence of a finitely generated module of finite injective dimension and finite Gorenstein dimension implies that the ring is Gorenstein. 
This was announced in \cite[4.1]{F}; it was proved in \cite[3.2]{Ho} and can also be deduced from earlier results \cite[3.3.5]{C2} and \cite[8.3]{C3}. 
\end{remark}
}

\begin{theorem}[Corollary \ref{5}]
Assume that the ring  $R$ is a generically Gorenstein ring. 
Let  $X$  be a complex of finitely generated projective modules. 
{
\begin{enumerate}
\item 
If $H(X)$ is bounded above, i.e. $X  \in D ^- (R)$, then there is an isomorphism  $X^* \cong \RHom _R(X, R)$ in the derived category $D(R)$
\item 
If  $H(X)$ and $H(X^*)$ are bounded above, i.e. $X, X^* \in D ^- (R)$, 
then we have the isomorphism in the derived category: 
$$
X \cong \RHom _R (\RHom _R (X, R), R). 
$$
\end{enumerate}}
\end{theorem}

\begin{proof}
(1) There is an integer $a$ such that $H^i(X) = 0$ for all $i \geqq a$. 
Let  
$$\xymatrix{\cdots \ar[r]^{d_W^{a-3}} & W^{a-2} \ar[r]^{d_W^{a-2}} & W^{a-1} \ar[r]^{\alpha} & Z^a(X) \ar[r] & 0}$$
be an $R$-projective resolution of the $a$th cocycle $Z^a(X) = \ker (\xymatrix@C=16pt{X^a \ar[r]^{d_X^a} & X^{a+1}})$, where each $W^i$ are finitely generated projective. 
Then we consider the complex 
$$
Y = [\xymatrix{\cdots \ar[r]^{d_W^{a-3}}& W^{a-2} \ar[r]^{d_W^{a-2}}& W^{a-1} \ar[r]^{\beta} & X^a \ar[r]^{d_X^a} & X^{a+1} \ar[r]^{d_X^{a+1}} & \cdots} ], 
 $$
where $\beta$ is the composition of $\alpha$ with the natural injection $Z^a(X) \hookrightarrow X^a$. 
Then we see that  $Y \in \K$  and $H(Y)=0$ by the construction. 

Then the identity mappings on $X^i$ for $i \geqq a$ can be extended to a chain map $\varphi : X \to Y$ as follows: 
$$
\xymatrix{\cdots \ar[r] & X^{a-3} \ar[r]^{d_X^{a-3}} \ar[d]^{\varphi ^{a-3}} & X^{a-2} \ar[r]^{d_X^{a-2}} \ar[d]^{\varphi ^{a-2}}& X^{a-1} \ar[r]^{d_X^{a-1}} \ar[d]^{\varphi ^{a-1}} & X^a \ar[r]^{d_X^a} \ar[d]^{=}& X^{a+1} \ar[r]^{d_X^{a+1}} \ar[d]^{=} & \cdots \\
\cdots \ar[r] & W^{a-3} \ar[r]^{d_W^{a-3}}& W^{a-2} \ar[r]^{d_W^{a-2}}& W^{a-1} \ar[r]^{\beta} & X^a \ar[r]^{d_X^a} & X^{a+1} \ar[r]^{d_X^{a+1}} & \cdots}
$$
Now set $U = Cone (\varphi)$.   
Then $U$ is isomorphic in  $\K$ to the complex 
$$
\xymatrix{\cdots \ar[r] & X^{a-2} \oplus W^{a-3} \ar[r] & X^{a-1} \oplus W^{a-2} \ar[r] & W^{a-1} \ar[r] & 0}.
$$
Since there is a triangle $\xymatrix{X \ar[r] & Y \ar[r] & U \ar[r] & X[1]}$ in $\K$ and since $H(Y)=0$, there is an isomorphism  $U \cong X[1]$ or $U[-1] \cong X$ in $D(R)$. 
Since $U[-1]$ is a complex bounded above that consists of projective modules, 
we have an isomorphism in $D(R)$; 
$$\RHom _R (X, R) \cong \RHom_R(U[-1], R) \cong (U[-1])^*.$$ 

On the other hand, taking the $R$-dual,  we also have a triangle in $\K$; 
$$\xymatrix{U^* \ar[r] & Y^* \ar[r] & X^* \ar[r] & (U[-1])^*}. $$
Since $H(Y^*)=0$ by Theorem \ref{last main theorem}, we have $X^* \cong (U[-1])^*$ in $D(R)$. 
Combining this with the above we consequently have $X^* \cong (U[-1])^* \cong \RHom _R (X, R)$ as desired. 

(2) is clear from (1), since the isomorphism $X^{**} \cong X$ holds  in $D(R)$ as well as  in $\K$. 
\end{proof}

As a miscellaneous result we obtain the following. 

\begin{theorem}[Corollary \ref{6}]
Assume that the ring  $R$ is a generically Gorenstein ring. 
Let  $X$  be a complex of finitely generated projective modules. 

If all the cohomology modules $H^i(X) \ (i \in \Z)$ have dimension at most $\ell$  as $R$-modules,  
then so {do} the modules  $H^i (X^*) \ (i \in \Z)$.  
\end{theorem}

\begin{proof}
The assumption exactly means that  $X_{\p}$  is acyclic for a prime ideal  $\p$  with $\dim R/\p > \ell$. 
Note that each localization  $R_{\p}$  is generically Gorenstein. 
Therefore  $(X^*)_{\p} = \Hom _{R_{\p}} (X_{\p},  R_{\p})$  is acyclic again for such $\p$  with  $\dim R/\p > \ell$,  by Theorem \ref{last main theorem}. 
\end{proof}

Now we introduce the dimension of {total cohomology module of} a complex $X$  as 
$$
\dim _R {H(X)} = \sup \{ \dim H^i(X) \ | \ i \in \Z \},  
$$  
which is the dimension of the big support of $H(X)$. 
(Note that we use the convention that $\dim _R M = -1 $ for  the trivial $R$-module $M=\{ 0 \}$.)    
Then the theorem above includes the following generalization of Theorem \ref{last main theorem}. 

\begin{corollary}
Let $R$ be a generically Gorenstein ring. 
Then, for a complex  $X \in \K$, we have the equality  
 $
 \dim _R {H(X)} = \dim _R {H(X^*)}.
 $
\end{corollary}

\vspace{12pt}

\end{document}